\documentclass[aap,preprint]{imsart}

\usepackage{mathtools}

\makeatletter
\DeclareRobustCommand\widecheck[1]{{\mathpalette\@widecheck{#1}}}
\def\@widecheck#1#2{%
    \setbox\z@\hbox{\m@th$#1#2$}%
    \setbox\tw@\hbox{\m@th$#1%
       \widehat{%
          \vrule\@width\z@\@height\ht\z@
          \vrule\@height\z@\@width\wd\z@}$}%
    \dp\tw@-\ht\z@
    \@tempdima\ht\z@ \advance\@tempdima2\ht\tw@ \divide\@tempdima\thr@@
    \setbox\tw@\hbox{%
       \raise\@tempdima\hbox{\scalebox{1}[-1]{\lower\@tempdima\box
\tw@}}}%
    {\ooalign{\box\tw@ \cr \box\z@}}}
\makeatother

\RequirePackage{amsthm,amsmath,amsfonts,amssymb}
\RequirePackage[authoryear]{natbib}
\RequirePackage[colorlinks,citecolor=blue,urlcolor=blue]{hyperref}
\RequirePackage{graphicx}

\usepackage{enumitem}
\usepackage{comment,bbm}

\startlocaldefs

\newcommand{\Mstar}{[m]}

\newcommand{\Qro}{\mathbb{Q}}
\newcommand{\Pro}{{\mathbb{P}}}
\newcommand{\Exp}{{\mathbb{E}}}

\newcommand{\bN}{\mathbb{N}}
\newcommand{\bR}{{\mathbb{R}}}

\newcommand{\cF}{\mathcal{F}}

\theoremstyle{plain}
\newtheorem{theorem}{Theorem}[section]
\newtheorem{lemma}[theorem]{Lemma}
\newtheorem{remark}{Remark}[section]

\theoremstyle{definition}
\newtheorem{definition}{Definition}[section]
\newtheorem{example}{Example}

\endlocaldefs

\begin{document}
\begin{frontmatter}

\title{Efficient Importance Sampling for Wrong Exit Probabilities over Combinatorially Many Rare Regions}

\runtitle{Efficient Simulation of Wrong Exit Probability}


\begin{aug}
\author[A]{\fnms{Yanglei} \snm{Song}\ead[label=e1]{yanglei.song@queensu.ca}}
\and
\author[B]{\fnms{Georgios} \snm{Fellouris}\ead[label=e2]{fellouri@illinois.edu}}

\runauthor{Y. Song and G. Fellouris}

\address[A]{Department of Mathematics and Statistics, 
Queen's University,
Kingston, ON, Canada \\
\printead{e1}}

\address[B]{Department of Statistics,
University of Illinois Urbana-Champaign,
Champaign, IL, USA\\
\printead{e2}}
\end{aug}


\begin{abstract}
We consider importance sampling for estimating the probability that a light-tailed $d$-dimensional random walk exits through one of many disjoint rare-event regions before reaching an anticipated target. This problem arises in sequential multiple hypothesis testing, where the number of such regions may grow combinatorially and in some cases exponentially with the dimension. While mixtures over all associated exponential tilts are asymptotically efficient, they become computationally infeasible even for moderate values of $d$. We develop a method for constructing asymptotically efficient mixtures with substantially fewer components by combining optimal tilts for a small number of regions with additional proposals that control variance across a large collection of regions. The approach is applied to the estimation of three probabilities that arise in sequential multiple testing, including a multidimensional extension of Siegmund's classical exit problem, and is supported by both theoretical analysis and numerical experiments.
\end{abstract}



\end{frontmatter}

\tableofcontents

\section{Introduction} 
Motivated by the problem of estimating error probabilities for sequential probability ratio tests (SPRT) \citep{alma993552393405158}, \citet{siegmund1976importance} studied the use of Monte Carlo simulation to evaluate the probability that a one-dimensional, light-tailed random walk with negative drift exits a bounded interval, containing zero in its interior, through the upper boundary. An importance sampling estimator was proposed, and it was shown that its relative error remains bounded as the probability of interest tends to zero; see Section \ref{subsec:literature_review} for further discussion.

When testing multiple pairs of simple hypotheses, a natural generalization of Wald's SPRT is to monitor each log-likelihood ratio statistic until all of them are simultaneously large in absolute value, and then to decide for each pair of hypotheses based on the sign of the corresponding test statistic \citep{de2012sequential,song2017asymptotically}. For this procedure, it is desirable to estimate the probability of at least one wrong decision. This consideration motivates an extension of Siegmund's problem to the multidimensional setting.

To be specific, let $\{X_n : n \geq 1\}$ be a sequence of independent and identically distributed (i.i.d.) $d$-dimensional random vectors defined on a probability space $(\Omega, \mathcal{F}, \mathbb{P})$. Consider the associated random walk
$S_n := \sum_{i=1}^n X_i$ for $n \geq 1$ with $S_0 = 0$, and define the stopping time
\begin{equation} \label{def:siemund_stopping}
T = \inf\left\{ n \geq 0 : \; S_{n,k} > b u \;\; \text{or} \;\; S_{n,k} < - b \ell, \;\; \text{for all } k \in [d] \right\},
\end{equation}
where $\ell, u > 0$, $b > 0$, $S_n = (S_{n,1}, \ldots, S_{n,d})^\top$, and $[d] := \{1,\ldots,d\}$. If every coordinate of $X_1$ has negative expectation, then with high probability all coordinates of $S_T$ will be negative. We refer to the problem of estimating the probability that at least one coordinate of the random walk at time $T$ is positive as the \textit{multidimensional Siegmund problem}.

In this work, we consider a general formulation that encompasses the multidimensional Siegmund problem, as well as the estimation of a broad class of rare-event probabilities arising in the sequential multiple testing literature under various information structures and error metrics; see, e.g., \cite{de2012sequential,de2012step,song2017asymptotically,song2019sequential,he2021asymptotically}. Although motivated by statistical applications, the proposed formulation is also relevant in queuing and insurance contexts; see, e.g., \cite{alma9917764973405158,glasserman2004monte,hult2006heavy}.

\subsection{Problem description}
For any Borel subset $W \subset \mathbb{R}^d$ and $b > 0$, define the hitting time of the dilated set $bW$ by the random walk $\{S_n\}$ as
\begin{align}
    \label{def:hitting_time}
    \tau_b(W) := \inf\{n \geq 0 : S_n \in bW\}, 
   \quad \text{where}  \quad  bW := \{b x : x \in W\}.
\end{align}
Fix disjoint open sets $W^{0}, W^{1}, \ldots, W^{J} \subset \mathbb{R}^d$ and set
\begin{align}\label{def:hitting_W_js} 
    \tau_b^j := \tau_b(W^j), \quad j=0,1,\ldots,J, 
    \qquad \tau_b^{*} := \min_{j \in [J]} \tau_b^j.
\end{align}

We are interested in estimating the probability that the random walk $\{S_n\}$ hits $bW^j$ for some $j \in [J]$ before it hits $bW^0$, i.e., $\Pro\left(\tau_{b}^{*} < \tau_b^{0}\right)$, under the following assumptions: (i) $W^1,\ldots,W^J$ are convex; (ii) the random walk is light-tailed; and (iii) the drift $\Exp[X_1]$ is directed away from $W^j$ for every $j \in [J]$. Note that $W^0$ need not be convex. In fact, it may be empty, in which case the quantity of interest reduces to the \emph{hitting probability} $\Pro(\tau_b^{*} < \infty)$, namely the probability that the random walk eventually hits $bW^j$ for some $j \in [J]$. In general, however, $W^0$ will be nonempty and, although this is not required for our theoretical results, in all applications we consider the drift $\mathbb{E}[X_1]$ points toward $W^0$. In view of this, we refer to $\Pro\left(\tau_{b}^{*} < \tau_b^{0}\right)$ as the  \emph{wrong exit probability}.

Since  $\Pro\left(\tau_{b}^{*} < \tau_b^{0}\right)$ decays exponentially fast as $b$ increases,  plain Monte Carlo estimation becomes impractical for large values of $b$. It is therefore natural to appeal to importance sampling (I.S.) techniques. However, as we will discuss, existing efficient I.S.~methods become computationally prohibitive when the number, $J$, of rare-event regions is very large (e.g., exponential in $d$), as it is typical in the motivating statistical applications.

Consider, for example, the multidimensional Siegmund problem. For each $A \subset [d]$, define
\begin{align}
   \label{def:W_A} 
   W^{A} = \{x \in \mathbb{R}^d : x_{k} > u \;\; \text{for}\;\; k \in A, \;\; x_{k'} < -\ell \;\; \text{for}\;\; k' \notin A\},
\end{align}
identify the empty set $\emptyset$ with $0$, and index $\{A \subset [d] : A \neq \emptyset\}$ by $[J]$ with $J := 2^d - 1$. Then the event that at least one coordinate of the random walk is positive at time $T$, as defined in \eqref{def:siemund_stopping}, coincides with exiting through some $bW^{A}$ for $A \neq \emptyset$ before hitting $bW^{\emptyset}$, that is,
\[
\{S_{T,k} > bu \text{ for some } k \in [d]\} \;\;=\;\; \{\tau_{b}^{*} < \tau_b^{\emptyset}\}.
\]
See Figure~\ref{fig:intersection_rule} for an illustration of the exit regions when $d=2$. In this context, $J = 2^d - 1$ grows exponentially with $d$ and is already enormous for moderate dimensions (e.g., $J > 10^{12}$ when $d=40$). Similar combinatorial complexity arises in events associated with other sequential multiple testing rules; see Sections~\ref{sec:gap_rule} and~\ref{sec:sum_inter}.

Thus, our main goal is the design of an efficient I.S.~approach for the estimation of $\Pro\left(\tau_{b}^{*} < \tau_b^{0}\right)$, \textit{which remains feasible even when $J$ is very large}. Before describing the main contributions of this work, we first review the relevant literature.

\begin{figure}[!t]
    \centering
    \includegraphics[width=0.4\textwidth]{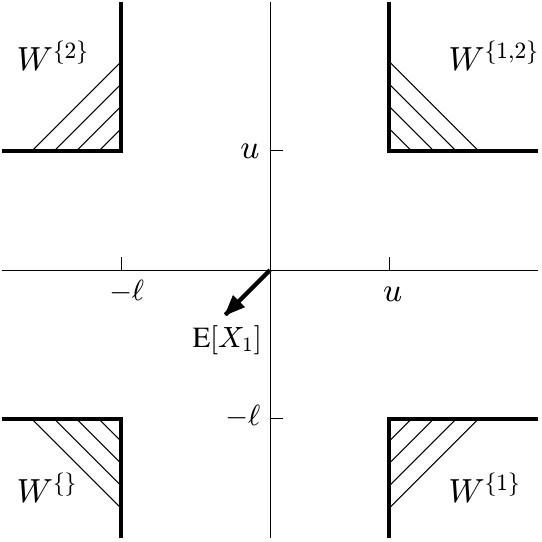}
    \caption{Exit regions of the multidimensional Siegmund problem with $d = 2$. For general $d$, the number of disjoint regions is $2^d$, which grows exponentially with $d$. }
    \label{fig:intersection_rule}
\end{figure}

\subsection{Literature review}\label{subsec:literature_review}
A wide range of variance reduction techniques have been developed for rare-event simulation; see~\cite{alma992931930105151,glasserman2004monte,asmussen2007stochastic} for comprehensive textbook treatments, and~\cite{59428b60bd3211dc88d5000ea68e967b,juneja2006rare,blanchet2012state} for survey articles. This work focuses on importance sampling (I.S.), a widely used method that simulates under alternative probability distributions and reweights samples using likelihood ratios to produce unbiased estimators; see Section~\ref{sec:preliminary} for a formal description.

Since our setting involves rare events associated with light-tailed random walks, we restrict our literature review to this class of problems. Importantly, I.S.~techniques differ substantially between light-tailed and heavy-tailed settings; for the latter, see~\cite{blanchet2008efficient,blanchet2008state,blanchet2010efficient}. Moreover, although our focus is on random walks, we should also mention several related works that are developed in the setting of more general Markov processes; see, e.g.,~\cite{ney1987markovI,ney1987markovII,sadowsky1990large,bucklew1990monte,lehtonen1992asymptotically,glasserman1995analysis,sadowsky1996monte,collamore2002,blanchet2007fluid}.

For light-tailed random walks, a powerful tool for designing I.S.~proposals is \emph{exponential tilting}. Specifically, let $\mu$ denote the distribution of $X_1$, and define the cumulant generating function $\Lambda : \mathbb{R}^d \to \mathbb{R} \cup \{\infty\}$ by $\Lambda(\theta) := \log \mathbb{E}[e^{\theta \cdot X_1}]$, where $\theta \cdot X_1$ denotes the standard inner product in $\mathbb{R}^d$. For $\theta \in \operatorname{dom}(\Lambda) := \{z \in \mathbb{R}^d : \Lambda(z) < \infty\}$, define the tilted measure $\mu_\theta$ on $\mathbb{R}^d$ via the Radon--Nikodym derivative 
$$
\frac{d\mu_\theta}{d\mu}(x) = \exp\left(\theta \cdot x - \Lambda(\theta)\right), \text{ for } x \in \bR^d,
$$
and let $\mathbb{P}_\theta$ denote the probability measure under which $X_n, n \geq 1$ are i.i.d.\ with law $\mu_\theta$.

As discussed earlier, in the one-dimensional setting ($d = 1$)  \citet{siegmund1976importance} considered the problem of estimating $\Pro(S_T \geq b u)$, where $T$ is the first exit time from $b[-\ell,u]$, i.e., as defined in~\eqref{def:siemund_stopping} with $d=1$, where $\ell, u > 0$ are fixed and $b > 0$ is an index parameter. Assume that $\mathbb{E}[X_1] < 0$, $\Pro(X_1 > 0) > 0$, and that the interior of $\operatorname{dom}(\Lambda)$ contains zero. Then there exists a unique $z > 0$ such that $\Lambda(z) = 0$. If $X_n$, $n \geq 1$, are generated under $\Pro_{z}$, the estimator $\exp(-z S_T)\, \mathbbm{1}\{S_T > bu\}$ is unbiased, and as shown in~\citet{siegmund1976importance}, its relative error---defined as the ratio of its standard deviation to the probability of interest---is uniformly bounded over $b > 0$. Moreover, under mild conditions, $\Pro_{z}$ is the unique I.S.~proposal with this property in the family of exponential tiltings $\{\Pro_{\theta} : \theta \in \operatorname{dom}(\Lambda)\}$.

\citet{collamore1996,collamore2002} considered the estimation of the hitting probability 
$\mathbb{P}(\tau_b(W) < \infty)$ for a general convex open set $W$ and a general light-tailed $d$-dimensional random walk, where $\tau_b(W)$ is defined as in~\eqref{def:hitting_time}. Specifically, \citet{collamore1996} showed, under mild conditions, that this probability decays exponentially as $b \to \infty$, with exponent given by a rate function $I : \mathbb{R}^d \to [0,\infty]$ (see Theorem~\ref{thm:Collomore_prob} for a precise statement):
\begin{align}
    \label{def:intro_rate}
    \lim_{b \to \infty} b^{-1} \log \mathbb{P}(\tau_b(W) < \infty) = -\inf_{x \in W} I(x).
\end{align}
Furthermore, \citet{collamore2002} established the existence of a unique point $v \in \partial W$ such that $I(v) = \inf_{x \in W} I(x)$. For this $v$, there exists a unique $\beta^{W} \in \operatorname{dom}(\Lambda)$ satisfying $\Lambda(\beta^{W}) = 0$ and $\nabla \Lambda(\beta^{W}) = v$, and the tilted measure $\mathbb{P}_{\beta^{W}}$ yields an \textit{asymptotically efficient} (A.E.) I.S.~proposal for estimating the hitting probability. Specifically, under $\mathbb{P}_{\beta^{W}}$, the estimator $\exp(-\beta^{W} \cdot S_{\tau_b(W)}) \, \mathbbm{1}\{\tau_b(W) < \infty\}$ is unbiased, and its second moment decays at the optimal exponential rate; see Definition~\ref{def:LR_theta} for the definition of asymptotic efficiency. Although this criterion is weaker than bounded relative error, it often provides sufficient accuracy in practice \citep{asmussen2007stochastic}.

As seen in the previous two examples, large deviations theory \citep{dembo2009large,borovkov2012large} often guides the construction of efficient I.S.~proposals. When the target set $W$ is more complex than a single convex set, as shown in \citet{sadowsky1990large,collamore2002}, a mixture of measures $\{\Pro_{\theta} : \theta \in \Theta\}$ can be effective, where $\Theta \subset \operatorname{dom}(\Lambda)$ is finite. Specifically, \citet{collamore2002} also considered the hitting probability $\Pro(\tau_b^* < \infty)$, where $\tau_b^*$ is defined in~\eqref{def:hitting_W_js}, i.e., the probability that the random walk ever hits $bW^j$ for some $j \in [J]$. Assuming that the sets $W^1, \ldots, W^J$ are disjoint, convex, and open subsets of $\mathbb{R}^d$, and that the drift of the random walk, $\mathbb{E}[X_1]$, is directed away from all of them, it was shown that a mixture of $\{\Pro_{\beta^j} : j \in [J]\}$ is asymptotically efficient for estimating $\Pro(\tau_b^* < \infty)$, where $\beta^j$ is the exponential tilt corresponding to the set $W^j$, as discussed above. However, when $J$ is very large, this mixture becomes computationally infeasible. Thus, even the efficient simulation of the \emph{hitting} probability $\Pro(\tau_b^* < \infty)$ remains an open problem in such regimes.

In a conference paper \citep{song2016logarithmically}, we studied the simulation of two wrong exit probabilities, which are special cases of the setups considered in Sections~\ref{sec:multi_Siegmund} and~\ref{sec:gap_rule}. However, that work requires 
the coordinates of 
$X_1$ to be log-likelihood ratios, independent, and “symmetric”; see Remark~\ref{remark:song_intersection} and Example~\ref{example:gap_ind_hom}. Moreover, it does not apply to probabilities such as the one in Section~\ref{sec:sum_inter}. 

Finally, state-dependent I.S.~schemes have also been developed for light-tailed systems; see \cite{DupuisWang2007,DupuisSezerWang2007}. However, their applicability to the large $J$ setting remains unclear.

\subsection{Contributions and organization}
We now outline our main contributions.

\smallskip
\noindent\textbf{(i).} We characterize the exponential decay rate as $b \to \infty$ of the wrong exit probability $\Pro\left(\tau_{b}^{*} < \tau_b^{0} \right)$. As this  is bounded above by the hitting probability $\Pro\left(\tau_{b}^{*} < \infty\right)$, whose exponential decay rate is established in~\cite{collamore1996}, our contribution lies in proving a matching lower bound, which we obtain through a pathwise large deviations analysis.


\smallskip
\noindent\textbf{(ii).} We show that the mixture of the $J$ tilted measures $\mathbb{P}_{\beta^1}, \ldots, \mathbb{P}_{\beta^J}$, which is asymptotically efficient (A.E.) for the hitting probability $\Pro\left(\tau_{b}^{*} < \infty\right)$, is also A.E.~for the wrong exit probability $\Pro\left(\tau_{b}^{*} < \tau_b^{0}\right)$ under mild conditions on the set $W^0$. However,  this mixture becomes computationally intractable when $J$ is large. Accordingly, we focus on the design of  feasible A.E.~mixtures that require substantially fewer than $J$ components.

\smallskip
\noindent\textbf{(iii).} In view of~\eqref{def:intro_rate}, 
$\mathcal{C}_* := \{j \in [J] : r_j = r_*\}$  is the collection of regions most likely to be hit in the event of a wrong exit, where $r_j := \inf_{x \in W^j} I(x)$ for $j \in [J]$ and $r_* := \min_{j \in [J]} r_j$. We show that for a mixture of tilted measures $\{\mathbb{P}_\theta : \theta \in \Theta\}$ to be A.E., the parameter set $\Theta \subset \operatorname{dom}(\Lambda)$ must include the tilts associated with \textit{all} regions in $\mathcal{C}_*$, i.e., $\{\beta^j : j \in \mathcal{C}_*\}$. Consequently, if $\mathcal{C}_*$ is too large (e.g., exponential in $d$), then no computationally feasible A.E.~mixture of this form can exist. A concrete example is provided in Subsection~\ref{subsec:negative}.

\smallskip
\noindent\textbf{(iv).} Since computing all $\{r_j\}$ is infeasible when $J$ is large, the set $\mathcal{C}_*$ and the value $r_*$ are generally unknown. However, guided by problem-specific intuition, we can often identify a candidate subset $\mathcal{C} \subset [J]$ that is likely to contain $\mathcal{C}_*$. Based on this insight, we propose a mixture of exponential tilts $\{\mathbb{P}_\theta : \theta \in \Theta\}$, where $\Theta \subset \operatorname{dom}(\Lambda)$ contains the set $\{\beta^j : j \in \mathcal{C}\}$ \textit{along with additional components designed to control the variance contributed by regions $W^j$ for $j \in [J] \setminus \mathcal{C}$}. The key idea is that, instead of assigning a separate tilt to each such region---which would be computationally infeasible---we identify a small set of tilts, not necessarily in $\{\beta^j : j \in [J] \setminus \mathcal{C}\}$, such that each tilt simultaneously accounts for many regions. Most importantly, we provide a sufficient condition for the asymptotic efficiency of this mixture, which in turn guides the selection of the additional components. See Subsection \ref{subsec:strategy}.

\smallskip
\noindent\textbf{(v).} We illustrate the proposed strategy through the estimation of three wrong exit probabilities motivated by problems in sequential multiple testing; see Sections~\ref{sec:multi_Siegmund}--\ref{sec:sum_inter}. In each case, the number of regions $J$ grows exponentially or combinatorially with the dimension $d$, yet we construct mixtures whose number of components grows only polynomially in $d$ with a substantially lower order. These mixtures are computationally feasible, and we provide numerical results that demonstrate their practical performance.

\smallskip
\noindent\textbf{Organization.} Section \ref{sec:preliminary} reviews the basics of importance sampling and relevant results from \cite{collamore1996,collamore2002}. In Section \ref{sec:prob_wrong_exit}, we develop a general strategy for constructing A.E.\ proposals for wrong exit probabilities and present necessary conditions for asymptotic efficiency. Sections \ref{sec:multi_Siegmund}--\ref{sec:sum_inter} present three applications of this strategy. Section \ref{sec:conclusion} concludes the paper. Proofs and technical details are deferred to the appendices.

\smallskip
\noindent \textbf{Notations.} 
Let $\bN := \{1, 2, \ldots\}$ denote the set of positive integers. 
For two vectors $x,y \in \bR^d$, denote by  $x \cdot y$ the inner product. For $x \in \bR^d$ and $r >0$, define $B(x,r) = \{y \in \bR^d: \|y-x\| < r\}$, where $\|\cdot\|$ denotes the Euclidean norm. For a subset  $A \subset \bR^d$ and a scalar $b \in \bR$, define $b A =\{bx : x \in A\}$ and  $\operatorname{cone}(A)  :=  \{\lambda x: \lambda > 0, x \in A\}$. 

For $A \subset \bR^d$, denote by $\operatorname{int}(A)$ the interior, $\overline{A}$ the closure, and $\partial A = \overline{A} \setminus \operatorname{int}(A)$ the boundary. For a convex $V \subset \bR^d$, denote by $\operatorname{ri}(V)$ its relative interior. 
For a function $f: \bR^d \to (-\infty, +\infty]$ and $a \in \bR$, denote by 
$\operatorname{dom}(f) = \{x \in \bR^d: f(x) < \infty\}$ its effective domain, 
by $\mathcal{L}_a f = \{x \in \bR^d: f(x) \leq a\}$ the level set of $f$, and 
by $\nabla f(x)$ the gradient at $x \in \bR^d$.

For brevity, we use A.E.~to denote ``asymptotically efficient'' and I.S.~to denote ``importance sampling''.


 \section{Importance Sampling  for  Hitting Probabilities}\label{sec:preliminary}
In this section, we introduce notation and review key concepts related to importance sampling. We also discuss the simulation of hitting probabilities for convex sets, following the foundational work of \cite{collamore1996, collamore_thesis, collamore2002}. Our approach to simulating wrong exit probabilities builds upon these developments.

\subsection{Basics of importance sampling} 
Let $\cF_n, n \geq 0$ be a filtration of the probability space  $(\Omega,\mathcal{F},\Pro)$. 
Let $\Qro$ be a probability measure on $(\Omega, \cF)$ such that $\Pro$ is absolutely continuous with respect to $\Qro$ on $\cF_n$ for each $n \geq 0$. We denote by $\Exp_{\Qro}$ the expectation under $\Qro$. Let $T$ be an $\{\cF_n\}$-stopping time and denote by $\cF_T$ the associated $\sigma$-algebra. Define
$$
 {Z}_{\Qro}:= \sum_{n=0}^{\infty} \frac{d\Pro}{d\Qro}(\cF_n) \mathbbm{1}\{T = n\},
$$
where $\frac{d\Pro}{d\Qro}(\cF_n)$ is the Radon–Nikodym derivative of $\Pro$ against $\Qro$ on $\cF_n$ for $n \geq 0$ and $\mathbbm{1}\{\cdot\}$ denotes the indicator function of an event. For each $b > 0$, let $A_b \in \cF_T$, that is, $A_b \cap \{T = n\} \in \cF_n$ for $n \geq 0$. By a change of measure, we have
\begin{align*}
    \Pro\left(A_b \cap \{T < \infty \}\right) = \Exp_{\Qro}\left[ {Z}_{\Qro} \,  \mathbbm{1}\{A_b \} \right].
\end{align*}
That is, under $\Qro$, the random variable ${Z}_{\Qro} \, \mathbbm{1}\{A_b\}$ is an unbiased estimator of $\Pro\left(A_b \cap \{T < \infty\}\right)$. We refer to $\Qro$ as a proposal measure. By Jensen's inequality, 
\begin{align*}
    \Exp_{\Qro}\left[ \left({Z}_{\Qro}\right)^2 \mathbbm{1}\{A_b\} \right] \geq \left(\Pro\left(A_b \cap \{T < \infty \}\right)\right)^2,
\end{align*}
which motivates the following definition.

\begin{definition} \label{def:AE}
Assume $\Pro\left(A_b \cap \{T < \infty \}\right) > 0$ for $b > 0$.
We say that a proposal measure $\mathbb{Q}$ is \emph{asymptotically efficient (A.E.)} for estimating $\Pro(A_b \cap \{T < \infty\})$ if
\begin{align*}
    \log \Exp_{\mathbb{Q}}\left[\left({Z}_{\Qro} \right)^2  \mathbbm{1}\left\{A_b\right\}\right] \;\sim\; 2 \log \Pro\left(A_b \cap \{T < \infty\}\right),
\end{align*}
where $\sim$ denotes that the ratio of the two sides converges to $1$ as $b \to \infty$.
\end{definition}
Thus, if $\mathbb{Q}$ is an asymptotically efficient proposal, then as $b \to \infty$, the variance under $\mathbb{Q}$ of the unbiased estimator
${Z}_{\Qro}  \mathbbm{1}\left\{A_b\right\}$  is as small as possible on the logarithmic scale. For further discussion, see, e.g., \cite{asmussen2007stochastic, glasserman1997}.

\subsection{Exponential tilting}
In this work we focus on the setting where $X_n$, $n \geq 1$, are i.i.d.~$d$-dimensional random vectors with common law $\mu$, and define the associated random walk by $S_n := \sum_{i=1}^{n} X_i$ for $n \geq 0$, with $S_0 = 0$. Let $\Lambda$ denote the cumulant generating function of $X_1$, and let $\Lambda^*$ denote its Legendre--Fenchel conjugate, defined respectively by
\begin{align*}
&\Lambda(\theta) := \log \Exp\left[ \exp\left( \theta \cdot X_1 \right) \right] = \log \int_{\mathbb{R}^d} \exp( \theta \cdot x ) \, d\mu(x), \quad \text{for } \theta \in \mathbb{R}^d, \\
&\Lambda^*(x) := \sup_{\theta \in \mathbb{R}^d} \left\{ \theta \cdot x - \Lambda(\theta) \right\}, \quad \text{for } x \in \mathbb{R}^d.
\end{align*}

We consider a special family of importance sampling proposals based on exponential tilting. 
For each $n \geq 1$  let $\cF_n$ be the $\sigma$-algebra generated by $X_1, \ldots, X_n$, and  $\cF_0$  the trivial $\sigma$-algebra. 
For each $\theta \in \operatorname{dom}(\Lambda)$  let $\Pro_{\theta}$ be a probability measure on $(\Omega, \mathcal{F})$ such that
\begin{equation}\label{def:exp_tilting}
\log \frac{d\Pro_{\theta}}{d\Pro}(\mathcal{F}_{n}) = \theta \cdot S_n - n \Lambda(\theta), \quad \text{for } n \in \mathbb{N},
\end{equation}
and denote by $\Exp_{\theta}$ the expectation under $\Pro_{\theta}$. Under $\Pro_{\theta}$, the sequence $\{X_n : n \geq 1\}$ is i.i.d. with common distribution $\mu_{\theta}$, whose Radon--Nikodym derivative relative to $\mu$ is given by
\begin{equation}\label{def:exp_tilting_common_dist}
\frac{d\mu_{\theta}}{d\mu}(x) = \exp\left( \theta \cdot x - \Lambda(\theta) \right), \quad x \in \mathbb{R}^d.
\end{equation}
For any $\theta \in \operatorname{dom}(\Lambda)$, the measures $\Pro$ and $\Pro_{\theta}$ are mutually absolutely continuous on $\mathcal{F}_n$ for each $n \geq 0$, and $\mathbb{L}_{\theta}^{-1}(T) \mathbbm{1}\{A_b\}$ is an unbiased estimator of  $\Pro\left(A_b \cap \{T < \infty\} \right)$ under $\Pro_{\theta}$, where we define
\begin{align}\label{def:LR_theta}
   \mathbb{L}_{\theta}(T) := \exp\left( \theta \cdot S_{T} - T \Lambda(\theta) \right) \, \mathbbm{1}\{T < \infty\}.
\end{align}

\subsection{Hitting probabilities}
For a subset $W \subset \bR^d$ and $b > 0$, recall from \eqref{def:hitting_time}  that  $\tau_b(W)$ denotes  the hitting time of the set $bW$ by the random walk $\{S_n : n \geq 0\}$. 
Assuming that $W \subset \mathbb{R}^d$ is  convex and open, and that  the random walk $\{S_n : n \geq 0\}$ is light-tailed and directed away from $W$, 
 \cite{collamore1996} characterizes  the asymptotic exponential decay rate of the hitting probability $\Pro\left( \tau_b(W) < \infty \right)$ as $b \to \infty$. To be specific, define $I : \mathbb{R}^d \to [0, \infty]$ as the support function of the set $\mathcal{L}_0 \Lambda$, i.e., for $x \in \bR^d$,
\begin{align}\label{def:rate_function}
I(x) :=\sup\left\{\theta \cdot x:\; \theta \in \mathcal{L}_0 \Lambda \right\}  = \sup\left\{\theta \cdot x:\; \log \Exp[\exp(\theta \cdot X_1)] \leq 0 \right\}.
\end{align}
Further, for each subset $W \subset \bR^d$, define the minimal rate as follows:
\begin{align}\label{def:r_W}
    r_W := \inf_{x \in W} I(x).
\end{align}

\begin{theorem}[From \cite{collamore1996}]\label{thm:Collomore_prob}
Let $W \subset \mathbb{R}^d$ be a convex open set. 
 Assume that (1) there exists  
 $\delta > 0$ such that $W$ does not intersect $\operatorname{cone}(B(\Exp[X_1], \delta))$, 
 and (2)  $\textbf{0}$ belongs to $\operatorname{int}(\operatorname{dom}(\Lambda))$ and  
$W$ intersects the relative interior of $cone(dom(\Lambda^*))$. Then
\begin{align*}
\lim_{b \to \infty}    \frac{1}{b} \log \Pro(\tau_b(W) < \infty)  =- r_W.
\end{align*}
\end{theorem} 


Condition (1) is essentially geometric, as it requires that the drift of the random walk, $\Exp[X_1]$, does not point toward $W$. This implies that $\mathbf{0} \notin W$ and that the random walk is unlikely to hit the set $bW$ when $b$ is large. Condition (2) further guarantees both the positivity and the exponential decay of the hitting probability $\Pro(\tau_b(W) < \infty)$ as $b \to \infty$. Indeed, if $\Lambda$ is finite in a neighborhood of the origin $\mathbf{0}$, then by Cram\'er's theorem \cite[Corollary~6.1.6]{dembo2009large} the sequence $\{S_n/n : n \geq 1\}$ satisfies a large deviations principle with convex, lower semicontinuous rate function $\Lambda^*$. Consequently, the domain of $\Lambda^*$ may be interpreted as the set of points that $S_n/n$ can approach for large $n$, and the cone of rays through $\operatorname{dom}(\Lambda^*)$ as the set of points that the random walk $\{S_n : n \geq 0\}$ can potentially reach. Thus, the assumption that $W$ intersects the relative interior of this set ensures that the random walk can hit $bW$ for sufficiently large $b > 0$ with nonzero probability, but that the probability of such an event decays exponentially as $b$ increases.

Finally, we note that condition (2) depends on both the random walk and the set $W$. In Appendix~\ref{app:discuss_collomore_prob} (see Remark~\ref{aux:remark_stronger}), we show that it is satisfied for \emph{any} convex open set $W$ provided the random walk satisfies the following condition:
\begin{align}
\label{assumption:stronger}
\operatorname{dom}(\Lambda) \text{ is open \;  and \;  } \{\theta \in \bR^d : \Lambda(\theta) \leq 0\} \text{ is bounded.}
\end{align}
In Appendix \ref{app:discuss_collomore_prob}, we also provide a \text{self-contained}, elementary proof of a weaker version of Theorem \ref{thm:Collomore_prob}, which requires assumption  \eqref{assumption:stronger} instead of (ii)
 (see Theorem \ref{app:thrm_weaker_prob}). 


\subsection{Importance sampling for hitting probabilities}
We next discuss the results in \citet{collamore2002} regarding the estimation via importance sampling of the hitting probability $\Pro(\tau_b(W) < \infty)$ for large values of $b > 0$,  when $W \subset \mathbb{R}^d$ is a convex  open set.
To present them, for each $\theta \in \bR^d$, we define the following function:
\begin{align} \label{def:rate_theta}
    \mathcal{V}_\theta(x)  :=  \sup\left\{\alpha \cdot x: \;\alpha \in \bR^d, \;\;\Lambda(\theta) + \Lambda(\alpha-\theta) \leq 0  \right\},\; \text{ for } \;x \in \bR^d,
\end{align}
where the supremum of an empty set is $-\infty$. For any $\theta \in \mathbb{R}^d$ and $W \subset \mathbb{R}^d$, we define
\begin{align}\label{def:V_W_theta}
    v(W;\theta) := \inf_{x \in \overline{W}} \mathcal{V}_\theta(x). 
\end{align}
Moreover,  we recall the definition of $r_W$ in \eqref{def:r_W}.

\begin{theorem}[From \cite{collamore2002}]\label{thrm:Collamore_simulation}
Let $W \subset \mathbb{R}^d$ be a convex open set. 
 Assume  that (1) there exists  
 $\delta > 0$ such that $W$ does not intersect $\operatorname{cone}(B(\Exp[X_1], \delta))$, and (2)
 $\operatorname{dom}(\Lambda)$ is open  and 
  $W$ intersects   the interior of $\operatorname{cone}(\operatorname{dom}(\Lambda^*))$. Then the following results hold.
  
\begin{enumerate}[label=(\roman*)]
    \item There exists a unique $\beta^{W} \in \bR^d$ such that
    $\Lambda(\beta^{W}) = 0$, 
    $$
    W \subset \{x \in \bR^{d}: \; \beta^{W} \cdot x \geq r_{W}\}, \quad \text{ and } \quad 
    \mathcal{L}_{r_W} I  \subset \{x \in \bR^{d}: \; \beta^{W} \cdot x \leq r_{W} \}.
    $$
    \item $\beta^{W}$ is the unique element in $\operatorname{dom}(\Lambda)$ such that
$$
\lim_{b \to \infty} \frac{1}{b} \log \Exp_{\beta^{W}}\left[ 
\mathbb{L}_{\beta^{W}}^{-2}(\tau_b(W))
\right] = -2 r_W= 2  \lim_{b \to \infty} \frac{1}{b} \log \Pro(\tau_b(W) < \infty).
$$
Thus, 
 $\Pro_{\beta^W}$ is the unique asymptotically efficient proposal for estimating $\Pro\left( \tau_b(W) < \infty \right)$ among all exponential tilting measures $\{\Pro_{\theta} : \theta \in \operatorname{dom}(\Lambda)\}$.
\item For $\theta \in \operatorname{dom}(\Lambda)$, if  $v(W;\theta) > 0$, then
\begin{align*}
    \limsup_{b \to \infty} \frac{1}{b} \log \Exp_{\theta}\left[ 
\mathbb{L}_{\theta}^{-2}(\tau_b(W))
\right] \leq   - v(W;\theta).
\end{align*}
\end{enumerate}
\end{theorem}

Condition (1) in Theorem~\ref{thrm:Collamore_simulation} is identical to that in Theorem~\ref{thm:Collomore_prob}, whereas condition (2) is slightly stronger, which holds for any $W$ provided the random walk satisfies \eqref{assumption:stronger}.



Part (i) is established in \cite{collamore2002}, with further details provided in Appendix~\ref{app:collomore_simulation}. The main idea of the proof is as follows. First, one shows that no point $x \in W$ attains the value $r_W = \inf_{x \in W} I(x)$, which implies that $W$ and the level set $\mathcal{L}_{r_W} I$ are disjoint convex sets. By the separation theorem, these sets can be separated by a hyperplane. Lemma~3.2 of \cite{collamore1996} then identifies the normal vector and offset of this separating hyperplane using tools from convex analysis.

Given part (i), part (ii) (without the uniqueness result) follows immediately. Specifically,  
\begin{align*}
\mathbb{E}_{\beta^{W}}\left[ 
\mathbb{L}_{\beta^{W}}^{-2}(\tau_b(W))
\right] 
&= \mathbb{E}_{\beta^{W}}\left[ 
\exp\left(-2\,\beta^{W} \cdot S_{\tau_b(W)}\right) \, \mathbbm{1}\{\tau_b(W) < \infty\}
\right] 
\leq \exp(-2 b r_W),
\end{align*}
where  the first equality holds because $\Lambda(\beta^W) =  0$, and  the  second by the fact that on the event $\{\tau_b(W) < \infty\}$ we have $S_{\tau_b(W)} \in bW$, and hence $\beta^{W} \cdot S_{\tau_b(W)} \geq b \, r_W$ by part~(i).  

Part (iii), which is critical for our purposes in this work, is independent of parts (i) and (ii). It establishes an asymptotic upper bound on the variance of the importance sampling estimator induced by an arbitrary exponential tilting. 
In Appendix \ref{app:collomore_simulation}, 
we provide an \textit{elementary} proof for this result (see Theorem \ref{app:thrm_weaker_simulation}) under the stronger assumption \eqref{assumption:stronger}.
In fact,  in the proof of Theorem \ref{app:thrm_weaker_simulation}, under the sole condition  $v(W;\theta) > 0$,   we establish the following weaker version of part (iii):
$$
    \limsup_{b \to \infty} \frac{1}{b} \log \Exp_{\theta}\left[ 
\mathbb{L}_{\theta}^{-2}(\tau_b(W))
\right] \leq   -    \sup_{\alpha \in \bR^d:  \Lambda(\theta) + \Lambda(\alpha - \theta) \leq 0} \; \inf_{x \in \overline{W}} (\alpha \cdot x),
$$
which however suffices  for all  applications in Section \ref{sec:multi_Siegmund}, \ref{sec:gap_rule} and \ref{sec:sum_inter}.

\begin{remark}
\cite{collamore2002} investigates the simulation of hitting probabilities for Markov additive processes, which include random walks as a special case. Moreover, a corresponding lower bound to part (iii) of Theorem \ref{thrm:Collamore_simulation} is also established.
\end{remark}

\section{Importance Sampling for Wrong Exit Probabilities}
\label{sec:prob_wrong_exit}
In this section we fix open, convex sets $W^1, \ldots, W^J \subset \mathbb{R}^d$, together with an open (but not necessarily convex) set $W^0 \subset \mathbb{R}^d$, and focus on the estimation of the \textit{wrong exit probability} $\Pro(\tau_b^* < \tau_b^0)$. Recall from~\eqref{def:hitting_W_js} that $\tau_b^j$ denotes the first hitting time of the set $bW^j$ for $j = 0,1,\ldots,J$, and that $\tau_b^*$ denotes the first hitting time of the union $bW^1 \cup \cdots \cup bW^J$.

We assume that  the random walk $\{S_n\}$ is directed away from the sets $W^j$, $j \in [J]$, in the sense that
\begin{align}\tag{C1}\label{assumptions_directions_of_mean}
\begin{split}
\text{there exists } \;\; \delta>0 \;\; \text{such that} \;\;  \operatorname{cone}(B(\Exp[X_1], \delta)) \cap W^j = \emptyset \;\; \text{for every } j \in [J].
\end{split}
\end{align}
We do not assume that the random walk is directed toward $W^0$, although this holds in all applications we consider. 
Furthermore,  we assume the following:
\begin{align} \tag{C2} \label{assumption_on_X}
    \begin{split}
        &\operatorname{dom}(\Lambda) \text{ is an open set}; \\
         &\operatorname{int}\left(\operatorname{cone}\left( \operatorname{dom}(\Lambda^*)\right)\right) \cap W^j \neq \emptyset, \quad \text{for  every }  j \in [J].
    \end{split}
\end{align}

Conditions \eqref{assumptions_directions_of_mean} and \eqref{assumption_on_X} impose the assumptions  in Theorem~\ref{thrm:Collamore_simulation} to each $W^j$, $j \in [J]$. In particular, part (i) of Theorem \ref{thrm:Collamore_simulation} motivates the following definition. Here, and in what follows, we recall the definition of  $I(\cdot)$  in \eqref{def:rate_function}, and we set
\begin{align}\label{def:rj_rstar}
r_j := r_{W^j} = \inf_{x \in W^j} I(x), \quad j \in [J].
\end{align}

\begin{definition}
\label{def:optimal_beta_j}
Suppose conditions \eqref{assumptions_directions_of_mean} and \eqref{assumption_on_X} hold. Then, for each $j \in [J]$ there exists a unique vector $\beta^j \in \mathbb{R}^d$ such that $\Lambda(\beta^j) =0$ and
\begin{align*}
    W^j \subset \{v \in \mathbb{R}^d : \beta^j \cdot v \geq r_j\}, \quad \text{and} \quad \mathcal{L}_{r_j} I \subset \{v \in \mathbb{R}^d : \beta^j \cdot v \leq r_j\}.
\end{align*}
We refer to $\beta^j$ as the \textit{optimal (exponential) tilting} associated with $W^j$.
\end{definition}

\begin{remark}
The computation of $\beta^j$ and $r_j$ is discussed in Subsection \ref{subsec:computation}.
\end{remark}

Finally, we assume that
 $W^0, W^1, \ldots, W^J$
 lie within disjoint cones and are separated from the origin. Specifically, we  assume 
\begin{align} \tag{C3} \label{assumption_on_regions}
    \begin{split}
        &\operatorname{cone}(W^j) \cap \operatorname{cone}(W^{j'}) = \emptyset \quad \text{for all } 0 \leq j \neq j' \leq J;\\
        &\text{there exists } \;\; \delta>0 \;\; \text{such that} \;\; B(\mathbf{0}, \delta) \cap W^j = \emptyset \quad \text{for each } j \in \{0\} \cup [J].
    \end{split}
\end{align}

\subsection{Characterization of exponential decay rates of wrong exit probabilities} 
For each $b > 0$, we have 
\[
\Pro(\tau_b^* < \tau_b^0) = \Pro(\tau_b^* < \infty) - \Pro\left( \tau_b^0 < \tau_b^*  < \infty \right).
\]
Thus, by applying Theorem \ref{thm:Collomore_prob} to each hitting probability $\Pro(\tau_b^{j} < \infty)$ for $j \in [J]$, we immediately obtain an upper bound on the exponential decay rate of the wrong exit probability as $b \to \infty$. In the next theorem, we establish a \emph{matching} lower bound, thereby characterizing its exponential decay rate. 


\begin{theorem}\label{thm:prob_rate}
Assume that conditions \eqref{assumptions_directions_of_mean}, \eqref{assumption_on_X}, and \eqref{assumption_on_regions} hold. Then, 
\begin{align*}
    &\lim_{b \to \infty} \frac{1}{b} \log \Pro(\tau_b^* < \tau_b^0,\; \tau_b^* = \tau_b^j) = -r_j, \quad \text{ for each } j \in [J];\\
    &\lim_{b \to \infty} \frac{1}{b} \log \Pro(\tau_b^* < \tau_b^0) = -r_*, \;\; \text{ where } r_* := \min_{j \in [J]} r_j.
\end{align*}
Moreover, we have $0 < r_j < \infty$ for each $j \in [J]$.
\end{theorem} 
\begin{proof}
See Appendix \ref{app:lower_bound_exit}. 
\end{proof}

When $W^0 = \emptyset$, so that $\tau_b^0 = \infty$, Theorem \ref{thm:prob_rate} follows directly from \cite{collamore1996} (see Theorem \ref{thm:Collomore_prob} above). In the general case, to derive a lower bound on the probability that the random walk hits $bW^j$ before entering any other region,
$\Pro(\tau_b^* < \tau_b^0,\; \tau_b^* = \tau_b^j)$,
for some $j \in [J]$,  we apply \textit{pathwise} large deviations results from \cite{borovkov2012large}.

\subsection{Importance sampling for the probability of wrong exit from a specific region}\label{sec:IS_general0}



Fix $j \in [J]$.   For any $\theta \in \operatorname{dom}(\Lambda)$, define
\begin{align*}
\widehat{\mathbb{W}}_{b}^{j}(\theta) :=  
\mathbb{L}_{\theta}^{-1}(\tau_b^*)  \, \mathbbm{1}\left\{
\tau_b^* < \tau_b^0,\; \tau_b^* = \tau_b^j \right\} ,
\end{align*}
where we recall the definitions of the tilted measure $\Pro_{\theta}$ in~\eqref{def:exp_tilting} and the likelihood ratio $\mathbb{L}_{\theta}(\cdot)$ in~\eqref{def:LR_theta}. By a change of measure, under $\Pro_\theta$ it is an unbiased estimator of the probability that the random walk hits $bW^{j}$ before any $bW^i$, $i \in \{0,1,\ldots,J\} \setminus \{j\}$, that is,
$$
\Pro\left(\tau_b^* < \tau_b^0,\; \tau_b^* = \tau_b^j\right)= \Exp_{\theta} \left[   \widehat{\mathbb{W}}_{b}^{j}(\theta) \right].
$$
 The following lemma, an immediate consequence of Theorems~\ref{thm:prob_rate}and~\ref{thrm:Collamore_simulation}, shows that 
 the exponential tilting $\beta^j$ associated with $W^j$ in Definition~\ref{def:optimal_beta_j} is the choice of $\theta$ that leads to an  A.E.~proposal for estimating the above probability. Recall the definition of   $r_j$ in~\eqref{def:rj_rstar}.

\begin{lemma}\label{lemma:upper_bound_var_single_j}
 Assume conditions \eqref{assumptions_directions_of_mean},  \eqref{assumption_on_X} and \eqref{assumption_on_regions} hold. For each $j \in [J]$, we have
$$
\lim_{b \to \infty} \frac{1}{b} \log \Exp_{\beta^{j}}\left[ 
\left(\widehat{\mathbb{W}}_{b}^{j}(\beta^j)\right)^2
\right]
 = -2r_j.
$$
Further,  for each  $\theta \in \operatorname{dom}(\Lambda)$, if $v(W^j;\theta) > 0$, then
\begin{align}\label{def:vj_def}
\limsup_{b \to \infty} \frac{1}{b} \log \Exp_{\theta}\left[ 
\left(\widehat{\mathbb{W}}_{b}^{j}(\theta)\right)^2
\right]
\leq 
v(W^j;\theta) := v_j(\theta),
\end{align}
where we recall the definition of $v(\cdot;\cdot)$ in \eqref{def:V_W_theta}.
\end{lemma}

\begin{proof}
    See Appendix \ref{app:proof_var_single_j}.
\end{proof}


\subsection{Importance sampling for the   wrong exit probability}\label{sec:IS_general}
For the estimation of the wrong exit probability $\Pro\left(\tau_b^* < \tau_b^0\right)$, we use the mixture measure
\begin{align*}
    \Pro_{\Theta} := \frac{1}{|\Theta|} \sum_{\theta \in \Theta} \Pro_{\theta}
\end{align*}
as an importance sampling proposal, where $\Theta \subset \operatorname{dom}(\Lambda)$ is a \emph{finite} subset and $|\Theta|$ denotes its cardinality. We denote by $\Exp_{\Theta}$ the expectation with respect to $\Pro_{\Theta}$.

\begin{remark}
For a fixed $\theta \in \operatorname{dom}(\Lambda)$, the distribution of $X_1,X_2,\ldots$ 
under $\Pro_{\theta}$ is the product measure of the distribution $\mu_{\theta}$ in \eqref{def:exp_tilting_common_dist}. On the other, for $\Theta \subset \operatorname{dom}(\Lambda)$ with $|\Theta| \geq 2$, $X_1, X_2,\ldots$ are not independent  under $\Pro_{\Theta}$.
\end{remark}

Recall $\mathbb{L}_{\theta}(\cdot)$ from~\eqref{def:LR_theta}.
For any choice of finite $\Theta \subset \operatorname{dom}(\Lambda)$ and $b>0$,  we define
\begin{align}
    \label{def:unbiased_estimator_Theta}
\widehat{\mathbb{W}}_{b}(\Theta)  :=  
   \left( \frac{1}{|\Theta|} \sum_{\theta \in \Theta} 
\mathbb{L}_{\theta}(\tau_b^*) \right)^{-1} \mathbbm{1}\{\tau_b^* < \tau_b^0\},
\end{align}
which is  an  unbiased estimator  for the wrong exit probability under $\Pro_\Theta$, that is,
$$
\Pro\left(\tau_b^* < \tau_b^0\right)= \Exp_{\Theta} \left[   \widehat{\mathbb{W}}_{b}(\Theta) \right].
$$



\begin{remark}\label{remark:simulation strategy}
 To simulate under the mixture measure $\Pro_{\Theta}$, we proceed as follows: first, sample $\theta \in \Theta$ uniformly at random; then generate a sequence of i.i.d.\ random variables $\{X_n : n \geq 1\}$ with distribution $\mu_{\theta}$ defined in \eqref{def:exp_tilting_common_dist}, and compute a realization of the estimator $\widehat{\mathbb{W}}_b(\Theta)$ defined in \eqref{def:unbiased_estimator_Theta}. Repeating this procedure multiple times and averaging the resulting values yields an estimate of the wrong exit probability.
\end{remark}  

We aim to choose the set $\Theta$ such that (i) the mixture $\Pro_\Theta$ is asymptotically efficient, and (ii) both constructing $\Theta$ and evaluating $\widehat{\mathbb{W}}_{b}(\Theta)$ are computationally tractable, so that the resulting I.S. estimator is feasible. 

\subsection{Sufficient conditions for asymptotic efficiency}
In view of Definition  \ref{def:AE} and Theorem \ref{thm:prob_rate}, $\Pro_{\Theta}$ is an A.E. proposal for estimating $\Pro(\tau_b^* < \tau_b^0)$ if
 \begin{align} \label{control_var}
        \limsup_{b\to \infty} \frac{1}{b} \log  \Exp_{\Theta}\left[ \left(\widehat{\mathbb{W}}_{b}(\Theta) \right)^2 \right] \leq  -2  r_*.
\end{align}

The next lemma shows that this is always the case if  $\Theta$ contains the optimal tilting corresponding to each region. In what follows, for any
$\mathcal{C} \subset [J]$ we set  
\begin{equation}
    \label{def:r_C}
    r_{\mathcal{C}} := \min_{j \in \mathcal{C}} r_j.
\end{equation}

 \begin{lemma}\label{lemma:full_mixture_efficient}
 Assume that conditions \eqref{assumptions_directions_of_mean},  \eqref{assumption_on_X} and  \eqref{assumption_on_regions} hold.   Let $\mathcal{C} \subset [J]$
 and $\Theta \subset \operatorname{dom}(\Lambda)$ be a finite set. 
 If   $\{\beta^j : j \in \mathcal{C}\} \subset \Theta$,  then
\begin{align} \label{control_var1}
        \lim_{b\to \infty} \frac{1}{b} \log  \Exp_{\Theta}\left[ \left(\widehat{\mathbb{W}}_{b}(\Theta) \right)^2 \sum_{j \in \mathcal{C}} \mathbbm{1}\{\tau_b^* = \tau_b^j\}\right] = -2  r_{\mathcal{C}}.
\end{align}
Consequently, if $\{\beta^j: j \in [J]\} \subset \Theta$, then \eqref{control_var} holds, and as a result, $\Pro_{\Theta}$ is an asymptotically efficient proposal for estimating the wrong exit probability $\Pro(\tau_b^* < \tau_b^0)$.
 \end{lemma}
 \begin{proof}
     See Appendix \ref{app:proof_full_mixture}.
 \end{proof}

Lemma \ref{lemma:full_mixture_efficient} establishes the asymptotic efficiency of the mixture $\Pro_\Theta$ for $\Theta = \{\beta^j : j \in [J]\}$. However, when $J$ is extremely large (e.g., exponential in the dimension $d$), it is computationally prohibitive both to compute $\{\beta^j : j \in [J]\}$ and to evaluate $\widehat{\mathbb{W}}_{b}(\Theta)$ in~\eqref{def:unbiased_estimator_Theta}, since the latter involves a sum over $J$ likelihood ratios. Thus, in what follows, our goal is to construct a \emph{feasible} A.E.~mixture $\Pro_{\Theta}$, whose cardinality $|\Theta|$ is \emph{much smaller} than $J$. For this, note that for any $\mathcal{C} \subset [J]$ the limit in~\eqref{control_var1} holds provided that $\Theta$ contains the tilts corresponding to the sets indexed by $\mathcal{C}$, i.e., $\{\beta^j : j \in \mathcal{C}\} \subset \Theta$.
If this is the case, and if additionally
\begin{align} \label{control_var2}
 \limsup_{b\to \infty} \frac{1}{b} \log \Exp_{\Theta}\left[ \left(\widehat{\mathbb{W}}_{b}(\Theta)\right)^2 \sum_{j \notin \mathcal{C}} \mathbbm{1}\{\tau_b^* = \tau_b^j\}\right] \leq -2 r_{\mathcal{C}},
\end{align}
then by Lemma~\ref{lemma:full_mixture_efficient} and Theorem \ref{thm:prob_rate} it follows that $r_* = r_{\mathcal{C}}$   and that $\Pro_{\Theta}$ is an A.E.~proposal.

Condition~\eqref{control_var2} requires that (i) $\mathcal{C}$ is large enough so that $r_{\mathcal{C}} = r_*$, and (ii) the second moment of $\widehat{\mathbb{W}}_{b}(\Theta)$ can be controlled when $\{S_n\}$ exits through one of the regions $bW^j$ with $j \notin \mathcal{C}$.
The following theorem provides a sufficient condition for~\eqref{control_var2}, and hence for the asymptotic efficiency of $\Pro_\Theta$.
Recall $r_{\mathcal{C}}$ from~\eqref{def:r_C} and $v_j(\cdot)$ from~\eqref{def:vj_def}.

\begin{theorem}\label{thm:main_strategy}
Assume that conditions \eqref{assumptions_directions_of_mean}, \eqref{assumption_on_X} and \eqref{assumption_on_regions} hold. Let $\mathcal{C} \subset [J]$ and $\Theta  \subset \operatorname{dom}(\Lambda)$ be a finite set. 
Suppose that $\{\beta^j : j \in \mathcal{C}\} \subset \Theta$  and 
\begin{align}\label{def:key_condition}
\text{for each } \; j \in [J] \setminus \mathcal{C} \; \text{ there exists }  \; \gamma \in \Theta \; \text{ such that } \;\;
v_j(\gamma) 
\;\geq\; 2r_\mathcal{C}.
\end{align}
Then,  $r_* := r_{\mathcal{C}}$,
and  $\Pro_{\Theta}$ is an asymptotically efficient proposal for estimating the wrong exit probability $\Pro(\tau_b^* < \tau_b^0)$; that is,
\[
\lim_{b \to \infty} \frac{1}{b} \log \Exp_{\Theta}\left[ \left(\widehat{\mathbb{W}}_{b}(\Theta) \right)^2 \right] 
= -2 r_{*} = 2  \lim_{b \to \infty} \frac{1}{b} \log \Pro(\tau_b^* < \tau_b^0).
\]
\end{theorem}
\begin{proof}
    See Appendix \ref{app:proof_main_strategy}.
\end{proof}

\begin{remark} \label{remark:key_condition_special}
If $\Theta = \{\beta^j : j \in \mathcal{C}\}$, then condition~\eqref{def:key_condition} reduces to
\begin{align*}
\text{for each } j \in [J] \setminus \mathcal{C}  \; \text{there exists } i \in \mathcal{C} \; \text{such that} \;\;
v_j(\beta^i) \;\geq\; 2 r_\mathcal{C}.
\end{align*}
\end{remark}

The construction of $\Theta$ for applying Theorem~\ref{thm:main_strategy} is discussed in Subsection~\ref{subsec:strategy}.

\subsection{A necessary condition for  asymptotic efficiency}
In this subsection, we establish a necessary condition for $\Pro_{\Theta}$ to be A.E.. Let $\mathcal{C}_*$ denote the set of indices corresponding to the regions with the smallest exponential rates:
\begin{align}\label{def:C_star}
    \mathcal{C}_* := \left\{j \in [J] :\;\; r_j = r_* \right\}.
\end{align}
By Theorem \ref{thm:prob_rate}, $\mathcal{C}_*$ indexes the regions from which the random walk $\{S_n\}$ is most likely to exit on the event of a wrong exit, i.e., $\{\tau_b^* < \tau_b^0\}$. It is thus natural to expect that an efficient choice of   $\Theta$ should include 
the optimal tiltings (recall  Definition \ref{def:optimal_beta_j}) associated with these most likely exit regions, namely that $
    \left\{\beta^j : j \in \mathcal{C}_* \right\} \subset \Theta$. We next establish the \emph{necessity} of this inclusion, under the additional assumption stated below:
\begin{align}\label{assumption:negative_results}  \tag{C4}
    \begin{split}
    &  \overline{\operatorname{cone}(W^j)} \;\cap\; \overline{\operatorname{cone}(W^{j'})} = \{\mathbf{0}\}, \;\;\; \text{ for } \;0 \leq j \neq j' \leq J. \\
    & \text{ Either } \; (i). \; W^0 = \emptyset, \quad  \text{ or } \;\; (ii).\; W^0 \text{ is convex and } \operatorname{int}\left(\operatorname{cone}\left( \operatorname{dom}(\Lambda^*)\right)\right) \cap W^0 \neq \emptyset.
    \end{split}
\end{align}
The first condition in \eqref{assumption:negative_results} is slightly stronger than the first part of \eqref{assumption_on_regions}. The second 
allows $W^0$ to be either empty or to satisfy assumptions analogous to those imposed on $W^j$ for $j \in [J]$ in  condition \eqref{assumption_on_X}.


\begin{theorem}\label{thm:negative}
Assume that conditions \eqref{assumptions_directions_of_mean}, \eqref{assumption_on_X}, \eqref{assumption_on_regions} and \eqref{assumption:negative_results} hold.
Let $\Theta \subset \operatorname{dom}(\Lambda)$ be finite.
If $\beta^j \not \in \Theta$ for some $j \in \mathcal{C}_*$, then 
$\Pro_{\Theta}$ fails to be asymptotically efficient for estimating $\Pro(\tau_b^* < \tau_b^0)$, i.e., 
 \begin{align*}
     \liminf_{b\to \infty} \frac{1}{b} \log \Exp_{\Theta}\left[ \left(\widehat{\mathbb{W}}_{b}(\Theta) \right)^2  \right] > - 2r_*.
 \end{align*}
\end{theorem}
\begin{proof}
See Appendix \ref{app:neg_result}. 
\end{proof}

 Theorem \ref{thm:negative} can be used to derive negative results: if the size of $\mathcal{C}_*$ grows exponentially with $d$, then it becomes \textit{computationally infeasible} to construct any asymptotically efficient proposal of the form $\Pro_{\Theta}$ for some finite $\Theta \subset \operatorname{dom}(\Lambda)$, since $\Theta$ must include $\{\beta^j: j \in \mathcal{C}_*\}$.

\begin{remark}
As demonstrated by the counterexamples in \citet{glasserman1997}, the condition $\{\beta^j : j \in \mathcal{C}_*\} \subset \Theta$ does not, in general, suffice for $\Pro_{\Theta}$ to be A.E.
\end{remark}

\subsection{Strategy for constructing $\Theta$}\label{subsec:strategy}

Theorem~\ref{thm:main_strategy} provides a strategy for selecting $\Theta$ so that the mixture $\Pro_\Theta$ is A.E.:  
\begin{itemize}
    \item[\textbf{Step 1.}] Select a subset $\mathcal{C} \subset [J]$ and compute $r_j$ and $\beta^j$ for $j \in \mathcal{C}$. 
    
    \item[\textbf{Step 2.}] Choose a subset $\Theta \subset \operatorname{dom}(\Lambda)$ such that (i) $\{\beta^j : j \in \mathcal{C}\} \subset \Theta$, and (ii)   we can verify that condition~\eqref{def:key_condition} holds.
\end{itemize}

For Step 1, we note that $r_j$ and $\beta^j$ do not in general admit closed-form expressions, and their computation typically requires solving a convex optimization problem (see Subsection~\ref{subsec:computation}). Hence, Step 1 is computationally feasible provided that $|\mathcal{C}|$ is not too large.

  Moreover,  Theorem~\ref{thm:negative} suggests that $\mathcal{C}$ should contain those regions most likely to produce a wrong exit, i.e., the set $\mathcal{C}_*$ defined in~\eqref{def:C_star}. However, when $J$ is large, it is generally not feasible to compute all values in $\{r_j : j \in [J]\}$, and consequently $\mathcal{C}_*$ will not be known. In practice, therefore, $\mathcal{C}$ should be selected based on \emph{application-specific intuition} about where wrong exits are most likely.

For Step~2, the more components we include in $\Theta$, the easier it is for condition~\eqref{def:key_condition} to hold. In the extreme, condition~\eqref{def:key_condition} is always satisfied if $\Theta$ includes the exponential tilting that corresponds to region $W^j$, $\beta^j$, defined in Definition~\ref{def:optimal_beta_j},
 \textit{for every  $j \notin \mathcal{C}$}. However,  this is computationally infeasible when $J$ is large. To ensure that $|\Theta|$ remains much smaller than $J$, we aim to include in $\Theta$ components that \emph{simultaneously} accommodate multiple indices $j \in [J] \setminus \mathcal{C}$, rather than being tailored to a single region. That is,  the goal is that for each $\gamma \in \Theta \setminus \{\beta^j : j \in \mathcal{C}\}$  there exist multiple $j \notin \mathcal{C}$ such that $v_j(\gamma) \geq 2 r_\mathcal{C}$. We stress that such $\gamma$ need not belong to $\{\beta^j : j \notin \mathcal{C}\}$.

Given a finite $\Theta \subset \operatorname{dom}(\Lambda)$, there are two ways to verify condition~\eqref{def:key_condition}. The first is analytical: for each $j \in [J] \setminus \mathcal{C}$, we need to find some $\gamma \in \Theta$, construct an \emph{analytical} lower bound for $\nu_j(\gamma)$, and verify that it is at least $2 r_\mathcal{C}$. In Theorems~\ref{thm:sufficient}, \ref{thm:Gap_sufficient}, and~\ref{thm:sum_intersection_rule}, we adopt this approach to derive sufficient conditions for three different problems. The second way is computational and requires the  numerical evaluation  of $v_j(\gamma)$ 
for \textit{all} 
$j \in [J] \setminus \mathcal{C}$ and some $\gamma \in \Theta$. If $[J] \setminus \mathcal{C}$ is large, this
is generally infeasible. However, under \textit{special} structures, such as independence or homogeneity, it may suffice to perform this computation for only a small subset of $[J] \setminus \mathcal{C}$. This is the case in Example~\ref{ex:cor_normal_siegmund} and Subsection~\ref{subsec:iid_case_Siegmund}.



\subsection{Discussion of computation} \label{subsec:computation}

To implement the above strategy, we need to be able to   (i) compute $r_j$ and $\beta^j$ for a given $j \in [J]$, and (ii)   obtain a \textit{lower bound} on $v_j(\gamma)$ defined in \eqref{def:vj_def} for some  $\gamma \in \operatorname{dom}(\Lambda)$ and $j \in [J]$.
We now provide further details on each of these two tasks. Fix $j \in [J]$ and $\gamma \in \operatorname{dom}(\Lambda)$.

We first consider task (i). As established in Lemma~3.2 of \cite{collamore2002} (see also Lemma \ref{app_lemma:upper_bound_var_single_j} for a detailed explanation), we have
\begin{align}\label{def:r_j_alt}
    r_j = \sup_{\theta \in \mathcal{L}_0\Lambda} \left( \inf_{x \in \overline{W^j}} \theta \cdot x \right).
\end{align}
If $W^j$ is defined via constraints, and if the \textit{strong duality} holds, then the inner optimization is equal to its convex dual for each $\theta$, and the overall problem becomes a convex optimization and can typically be solved using efficient numerical algorithms.  For example, 
in applications discussed below, each region $W^j$ is of the form
\begin{align}
    \label{def:special_W_j}
W^j = \left\{x \in \mathbb{R}^d : (A^j)^\top x > b^j\right\},
\end{align}
for some integer $n^j \geq 1$, $A^j \in \bR^{d \times n^j}$ and $ b^j \in \bR^{n^j}$. In this case, the inner optimization is a linear programming for each $\theta \in \bR^d$, and when the \textit{strong duality} holds, we have
\begin{align*}
    r_j =  \max_{\theta \in \bR^d, \;y \in \bR^{n^j}} & b^j \cdot y,\;\; \text{ subject to }\;\;
    \Lambda(\theta) \leq 0,\;\; A^j y = \theta, \;\; y \geq 0,
\end{align*}
where the constraint $y \geq 0$ is interpreted coordinatewise. 

Next, we consider task (ii). By the definition of $v(\cdot;\cdot)$ in \eqref{def:V_W_theta}, and the max-min inequality,
\begin{align}\label{def:lower_bound_V_gamma}
    v_j(\gamma) \geq \sup_{\theta \in \bR^{d}: \Lambda(\gamma) + \Lambda(\theta -\gamma) \leq 0} \left( 
    \inf_{x \in \overline{W}^j} \theta \cdot x
    \right).
\end{align}
If $W^j$ is of the form \eqref{def:special_W_j} and if $\Lambda(\gamma) \leq 0$,  due to weak duality, we \emph{always} have 
\begin{align*}
    v_j(\gamma) \geq \max_{\theta \in \bR^{d}, y \in \bR^{n^j}}  b^j \cdot y,\;\; \text{ subject to }  \Lambda(\theta - \gamma) \leq 0,\;\; A^j y = \theta,\;\; y \geq 0.
\end{align*}
It is important to note that the convex optimization problem on the right-hand side is typically not solved exactly due to computational complexity. Instead, we construct a feasible solution, which provides a \textit{lower bound} for $v_j(\gamma)$.



\section{Multidimensional Siegmund Problem}\label{sec:multi_Siegmund}
In this section, we consider the multidimensional Siegmund problem.  Specifically, assume that $\Exp[X_{1,k}] < 0$ for each $k \in [d]$, where $X_1 = (X_{1,1},\ldots,X_{1,d})^\top$. Let $\ell,u > 0$ be fixed. For each $A \subset [d]$, recall the definition of $W^A$ in \eqref{def:W_A}.  
Further, denote by $\tau^{A}_{b}$ the first hitting time of $bW^{A}$ for $A \subset [J]$ and define
$$
\mathcal{D}_b  :=  A \;\;\text{ if and only if }\;\; \tau_b(W^{A}) \leq \tau_b(W^{\tilde{A}})\, \text{ for any } \tilde{A} \subset [d]. 
$$
We expect the random walk $\{S_n\}$ to first hit $b W^{\emptyset}$, and the goal is to use importance sampling to estimate the probability of it reaching other regions first.

First, we make connections with the setup in Section \ref{sec:prob_wrong_exit}. Denote by $\mathcal{P}_d$ all subsets of $[d]$  and by $\mathcal{P}_d^{-} = \mathcal{P}_d \setminus \{\emptyset\}$ the power set with the empty set removed. We identify $\emptyset$ with $0$ and $\mathcal{P}_d^{-}$ with $[J]$, where $J = 2^d-1$. For example, $W^{0}$ and $\tau_b^0$ is identified with $W^{\emptyset}$ and $\tau_b^\emptyset$, while
$\{\tau_b^A, W^{A}: A \in \mathcal{P}_d^{-}\}$ with $\{\tau_b^j, W^{j}:j\in [J]\}$.  Similarly, $\{r_{j}, v_j(\gamma), \beta^j: j \in [J]\}$  
in \eqref{def:rj_rstar}, \eqref{def:vj_def}, and Definition \ref{def:optimal_beta_j} are identified with $\{r_{A}, v_A(\gamma), \beta^{A}: A  \in \mathcal{P}_d^{-}\}$, and 
\begin{equation}
    \label{def:siegmund}
\Pro\left(\mathcal{D}_b \neq \emptyset\right) =
\Pro\left(\tau_b^* < \tau_b^{\emptyset}\right), \;\; \text{ with }
\tau_b^{*} = \min\{\tau_b^{A}: A \in \mathcal{P}_d^{-}\}.
\end{equation}
Note that conditions \eqref{assumptions_directions_of_mean} and  \eqref{assumption_on_regions} always hold in this setup. In addition, throughout this section, we assume that \eqref{assumption_on_X} holds.

\subsection{A general proposal}
In this example, we have $J = 2^d - 1$. Thus, a mixture of $J$ components, as suggested in Lemma \ref{lemma:full_mixture_efficient}, is not computationally feasible even for moderate values of $d$; for instance, $2^d > 10^{12}$ when $d = 40$. Instead, we apply the strategy proposed in Theorem \ref{thm:main_strategy} and show that, under suitable conditions, mixtures with significantly fewer components remain asymptotically efficient. For each $A \in \mathcal{P}_d^{-}$, we begin by characterizing the pair $(r_A, \beta^A)$ and lower bounding $v_A(\gamma)$, following the discussions in \eqref{def:r_j_alt} and \eqref{def:lower_bound_V_gamma}.

\begin{lemma}\label{lemma:opt_probs}
For each $A \in \mathcal{P}_d^{-}$, $r_A$ and $\beta^A$ are the optimal value and solution of the following optimization problem:
\begin{align*}
& \max_{\theta \in \bR^d} \;\;  u \sum_{k \in A} \theta_k -  \ell \sum_{k' \in A^c} \theta_{k'}, \quad \textup{ subject to }  \\
& \Lambda(\theta) \leq 0, \quad \theta_k \geq 0\; \text{ for } \;k \in A, \quad \theta_{k'} \leq 0\; \text{ for } \;k' \in A^c.
\end{align*}
Further, for each $\gamma \in \bR^d$ such that $\Lambda(\gamma) \leq 0$, $v_A(\gamma)$ is lower bounded by
\begin{align*}
& \max_{\theta \in \bR^d} \;\;  u \sum_{k \in A} \theta_k -  \ell \sum_{k' \in A^c} \theta_{k'}, \quad \textup{ subject to }  \\
& \Lambda(\theta - \gamma) \leq 0, \quad \theta_k \geq 0\; \text{ for } \;k \in A, \quad \theta_{k'} \leq 0\; \text{ for } \;k' \in A^c.
\end{align*}
\end{lemma}
\begin{proof}
    See Appendix \ref{app:multi_ruin}.
\end{proof}

\begin{remark}
The Karush–Kuhn–Tucker (KKT) conditions characterizing $\beta^A$, given in Appendix \ref{app:KKT}, provide a basis for analyzing $r_A$ and $\beta^A$. In addition,
a \textit{feasible} solution to the second optimization problem in Lemma \ref{lemma:opt_probs} provides a lower bound for $v_A(\gamma)$, 
as required by Theorem \ref{thm:main_strategy}.
\end{remark}

We next apply the strategy from Section \ref{subsec:strategy} to construct a subset $\Theta \subset \operatorname{dom} \Lambda$ such that the resulting measure $\Pro_{\Theta}$ is asymptotically efficient. We begin with guessing $\mathcal{C}_*$ in \eqref{def:C_star}: if the coordinates of $X_1$ are not highly correlated and $u$ is not far smaller than $\ell$, then we expect that on the event $\{\mathcal{D}_{b} \neq \emptyset\}$, exactly one coordinate is most likely to be positive, that is,  $|\mathcal{D}_{b}| = 1$. This suggests that $r_*$ in Theorem \ref{thm:prob_rate} is most likely attained by regions indexed by singleton subsets, i.e., $\{W^{\{k\}} : k \in [d]\}$. Motivated by this intuition, we set $\mathcal{C} = \{\{k\} : k \in [d]\}$ and include the corresponding optimal tiltings $\{\beta^{\{k\}} : k \in [d]\}$, as defined in Definition \ref{def:optimal_beta_j}, in a candidate set $\Theta$, for which the associated mixture measure $\Pro_{\Theta}$ serves as a potential proposal.

Furthermore, we include additional tilting parameters in $\Theta$ to ensure that condition \eqref{def:key_condition} is satisfied. 
To this end, we consider two collections of parameters in $\operatorname{dom}(\Lambda)$. 
First, for each $k \in [d]$, let $\gamma^k$ be the  solution of the following optimization problem:
\begin{align}
\label{def:gamma_k}
\max_{\theta \in \bR^d} \; u \theta_k,\quad \text{ subject to }\quad
\Lambda(\theta) \leq 0, \quad\theta_k \geq 0, \quad\theta_{k'} = 0 \text{ for } k' \neq k.
\end{align}
Clearly, if we denote by $z_k$ the unique positive number such that $\Exp[\exp(z_k X_{1,k})] = 1$, then 
\begin{align*}
\gamma^{k}_{j} = 
    z_k \;\; \text{ if } j = k, \quad \text{ and }\quad 
     \gamma^{k}_{j} = 0 \;\; \text{ otherwise},
\end{align*}
which implies that the optimal value of \eqref{def:gamma_k} is given by $u z_k$.

Second, for each $m \neq m'\in [d]$, denote by $s_{m,m'}$ and $\gamma^{m,m'
}$ the optimal value and maximizer of the following convex optimization problem:
\begin{align}
\label{def:gamma_k_kp}
\begin{split}
& \max_{\theta \in \bR^d} \;\;  u(\theta_m + \theta_{m'}), \quad \textup{ subject to }  \\
& \Lambda(\theta) \leq 0, \quad \theta_m \geq 0,\quad \theta_{m'} \geq 0,\quad 
\theta_{k} = 0\; \text{ for } \;k \not\in \{m,m'\},
\end{split}
\end{align}
which essentially only involves two variables $\theta_m, \theta_{m'}$, since other variables are fixed to zero.  

The motivation for considering the parameter collections in \eqref{def:gamma_k} and \eqref{def:gamma_k_kp} is twofold. First, their sizes ($d$ and $d(d-1)/2$, respectively) grow polynomially in $d$, in contrast to the exponential growth of all subsets of $[d]$. Second, they enable coverage of multiple subsets $A \subset [d]$ satisfying \eqref{def:key_condition}. Specifically, fix $k \neq k' \in [d]$. For any $A \subset [d]$ containing $\{k, k'\}$,  $\gamma^k + \gamma^{k,k'}$ is a \textit{feasible} solution to the second optimization problem in Lemma~\ref{lemma:opt_probs} with $\gamma = \gamma^k$, and thus
\begin{align}
    \label{aux:inter_gamma_k_lower}
v_A(\gamma^k) \geq u z_k + s_{k,k'}.
\end{align}
In other words, $u z_k + s_{k,k'}$ provides a lower bound for $v_A(\gamma^k)$ whenever $k \neq k' \in A$, and thus $\gamma^k$ can effectively cover all subsets $A \subset [d]$ containing $k$, whose total number is $2^{d-1}$. Similarly, if $A$ contains $\{k, k'\}$, then $\gamma^{k,k'} + \gamma^{k,k'}$ is a \textit{feasible} solution to the same problem with $\gamma = \gamma^{k,k'}$, and hence
\begin{align}
    \label{aux:inter_gamma_k_kp_lower}
v_A(\gamma^{k,k'}) \geq 2 s_{k,k'},
\end{align}
so that $\gamma^{k,k'}$ covers all $A \subset [d]$ with $\{k, k'\} \subset A$, of size $2^{d-2}$. 

Based on the preceding analysis, we derive sufficient conditions for constructing asymptotically efficient proposals.

\begin{theorem}\label{thm:sufficient}
Consider the following two conditions.
\begin{align*}
& (H1): \min_{k \neq k' \in [d] } \left(u z_k + s_{k,k'}\right) \geq \min_{k \in [d]} 2 r_{\{k\}}.\\
& (H2): \min_{k \neq k' \in [d] } 2s_{k,k'} \geq  \min_{k \in [d]} 2r_{\{k\}}.
\end{align*}
Under either condition, $r_* = \min_{k \in [d]} r_{\{k\}}$. Further, define $\Theta^{(0)}  :=  \{\beta^{\{k\}}: k \in [d]\}$ and
\begin{align*}
\Theta^{(1)}  :=  \Theta^{(0)} \cup \{\gamma^{k}: k \in [d]\},\quad
\Theta^{(2)}  :=  \Theta^{(0)} \cup \{\gamma^{k,k'}: k \neq k' \in [d]\}.
\end{align*}
If $(H1)$ holds, then $\Pro_{\Theta^{(1)}}$ is asymptotically efficient. If $(H2)$ holds, then $\Pro_{\Theta^{(2)}}$ is asymptotically efficient.
\end{theorem}
\begin{proof}
Let $A \subset [d]$ be arbitrary such that $|A| \geq 2$, and let $k \neq k' \in A$. First, under condition $(H1)$ and due to \eqref{aux:inter_gamma_k_lower}, we have $v_A(\gamma^k) \geq u z_k + s_{k,k'} \geq 2 \min_{k \in [d]} r_{\{k\}}$. Then the claims under $(H1)$ follow immediately from Theorem \ref{thm:main_strategy}. 

Similarly, the claims under $(H2)$ follow directly from \eqref{aux:inter_gamma_k_kp_lower} and Theorem \ref{thm:main_strategy}.
\end{proof}

Theorem
\ref{thm:sufficient}  suggests a procedure that solves $d+d(d-1)/2$ numerical optimization problems for $\{ r_{\{k\}}, \beta^{\{k\}}: k \in [d]\}$ and $\{s_{k,k'}, \gamma^{k,k'}: k \neq k' \in [d]\}$. If $(H1)$ (resp.~$(H2)$) holds, then a mixture with  $2d$ (resp.~$d(d+1)/2$) components is asymptotically efficient. 

\begin{remark}
Since $\gamma^k$ in \eqref{def:gamma_k} is a feasible solution to the problem in \eqref{def:gamma_k_kp}, we have $u z_k \leq s_{k,k'}$, and thus condition $(H2)$ is weaker than $(H1)$. Adding more elements from $\operatorname{dom}(\Lambda)$ to $\Theta^{(2)}$ allows for constructing asymptotically efficient proposals under conditions weaker than (H2), at the cost of increased computation. For instance, one could consider  $\gamma \in \operatorname{dom}(\Lambda)$ with \textit{three} nonzero coordinates.
\end{remark}

\begin{remark}\label{remark:general_sufficient_condition}
It is not always necessary to add components to $\Theta^{(0)}$ to achieve A.E. 
For example, by Remark \ref{remark:key_condition_special},
if for each $A \subset [d]$ with $|A| \geq 2$ there is a  $j \in [d]$ such that
\begin{align}\label{def:general_sufficient_condition}
    v_A(\beta^{\{j\}})\geq 2\min_{k \in [d]} r_{\{k\}},
\end{align}
then   $\Pro_{\Theta^{(0)}}$ is asymptotically efficient. However, this requires solving the second optimization in Lemma \ref{lemma:opt_probs} for all $A \subset [d]$ with $|A| \geq 2$. While computationally prohibitive in general, this becomes feasible under special structures, as we will see in the following examples.
\end{remark}

\begin{example}[correlated normals]\label{ex:cor_normal_siegmund}
Assume that $X_1$ has the multivariate normal distribution such that 
\begin{align}
    \label{def:cor_normal}
\Exp[X_1] = -\frac{1}{2}\mathbbm{1}_d, \quad
\text{Cov}(X_1) = \Sigma_{\rho}, \;\;\text{ where }\;\; \Sigma_{\rho,k,k'} = \begin{cases}
1, \text{ if } k = k' \in [d]\\
\rho,  \text{ if } k \neq k' \in [d]
\end{cases},
\end{align}
where $\mathbbm{1}_d$ is the $d$-dimensional vector with all entries being $1$ and $\rho \in [0,1)$.  The cumulant generating function of $X_1$ is given by: 
\begin{equation}
    \label{example:cor_normal_Lambda}
\Lambda(\theta) = -\frac{1}{2}\mathbbm{1}_d^{\top} \theta + \frac{1}{2}\theta^{\top} \Sigma_{\rho} \theta, \;\; \text{ for } \theta \in \bR^d.
\end{equation}
Furthermore, for each $\theta \in \mathbb{R}^d$, the distribution of $X_1$ under $\Pro_{\theta}$ is multivariate normal with
\begin{align*}
    \Exp_{\theta}[X_1] = - \frac{1}{2} \mathbbm{1}_d + \Sigma_\rho \theta, \quad
    \text{Cov}_{\theta}(X_1) = \Sigma_\rho.
\end{align*}


For $k \neq k' \in [d]$, the optimization problem in \eqref{def:gamma_k} simplifies to
$$
\max_{\theta_{k} \geq 0}\; u \theta_k, \quad\;\;\text{ such that } \;\;  -\frac{1}{2}\theta_k + \frac{1}{2} \theta_k^2 \leq 0,
$$
while the problem in \eqref{def:gamma_k_kp} becomes
$$
\max_{\theta_{k}, \theta_{k'} \geq 0}\; u (\theta_k + \theta_{k'}), \quad\;\;\text{ such that } \;\;  -\frac{1}{2}(\theta_k + \theta_{k'}) + \frac{1}{2} (\theta_k^2 + 2\rho \theta_k \theta_{k'} + \theta_{k'}^2) \leq 0.
$$
As a result, their optimal values are, respectively,
$$
u z_k = u, \qquad s_{k,k'} = 2u/(1+\rho).
$$
Due to symmetry, we have $r_{\{k\}}=r$ for every $k \in [d]$.  Thus, condition   $(H1)$ (resp.~$(H2)$) in Theorem \ref{thm:sufficient} holds  if
$r \leq u/(1+\rho) + u/2$ (resp. $r \leq 2u/(1+\rho)$).  Further, as discussed in Remark \ref{remark:general_sufficient_condition}, the mixture
$\Pro_{\Theta^{(0)}}$ is asymptotically efficient if condition \eqref{def:general_sufficient_condition} holds.  Due to the homogeneity of the current example, it is feasible to check this condition as it suffices to show that for any $2 \leq m \leq d$ and $A =\{1,\ldots,m\}$ we have
$v_A(\beta^{\{1\}}) \geq 2 r$, 
which only requires us to solve $d-1$ optimizations in Lemma \ref{lemma:opt_probs}.

In Figure \ref{fig:correlated_normals_rho} we plot $r$ as a function of $\rho \in \{0,0.01,\ldots,0.90\}$ for 
$\ell = u =1$ and $d=50$.
Further, the top solid curve  in this figure corresponds to the function $\rho \mapsto 2/(1+\rho)$, and the bottom one to $\rho \mapsto 1/(1+\rho) + 1/2$. Thus, condition   $(H1)$ (resp.~$(H2)$) is satisfied as long as the bottom (resp. top) solid curve is above $r$. In this context,  if $\rho \leq 0.45$, then condition (H1) holds and the proposal $\Pro_{\Theta^{(1)}}$ is asymptotically efficient. Similarly, if $\rho \leq 0.54$, then (H2) holds and $\Pro_{\Theta^{(2)}}$ is asymptotically efficient.

In Table \ref{tab:vary_u}, we fix $d = 50$ and $\ell = 1$, and we report the largest value of $\rho$ in the grid $\{0, 0.01, \ldots, 0.90\}$ for which conditions (H1),  (H2), and \eqref{def:general_sufficient_condition} hold  for different values of  $u$. We see that  smaller values of $u$ correspond to less favorable settings for our simulation purposes. For example, the mixture $\Pro_{\Theta^{(2)}}$, which contains at most $d(d+1)/2$ components, is provably asymptotically efficient whenever $\rho \leq 0.54$ (respectively, $\rho \leq 0.26$) 
for $u = 1$ (respectively, $u = 1/3$).


\end{example}

\begin{figure}[!tb]
    \centering
    \includegraphics[width=0.9\textwidth]{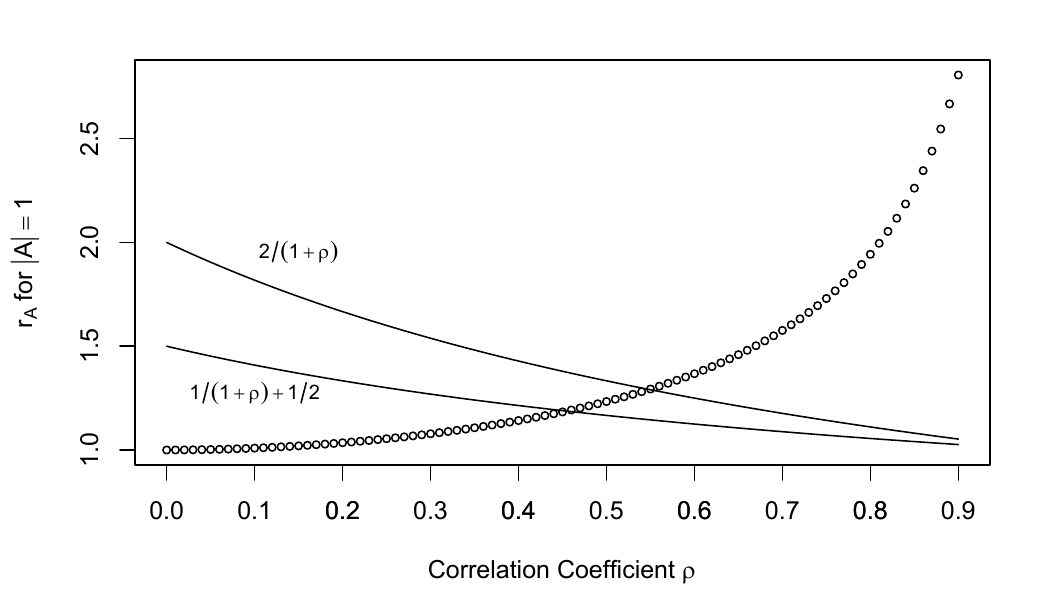}
    \caption{A plot of $r = r_{\{1\}}$ (circles) against the correlation coefficient $\rho$ for $d=50,\ell=1,u=1$. When the lower (resp.~upper) solid curve is above the circles,  $(H1)$ (resp.~$(H2)$) holds.}
    \label{fig:correlated_normals_rho}
\end{figure}

\begin{table}[!tb]
\begin{tabular}{c|c|c|c|c|c}
           & $u=3$ & $u=2$ & $u=1$ & $u=1/2$ & $u=1/3$ \\ \hline
(H1) holds & 0.61  & 0.57  & 0.45  & 0.25    & 0.09    \\ \hline
(H2) holds & 0.67  & 0.64  & 0.54  & 0.39    & 0.26   \\
\hline
\eqref{def:general_sufficient_condition} holds & 0.61  &   0.58 &  0.50 & 0.39    &  0.32
\end{tabular}
\caption{For different values of $u$, we list the maximal number within $\{0,0.01,\ldots,0.90\}$ such that $(H1)$ and $(H2)$ in Theorem \ref{thm:sufficient} and condition \eqref{def:general_sufficient_condition} hold.} 
\label{tab:vary_u}
\end{table}

\subsection{Independent coordinates}\label{subsec:independent_intersection_rule}
In this subsection, we assume that the coordinates of $X_1$ are \textit{independent}. In this case,  $\gamma^k$ in \eqref{def:gamma_k} has special interpretations as discussed below. Further, if, in addition, the coordinates of $X_1$ have the same distribution, then under mild conditions, a mixture of $d$ components is asymptotically efficient; see Subsection \ref{subsec:iid_case_Siegmund}.

\begin{definition}
Let $\nu$ and $\tilde{\nu}$ be probability measures on $\mathbb{R}$. We say $\tilde{\nu}$ is Siegmund exponential tilting of $\nu$ if there exists $z \neq 0$ such that $\frac{d\tilde{\nu}}{d\nu}(x) = \exp(zx)$ for all $x \in \mathbb{R}$. Note that when such a $z \in \bR$ exists, it must be unique.
\end{definition}

\begin{remark}
To justify the terminology, we consider $d=1$ and assume that the distribution of $X_1$ is given by $\nu$ and $\Exp[X_1] < 0$. Denote by $\tau = \inf\{n \geq 1: S_n > b \text{ or } S_n < -a\}$ for some $a,b > 0$. To simulate the probability that $\{S_{\tau} > b\}$, as shown in \cite{siegmund1976importance}, the unique asymptotically efficient exponential-tilting distribution is $\tilde{\nu}$.
\end{remark}

\begin{remark}
Siegmund exponential-tilting pairs appear naturally in hypothesis testing problems. 
Let $\nu_0$ and $\nu_1$ be two probability distributions on $\bR$, and denote by $\ell(x)  :=  \log(d\nu_1/d\nu_0(x))$ for $x \in \bR$ the log-likelihood ratio function. Denote by $\nu$ (resp.~$\tilde{\nu}$) the distribution of $\ell(Z)$ if $Z$ has the distribution $\nu_0$ (resp.~$\nu_1$). Then $\tilde{\nu}$ is Siegmund exponential-tilting distribution of $\nu$ (and vice versa).
\end{remark}

Denote by $\mu_k^{(0)}$ the distribution of $X_{1,k}$  and by $\Lambda^{k}(t)  :=  \log \Exp[\exp(t X_{1,k})]$ for $t \in \bR$ the cumulant generating function of $X_{1,k}$ under $\Pro$. Denote by $\mu_k^{(1)}$ the distribution of $X_{1,k}$  under $\Pro_{\gamma^{k}}$, where $\gamma^k$ is defined in \eqref{def:gamma_k} and recall that $z_k >0 $ is the unique positive number such that $\Lambda^{k}(z_k) = 0$. For $k \in [d]$, 
under $\Pro_{\gamma^k}$, the coordinates of $X_1$ remain independent and
$$
X_{1,k'} \sim \mu^{(0)}_{k'}\;\; \text{ if } k' \neq k,\quad \text{ and }\quad X_{1,k} \sim \mu^{(1)}_k.
$$
That is, under $\Pro_{\gamma^k}$, the distribution of the $k$-th coordinate is tilted from $\mu^{(0)}_k$ to its Siegmund exponential tilting $\mu^{(1)}_k$, while the other coordinates retain their original distributions.



Next, we show that under certain conditions,
$\beta^{\{k\}}$, as defined in Definition \ref{def:optimal_beta_j}, is in fact equal to $\gamma^{k}$ in \eqref{def:gamma_k}, for $k \in [d]$. 
Denote by
$$
\kappa_k^{(0)}  :=  -(\Lambda^k)'(0) > 0, \quad \kappa_k^{(1)}  :=   (\Lambda^k)'(z_k) > 0.
$$
Note that $-\kappa_k^{(0)}$ and $\kappa_k^{(1)}$ are  the expectation of $X_{1,k}$ under $\Pro$ and $\Pro_{\gamma^k}$, respectively.

\begin{lemma}\label{lemma:independent_one_change}
Assume the coordinates of $X_1$ are independent. 
For each $k \in [d]$, 
$$\beta^{\{k\}} = \gamma^{k} \quad \text{ if and only if } \quad (\ell/u) \kappa_k^{(1)}  \leq \kappa_{k'}^{(0)}, \quad \forall\;  k' \neq k.$$

Assume for some $k^* \in [d]$, $(\ell/u) \kappa_{k^*}^{(1)} \leq \kappa_{k}^{(0)}$ for each $k \neq k^* \in [d]$. If $z_{k^*} \leq z_{k} + z_{k'}/2$ for any  $k \neq k'$ and $z_{k} \geq z_{k'}$, then 
 $(H1)$ in Theorem \ref{thm:sufficient}  holds. If  $z_{k^*} \leq z_{k} + z_{k'}$ for any  $k \neq k'$,  then  $(H2)$ holds.
\end{lemma}
\begin{proof}
See Appendix \ref{app:independent_streams}.
\end{proof}

\begin{remark}
If $d = 1$, then $[d]$ is a singleton, so the conditions in Lemma~\ref{lemma:independent_one_change} automatically hold, recovering the result of \cite{siegmund1976importance} for one-dimensional random walks.
\end{remark}

\begin{remark}\label{remark:song_intersection}
If $\ell = u$, $\kappa_{k}^{(1)} \leq \kappa_{k}^{(0)}$ for $k \in [d]$ and $z_{k} = z_{k'}$ for $k, k' \in [d]$,  $\Theta^{(1)}$ in Theorem  \ref{thm:sufficient} reduces to $\{\gamma^{k}: k \in [d]\}$, which recovers Theorem 1 in \cite{song2016logarithmically} as a special case of the above lemma.
\end{remark}

\subsection{The i.i.d.~case} \label{subsec:iid_case_Siegmund}
Next, we consider the case where the coordinates of $X_1$ are i.i.d. We show that if either the dimension of the random walk is sufficiently large or the upper boundary $u$ is not much smaller than the lower boundary $\ell$, then $r_{*}$ is attained by $\{r_{\{k\}}: k \in [d]\}$, and a mixture with $d$ components is asymptotically efficient.

\begin{theorem}\label{thm:homo_case}
Assume the coordinates of $X_1$ are independent and identically distributed. If $d \geq 1+\ell/u$, then $r_* = \min_{k \in [d]} r_{\{k\}}$. Further, if $d \geq 2+2\ell/u$, then $\Pro_{\Theta^{(0)}}$ is asymptotically efficient with $\Theta^{(0)} = \{\beta^{\{k\}}: k \in [d]\}$.
\end{theorem}
\begin{proof}
See Appendix \ref{app:homo_streams}.
\end{proof}

The condition for the mixture $\Pro_{\Theta^{(0)}}$ to be A.E.~in Theorem~\ref{thm:homo_case} depends only on the dimension $d$ and the ratio $\ell/u$, not on the distribution of the random walk. In particular, if $u \geq 2\ell$ (respectively, $u \geq \ell$), then $\Pro_{\Theta^{(0)}}$ is asymptotically efficient for any $d \geq 3$ (respectively, $d \geq 4$).

If this condition is violated, that is, if $d < 2 + 2\ell/u$, then additional components may need to be added to $\Theta^{(0)}$ in order to achieve asymptotic efficiency. In such situations, the choice must be made on a case-by-case basis; for example, when $d$ is relatively small, it may be feasible to apply the full mixture $\{\Pro_{\beta^A} : A \in \mathcal{P}_{d}^-\}$ in Lemma~\ref{lemma:full_mixture_efficient}.

If $u \ll \ell$ such that $d < 1+\ell/u$, $r_*$ may be attained with $A = [d]$, that is, $r_{[d]} = r_*$. In other words, among all regions $\{bW^{A}: A\in \mathcal{P}_{d}^{-}\}$, $bW^{[d]}$ is the one most likely to be hit, as illustrated by the following numerical example. Note that in the homogeneous setup of this subsection, $r_A$ and $\beta^A$ depend only on the size of $A$. Thus, the computation of $\{r_A, \beta^A : A \in \mathcal{P}_{d}^-\}$ requires solving only $d$ optimization problems.

\begin{example}[i.i.d. coordinates] Assume that the coordinates of $X_1$ are i.i.d., and that $X_{1,1}$ has the normal distribution with mean $-1/2$ and variance $1$.

In Figure \ref{fig:homo_norm}, we let $u = 0.004, \ell = 1, d=50$, and plot $r_A$ for all possible sizes of $A$. In particular, $r_*$ is attained by $r_{[d]}$. For this example, by Thereom \ref{thm:main_strategy}, we can show that $\Pro_{\Theta}$ with $\Theta = \Theta^{(0)} \cup \{\beta^{[d]}\}$ (of size $d+1$) is asymptotically efficient, since for any $2 \leq m \leq d-1$ and $A =\{1,\ldots,m\}$, we can show numerically that
$    
v_A(\beta^{\{1\}})
\geq 2r_*.
$
\end{example}

 \begin{figure}[!tb]
     \centering
     \includegraphics[width=0.85\textwidth]{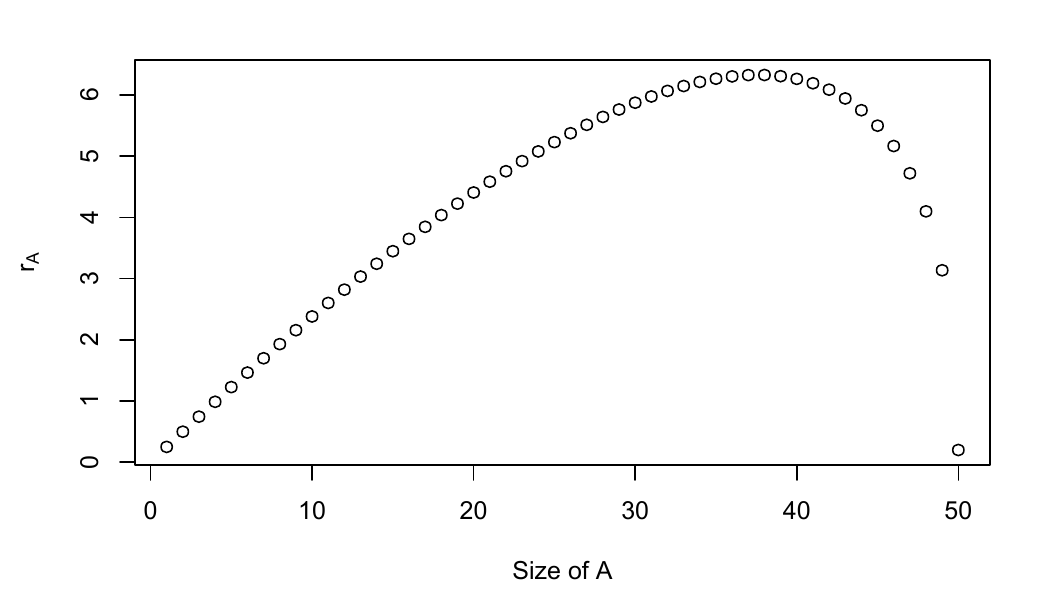}
     \caption{Homogeneous case with $u = 0.004, \ell=1, d = 50$. Here  $r_{\{1\}} \approx 0.2507 $ and $r_{[d]} = 0.2$.}
     \label{fig:homo_norm}
 \end{figure}

\subsection{Two numerical studies} 
First, we consider the i.i.d.~case as in Theorem \ref{thm:homo_case} with $\ell = u = 1, d = 400$. 
 Specifically, we denote by $\mathcal{E}(\lambda)$ the exponential distribution with rate $\lambda$, and assume that the coordinates of $X_1$ are i.i.d. and $X_{1,1}$ has the distribution $\mathcal{E}(2) - \log(2)$. In particular, $z_1 = 1$ and $\kappa^{(0)}_1 \approx 0.1931, \kappa^{(1)}_1 \approx 0.307$.   By Lemma \ref{lemma:independent_one_change}, $\beta^{\{k\}} \neq \gamma^{k}$. Theorem \ref{thm:homo_case} applies as long as $d \geq 4$, that is, $\Pro_{\Theta^{\{0\}}}$ is asymptotically efficient. For $k \in [d]$,
 under $\Pro_{\beta^{\{k\}}}$, $\{X_{1,j}, j \in [d]\}$ are independent, and $X_{1,j}$ has the distribution of $\mathcal{E}(2-\beta^{\{k\}}_j) - \log(2)$ for $j \in [d]$. For  $k \in [d]$, $\beta^{\{k\}}_j \approx 0.8718$ if $j = k$ and $\approx - 4.1194\times 10^{-4}$ if $j \neq k$.

Given $\Theta^{\{0\}}$, for each $b$, we obtain $N = 10^5$ realizations of $\widehat{\mathcal{\mathbb{W}}}_b(\Theta^{\{0\}})$ under $\Pro_{\Theta^{\{0\}}}$ as discussed in Remark \ref{remark:simulation strategy}. We then estimate the wrong exit probability by the average. Further, we compute the relative error as the estimated standard deviation of \textit{the estimator} (i.e., the sample standard deviation of these realizations normalized by $1/\sqrt{N}$) against the estimated probability. In Figure \ref{fig:homo_exp_d400}, the left panel plots the $-\log$ of the estimation probability as we vary the index $b$, while the right panel plots the relative error against $-\log_{10}$ of the estimated probability. As expected, the exponent of the estimated probability is linear in $b$. Further, for probabilities as small as $10^{-9}$, the relative error is less than $2\%$.

Second, we consider the correlated normal case in Example \ref{ex:cor_normal_siegmund} with $d = 100, \ell = u = 1$. Specifically, assume that $X_1$ follows a multivariate normal distribution with mean $-\frac{1}{2}\mathbbm{1}_d$ and covariance matrix $\Sigma_\rho$,  defined in~\eqref{def:cor_normal}.
By numerical computation, if $\rho \leq 0.51$, then condition \eqref{def:general_sufficient_condition} holds and thus $\Pro_{\Theta^{\{0\}}}$ is asymptotically efficient.
In Figure \ref{fig:corr_normal_two_rhos}, for $\rho = 0.2$ and $\rho = 0.45$, we plot the estimated relative error against $-\log_{10}$ of the estimated probability. We observe that for probabilities as small as $10^{-10}$, the estimated relative error is below $2\%$.

\begin{figure}[!tb]
    \centering
    \includegraphics[width=\textwidth]{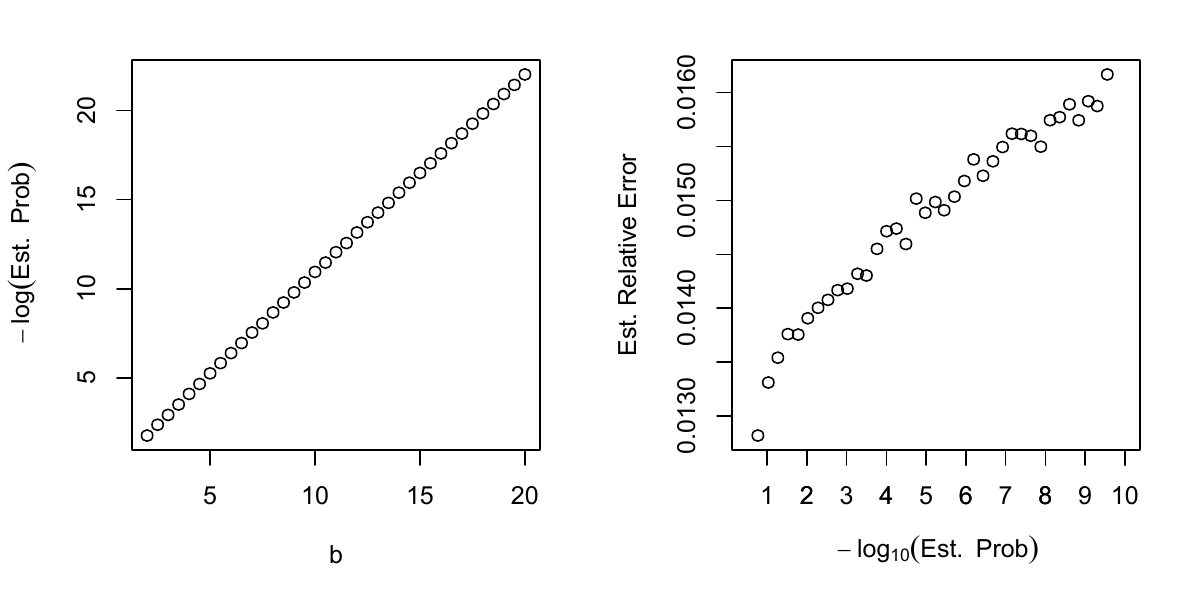}
    \caption{Multidimensional Siegumnd's problem with $d=400$ and $\ell = u = 1$. $X_1$ has i.i.d.~exponential coordinates. The left figure is $-\log$  of the estimated probability against the index $b$, and the right is the estimated relative error against  $-\log_{10}$ of the estimated probability. The results are based on $10^5$ repetitions.}
    \label{fig:homo_exp_d400}
\end{figure}

\begin{figure}[!tb]
    \centering
    
\includegraphics[width=\textwidth]{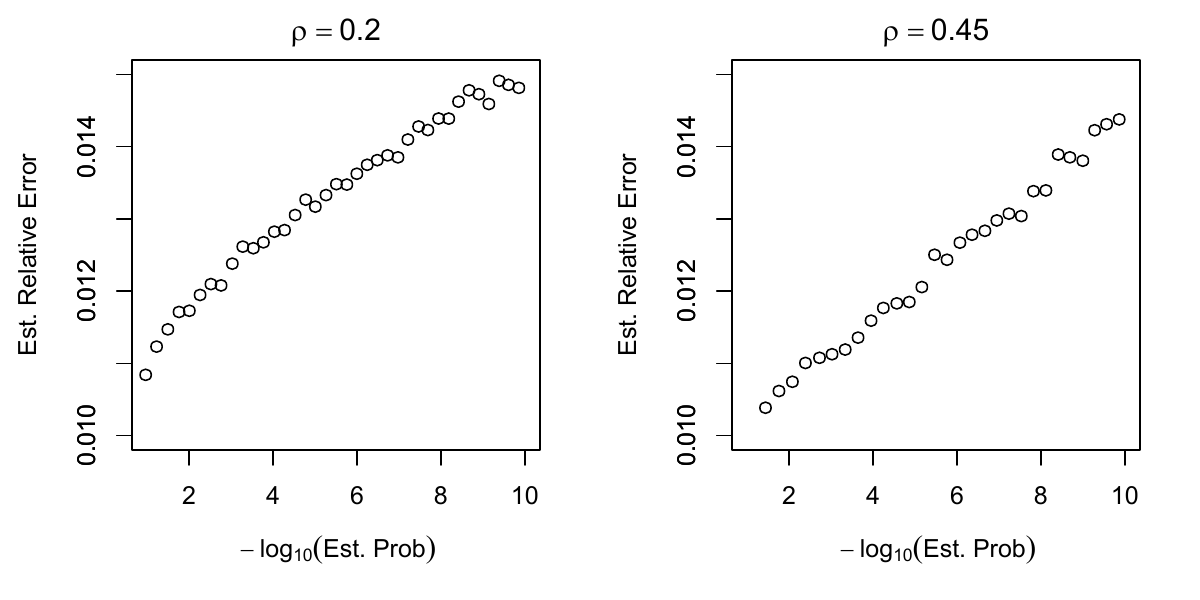}
    \caption{
    Siegumnd's problem with $d=100$ and $\ell = u = 1$.
    The coordinates of $X_1$ are correlated normals with correlation  $\rho = 0.2$ (left) and $\rho = 0.45$ (right). The x-axis is $-\log_{10}$ of the estimated probability, and the y-axis is the estimated relative error. The results are based on $10^5$ repetitions.}
    \label{fig:corr_normal_two_rhos}
\end{figure}

\subsection{Negative results}\label{subsec:negative}
In this subsection, we consider a \textit{modified} multidimensional Siegmund problem to illustrate a negative example. In this setting, any asymptotically efficient mixture $\mathbb{P}_{\Theta}$ must include exponentially many components in $d$, making it computationally infeasible.

Specifically, assume that $d$ is \emph{even}, and  that $\Exp[X_{1,k}] < 0$ for each $k \in [d]$, where $X_1 = (X_{1,1},\ldots,X_{1,d})^\top$. Let $\ell,u > 0$ be fixed. Denote by $\mathcal{S} = \{A \subset [d]: |A| = d/2\}$  all subsets of $[d]$ with size $d/2$. Suppose the goal is to use importance sampling to estimate the probability that the random hits 
$b W^{A}$ \textit{for some} $A \in \mathcal{S}$ before $b W^{\emptyset}$, that is,
\begin{align}
    \label{def:modified_siegmund}
    \Pro\left(\tilde{\tau}_b^{*} < \tau_b^{\emptyset} \right), \quad \text{ where } \; \tilde{\tau}_b^{*}:= \min\{\tau_b^{A}:\; A \in \mathcal{S}\},
\end{align}
and  $W^A$ is defined in  \eqref{def:W_A}. Compared to \eqref{def:siegmund}, the formulation in \eqref{def:modified_siegmund} involves fewer regions, yet still fits within the framework developed in Section \ref{sec:prob_wrong_exit}. Specifically, we let $J = \binom{d}{d/2}$ and identify $W^{\emptyset}$ with $W^{0}$ and $\{W^{A}, r_A, \beta^A:A \in \mathcal{S}\}$ with $\{W^j, r_j, \beta^j: j \in [J]\}$. Further, define
\[
r_* = \min\{r_A : A \in \mathcal{S}\}, \quad
\mathcal{C}_* = \{A \in \mathcal{S} : r_A = r_*\}
\]
as the minimal rate and the collection of regions that attain it, respectively.

Now, assume that the coordinates of $X_1 = (X_{1,1},\ldots,X_{1,d})^\top$ are exchangeable, that is, for any permutation $\sigma$ of $\{1,\ldots,d\}$, $(X_{1,\sigma(1)},\ldots,X_{1,\sigma(d)})^{\top}$ has the same distribution of $(X_{1,1},\ldots,X_{1,d})^\top$. This holds, for example, when the coordinates of $X_1$ are i.i.d..

Under exchangeability, it is clear that $r_A = r_*$ for every $A \in \mathcal{S}$. Consequently, $\mathcal{C}_* = \mathcal{S}$. Then, by Theorem \ref{thm:negative}, any asymptotically efficient mixture $\Pro_{\Theta}$ must include the set $\{\beta^A : A \in \mathcal{S}\}$ in $\Theta$, whose cardinality is $\binom{d}{d/2}$. This becomes computationally infeasible even for moderate values of $d$; for example, $\binom{40}{20} > 10^{11}$.

\section{Gap Rule from Sequential Multiple Testing}\label{sec:gap_rule}

Let $2 \leq m \leq d-2$ be an integer, and recall that $\Mstar = \{1,\ldots, m\}$. Assume that the first $m$ coordinates of the random walk have positive drift and the remaining ones negative, that is, 
$$
\Exp[X_{1,k}] > 0 \quad \text{ if } 1 \leq k \leq m, \qquad \text{ and } \qquad 
\Exp[X_{1,k}] < 0 \quad \text{ if } m+1 \leq k \leq d.
$$

In this section, we consider the first time that the top $m$ coordinates of the random walk $\{S_n\}$ exceed the remaining ones by at least $b > 0$, and we are interested in estimating the probability that, at the stopping time, at least one of these top $m$ coordinates does not belong to the set $[m]$, that is, the coordinates with positive drift \citep{song2017asymptotically}.

To be specific, let $\mathcal{G}_m$ denote the collection of subsets of $[d]$ with cardinality $m$, and let $\mathcal{G}_m^{-}$ denote the subsets of cardinality $m$ excluding $\Mstar$, i.e., 
$\mathcal{G}_m  :=  \{A \subset [d]: |A| = m\}$ and  $\mathcal{G}_m^{-}  :=  \mathcal{G}_m \setminus \{\Mstar\}$. For each $A  \in \mathcal{G}_m$ 
 let $\tau_b^A$ denote the first time the random walk hits  the set $b W^{A}$, where   
\begin{equation}
    \label{def:gap_W_A}
W^{A} = \{x \in \bR^d: x_{k} - x_{k'} > 1 \text{ for every } k \in A,\; k' \in A^c\},
\end{equation}
For large $b > 0$, we expect the random walk to exit from $b W^{\Mstar}$, and we are interested in estimating the probability that it hits   $b W^{A}$ for some 
$ A \in \mathcal{G}_m^{-}$ before  $b W^{\Mstar}$, i.e.,  
$$\Pro\left(  \min_{A \in \mathcal{G}_m^{-}} \tau_b^A < \tau_b^{[m]} \right).$$

  To connect with the notation in Section~\ref{sec:prob_wrong_exit}, 
  we identify $\{W^{\Mstar}, \tau_b^{\Mstar}\}$ with $\{W^{0}, \tau_b^{0}\}$ and $\{W^{A}, \tau_b^{A}: A \in \mathcal{G}_m^{-}\}$ with $\{W^{j}, \tau_b^{j}: j \in [J]\}$; in particular, $J = \binom{d}{m}-1$. Similar identification applies to $\{r_A, v_A(\gamma), \beta^{A}: A \in \mathcal{G}_m^{-}\}$. 

Note that conditions \eqref{assumptions_directions_of_mean} and \eqref{assumption_on_regions}  always hold in this setup. In addition, throughout this section, we assume that \eqref{assumption_on_X} holds.

\begin{remark}
In general, if there exists $a \in \mathbb{R}$ such that $\mathbb{E}[X_{1,k}] > a$ for all $k \in \Mstar$ and $\mathbb{E}[X_{1,k'}] < a$ for all $k' \not\in \Mstar$, then with high probability, the random walk $\{S_n\}$ exits through the region $b W^{\Mstar}$. In this case, one may instead consider the centered random vector $X_1 - a$, where the subtraction is understood coordinatewise.
Furthermore, the gap is assumed to be $1$ in the definition \eqref{def:gap_W_A} of $W^A$. This is without loss of generality, as any constant gap can be absorbed into the scaling parameter $b$.
\end{remark}

\subsection{A general proposal}
If $d$ and $m$ are moderately large, the full mixture in Lemma \ref{lemma:full_mixture_efficient} with $\binom{d}{m}-1$ components
is computationally infeasible; for example, $\binom{50}{20} > 10^{13}$.  We again apply Theorem \ref{thm:main_strategy} to derive sufficient conditions under which mixtures with significantly fewer components are asymptotically efficient.

The following lemma characterizes $\{r_A, \beta^{A}: A \in \mathcal{G}_m^{-}\}$ and lower bounds $v_A(\gamma)$ 
for $\gamma \in \operatorname{dom}(\Lambda)$, following  \eqref{def:r_j_alt} and \eqref{def:lower_bound_V_gamma}.

\begin{lemma}\label{lemma:gap_opt_probs}
Let $ A \in \mathcal{G}_m^{-}$. Then $r_A$ and $\beta^A$ are the optimal value and solution of the following optimization problem:
\begin{align*}
 \max_{\theta \in \bR^d} \;\;  \sum_{k \in A} \theta_k, \quad &\textup{ subject to }   \Lambda(\theta) \leq 0, \quad \sum_{i \in [d]} \theta_i = 0, \\
& \quad \theta_k \geq 0\; \text{ for } \;k \in A, \quad \theta_{k'} \leq 0\; \text{ for } \;k' \in A^c.
\end{align*}
Further, for each $\gamma \in \bR^d$ such that $\Lambda(\gamma) \leq 0$, $v_A(\gamma)$ is lower bounded by
\begin{align*}
\max_{\theta \in \bR^d} \;\;  \sum_{k \in A} \theta_k, \; &\textup{ subject to }  \Lambda(\theta - \gamma) \leq 0, \quad \sum_{i \in [d]} \theta_i = 0, \\
&\theta_k \geq 0\; \text{ for } \;k \in A, \quad \theta_{k'} \leq 0\; \text{ for } \;k' \in A^c.
\end{align*}
\begin{proof}
    See Appendix \ref{app:gap_rule}.
\end{proof}
\end{lemma}

Following the strategy in Subsection \ref{subsec:strategy}, we begin with guessing $\mathcal{C}_*$ in \eqref{def:C_star}. As discussed above, with high probability, the random walk $\{S_n\}$ exits from $bW^{\Mstar}$. 
When this is not the case, we typically expect the exit to occur through $bW^{A}$ for some 
$A \in \mathcal{G}_m^{-}$ that is \emph{not} too different from  $\Mstar$. Thus define
\begin{align*}
    \mathcal{G}_{m}^{1}  :=  \{(\Mstar \setminus \{k\}) \cup \{k'\}:\;\; k \in \Mstar, \;\; k' \not \in \Mstar\},
\end{align*}
which are subsets that have size $m$ and differ from  $\Mstar$ in two indexes. Based on this intuition, we set $\mathcal{C} = \mathcal{G}_{m}^{1}$, and include the corresponding optimal tiltings $\{\beta^{A} : A \in \mathcal{G}_{m}^{1}\}$, as defined in Definition \ref{def:optimal_beta_j}, in a candidate set $\Theta$, for which   $\Pro_{\Theta}$ serves as a potential proposal.

Further, we include in $\Theta$ the following collections of parameters in $\operatorname{dom}(\Lambda)$.  First,
for $\ell \in \Mstar$ and $\ell' \not\in \Mstar$, denote by $\tilde{z}_{\ell,\ell'}$ and $\tilde{\gamma}^{\ell,\ell'}$ the optimal value and solution to the following optimization problem:
\begin{align}\label{gap:2index_opt}
\begin{split}
 \max_{\theta \in \bR^d} \;\;  \theta_{\ell'} \quad &\textup{ subject to }   \Lambda(\theta) \leq 0, \quad  \theta_{\ell} + \theta_{\ell'} = 0, \\
& \quad \theta_{\ell} \leq 0, \quad \theta_{\ell'} \geq 0,\quad 
\theta_k = 0\; \text{ for } \;k \not\in \{\ell,\ell'\}.
\end{split}
\end{align}

Second, for $\ell_1 \neq \ell_2 \in \Mstar$ and $\ell_1' \neq \ell_2' \not \in \Mstar$, denote by $\tilde{s}_{\ell_1,\ell_2,\ell_1',\ell_2'}$ and $\tilde{\gamma}^{\ell_1,\ell_2,\ell_1',\ell_2'}$ the optimal value and solution to the following optimization problem:
\begin{align}\label{gap:4index_opt}
\begin{split}
    & \max_{\theta \in \bR^d} \;\;  \theta_{\ell_1'} + \theta_{\ell_2'} \quad \textup{ subject to }   \Lambda(\theta) \leq 0, \quad \theta_{\ell_1} + \theta_{\ell_2} + \theta_{\ell'_1} +\theta_{\ell'_2}  = 0, \\
 &\theta_{\ell_1} \leq 0, \quad \theta_{\ell_2} \leq 0, \quad \theta_{\ell_1'} \geq 0,\quad \theta_{\ell_2'} \geq 0, \quad  \theta_k = 0\; \text{ for } \;k \not\in \{\ell_1,\ell_2,\ell_1',\ell_2'\}.
 \end{split}
\end{align}

Let $A \in \mathcal{G}_m^{-} \setminus\mathcal{G}_m^{1}$. That is, $A$ has size $m$ and differs from $\Mstar$ in more than two elements. Thus there exist
\begin{align}\label{aux:A_ell_12_ellp_12}
    \ell_1,\ell_2 \in \Mstar \setminus A, \quad \text{ and } \quad
\ell_1',\ell_2' \in A\setminus \Mstar.
\end{align}
For the second optimization problem in Lemma \ref{lemma:gap_opt_probs} with $\gamma = \tilde{\gamma}^{\ell_1,\ell_1'}$, a feasible solution is given by $\tilde{\gamma}^{\ell_1,\ell_1'} + \tilde{\gamma}^{\ell_1,\ell_2, \ell_1',\ell_2'}$, and as a result, we have
\begin{align}
    \label{aux:gap_lower_bound_12}
    v_A(\gamma^{\ell_1,\ell_1'}) \geq \tilde{z}_{\ell_1,\ell_1'} + \tilde{s}_{\ell_1,\ell_2,\ell_1',\ell_2'}.
\end{align}
In other words, $\tilde{z}_{\ell_1,\ell_1'} + \tilde{s}_{\ell_1,\ell_2,\ell_1',\ell_2'}$ provides a lower bound for $v_A(\gamma^{\ell_1,\ell_1'})$ for all $A \in \mathcal{G}_m^{-} \setminus \mathcal{G}_m^{1}$ such that \eqref{aux:A_ell_12_ellp_12} holds. In this sense, $\gamma^{\ell_1,\ell_1'}$ can be used to cover all such $A \in \mathcal{G}_m^{-} \setminus \mathcal{G}_m^{1}$ satisfying $\ell_1 \in \Mstar \setminus A$ and $\ell_1' \in A \setminus \Mstar$.


Similarly, for the second optimization problem in Lemma \ref{lemma:gap_opt_probs} with $\gamma = \tilde{\gamma}^{\ell_1,\ell_2,\ell_1',\ell_2'}$, a feasible solution is given by $ 2 \tilde{\gamma}^{\ell_1,\ell_2, \ell_1',\ell_2'}$, and as a result,
\begin{align}
    \label{aux:gap_lower_bound_12_prime}
    v_A(\gamma^{\ell_1,\ell_2,\ell_1',\ell_2'}) \geq 2\tilde{s}_{\ell_1,\ell_2,\ell_1',\ell_2'}.
\end{align}
Thus, $2\tilde{s}_{\ell_1,\ell_2,\ell_1',\ell_2'}$ is a lower bound for $v_A(\gamma^{\ell_1,\ell_2,\ell_1',\ell_2'})$ for all $A \in \mathcal{G}_m^{-} \setminus \mathcal{G}_m^{1}$ such that \eqref{aux:A_ell_12_ellp_12} holds, and in this sense, $\gamma^{\ell_1,\ell_2,\ell_1',\ell_2'}$ can be used to cover this collection of $A$.


\begin{theorem}
\label{thm:Gap_sufficient}
Consider the following two conditions.
\begin{align*}
& (H1'): \;\;\min_{\ell_1 \neq \ell_2 \in \Mstar, \ \ell_1' \neq \ell_2'  \not \in \Mstar } \left(\tilde{z}_{\ell_1, \ell_1'} + 
\tilde{s}_{\ell_1,\ell_2,\ell_1',\ell_2'}\right) \geq \min_{A \in \mathcal{G}_m^{1}} 2  r_{A}.\\
& (H2'): \;\;\min_{\ell_1\neq \ell_2 \in \Mstar,\  \ell_1'\neq \ell_2'  \not\in \Mstar} 
2\tilde{s}_{\ell_1,\ell_2,\ell_1',\ell_2'}\geq  \min_{A \in \mathcal{G}_m^{1}} 2r_{A}.
\end{align*}
Under either condition, $r_* =  \min_{A \in \mathcal{G}_m^{1}} r_{A}$. Further, define $\widetilde{\Theta}^{(0)}  :=  \{\beta^{A}: A \in \mathcal{G}_m^{1}\}$ and
\begin{align*}
&\widetilde{\Theta}^{(1)}  :=  \Theta^{(0)} \cup \{\tilde{\gamma}^{\ell,\ell'}: \ell \in \Mstar, \ell'  \not \in \Mstar\}\\
&\widetilde{\Theta}^{(2)}  :=  \Theta^{(0)} \cup \{\tilde{\gamma}^{\ell_1,\ell_2,\ell_1',\ell_2'}: \ell_1 \neq \ell_2 \in \Mstar, \ell_1' \neq \ell_2' \not \in \Mstar\}.
\end{align*}
If $(H1')$ holds, then $\Pro_{\widetilde{\Theta}^{(1)}}$ is asymptotically efficient. If $(H2')$ holds, then $\Pro_{\widetilde{\Theta}^{(2)}}$ is asymptotically efficient.
\end{theorem}
\begin{proof}
The claims under condition $(H1')$ (resp.~$(H2')$) follow directly from equation \eqref{aux:gap_lower_bound_12} (resp.~\eqref{aux:gap_lower_bound_12_prime}) and Theorem~\ref{thm:main_strategy}.
\end{proof}

The above theorem suggests solving $m(d-m)$ optimization problems to compute $\{r_A, \beta^{A}: A \in \mathcal{G}_m^{1}\}$, in addition to $m(d-m) + \binom{m}{2}\binom{d-m}{2}$ optimizations problems to obtain 
$\{\tilde{z}_{\ell,\ell'}, \tilde{\gamma}^{\ell,\ell'}: \ell \in \Mstar, \ell' \not \in \Mstar\}$ and
$\{\tilde{s}_{\ell_1,\ell_2,\ell_1',\ell_2'}, \tilde{\gamma}^{\ell_1,\ell_2,\ell_1',\ell_2'}: \ell_1\neq \ell_2 \in \Mstar, \ell_1'\neq \ell_2' \not \in \Mstar\}$. Then if $(H1')$ (resp.~$(H2')$) holds, a mixture of $2m(d-m)$ (resp.~$m(d-m)+\binom{m}{2}\binom{d-m}{2}$) components is asymptotically efficient.

\begin{remark}\label{remark:gap_weaker}
Since $\tilde{\gamma}^{\ell_1,\ell_1'}$ in \eqref{gap:2index_opt} is a feasible solution to the problem in \eqref{gap:4index_opt}, we have $\tilde{z}_{\ell_1,\ell_1'} \leq s_{\ell_1,\ell_2,\ell_1',\ell_2'}$, and thus 
condition $(H2')$ is weaker than $(H1')$. By adding more elements from $\operatorname{dom}(\Lambda)$ to $\widetilde{\Theta}^{(2)}$,  $(H2')$ can be further relaxed, at the cost of additional computation. 
\end{remark}

\subsection{Examples and special cases} \label{subsec:gap_examples}

\begin{example}[correlated normals---Gap rule]\label{example:corr_normal_Gap} Let $\mu_{+} > 0$, $\mu_{-} < 0$, $\sigma^2 > 0$ and $\rho \in [0,1)$ be constants.  Assume that $X_1$ has the multivariate normal distribution such that 
\begin{align*}
\Exp[X_{1,k}] = \begin{cases}
\mu_{+}, \; & \text{ if } k \in \Mstar\\
\mu_{-},\; & \text{ if } k' \not\in \Mstar
\end{cases}, \;\;
\text{Cov}(X_{1,k}, X_{1,k'}) = 
\begin{cases}
\sigma^2, \; & \text{ if } k = k' \\
\sigma^2 \rho,\; & \text{ if }  k \neq k'
\end{cases}.
\end{align*}
In Appendix \ref{app:gap_example_corrnormal_proof}, we show that for any $\ell \in \Mstar, \; \ell' \not \in \Mstar$,
$$r_{A} = \tilde{z}_{\ell,\ell'}, \quad
\beta^{A} = \tilde{\gamma}^{\ell,\ell'}, \text{ where }
A = (\Mstar \setminus \{\ell\}) \cup \{\ell'\}.
$$
Thus, $\widetilde{\Theta}^{(0)} = \widetilde{\Theta}^{(1)}$. 
Due to Remark \ref{remark:gap_weaker},
we have that $(H1')$ holds.  By Theorem \ref{thm:Gap_sufficient},
$r_* =  \min_{A \in \mathcal{G}_m^{1}} r_{A}$, and 
the mixture $\Pro_{\widetilde{\Theta}^{0}}$ of $m(d-m)$ components is asymptotically efficient.

In Figure \ref{fig:Gap_d100m50},  we present a numerical study that uses $\Pro_{\widetilde{\Theta}^{0}}$ as the importance sampling proposal, under the setup that $\mu_{+} = -\mu_{-} = 1/2$, $\sigma^2 = 1$, $\rho \in \{0.1,0.4\}$, $d = 100$ and $m = d/2$. We use $N = 10^6$ repetitions and plot the estimated relative error against $-\log_{10}$ of the estimated probability. We observe that for probabilities as small as $10^{-10}$, the estimated relative error is below $3.5\%$. In Appendix \ref{app:gap_example_corrnormal_more_sim}, we present additional numerical results for $m = 10$ and $m = 30$, with other parameters the same.

\begin{figure}[!t]
    \centering
    \includegraphics[width = \textwidth]{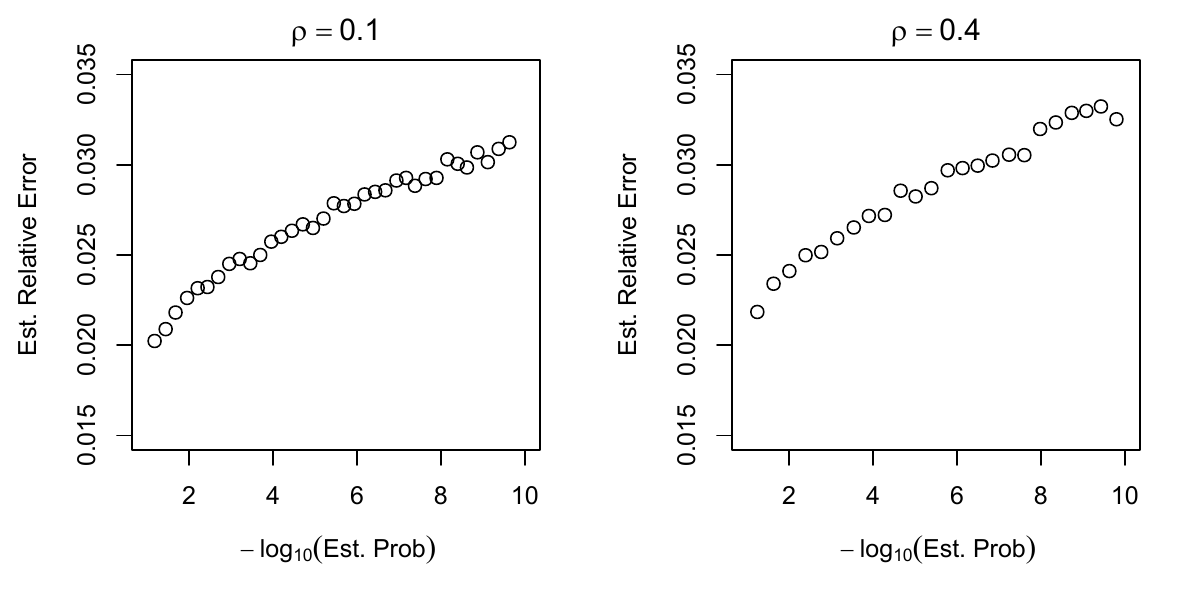}
    \caption{Gap rule with $d=100$ and $m = 50$ under the setup of Example \ref{example:corr_normal_Gap}. The x-axis is $-\log_{10}$ of the estimated probability, and the y-axis is the estimated relative error. The results are based on $10^6$ repetitions.}
    \label{fig:Gap_d100m50}
\end{figure}

\end{example}

\begin{example}[independent normals---Gap rule]\label{example:gap_ind_normal}
Assume that the coordinates of $X_1$ are independent. For $k \in [d]$, $X_{1,k}$ has the normal distribution such that if $k \in \Mstar$, $\Exp[X_{1,k}] = 1/2,  \text{Var}(X_{1,k}) = 1$; otherwise, $\Exp[X_{1,k}] = -1/2,  \text{Var}(X_{1,k}) = v$, where $v > 0$. Elementary 
calculation shows that for $\ell_1 \neq \ell_2 \in \Mstar, \ \ell_1' \neq \ell_2'  \not\in \Mstar$,
$$
\tilde{z}_{\ell_1, \ell_1'} = 2/(1+v),\quad
\tilde{s}_{\ell_1,\ell_2,\ell_1',\ell_2'} = 4/(1+v). 
$$
Therefore,  condition  $(H1')$ (resp. $(H2')$)  in Theorem \ref{thm:Gap_sufficient} holds, and as a result $\Pro_{\widetilde{\Theta}^{(1)}}$ (resp. $\Pro_{\widetilde{\Theta}^{(2)}}$) is asymptotically efficient, 
if  $3/(1+v) \geq  \min_{A \in \mathcal{G}_m^{1}} r_{A}$
(resp.  $4/(1+v) \geq  \min_{A \in \mathcal{G}_m^{1}} r_{A}$).  

In Figure \ref{fig:Gap_ind_normals} we set $d = 50, m = 25$,  and plot $3/(1+v) - \min_{A \in \mathcal{G}_m^{1}} r_{A}$ (bottom dashed line) and  $4/(1+v) - \min_{A \in \mathcal{G}_m^{1}} r_{A}$ (top solid line), respectively, as a function of $v$. By Theorem \ref{thm:Gap_sufficient}, if the bottom (resp. top) line is non-negative, then $(H1')$ (resp. $(H2')$) holds. In particular, if $v \in (0.15,6.95)$ (resp. $v \in (0.071, 14.15)$), then $\Pro_{\widetilde{\Theta}^{(1)}}$ (resp. $\Pro_{\widetilde{\Theta}^{(2)}}$) is asymptotically efficient.
\end{example}

\begin{figure}[!tb]
    \centering
    \includegraphics[width=0.85\textwidth]{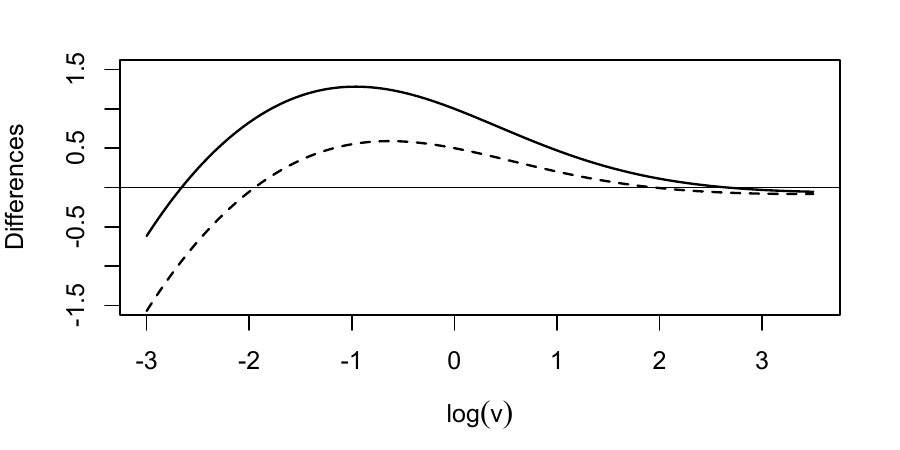}
    \caption{Plot the two differences in Example \ref{example:gap_ind_normal} verse $\log(v)$ for $d=50,m=25$.}
    \label{fig:Gap_ind_normals}
\end{figure}

\begin{example}[Independent, ``homogeneous'' coordinates]\label{example:gap_ind_hom}
Assume the coordinates of $X_{1}$ are independent. For $k \in [d]$, denote by $\Lambda^{k}$ the cumulant generating function of $X_{1,k}$. Assume that for some $I_{+} >0$ and $z, I_{-} < 0$, 
\begin{align}\label{def:gap_example_homo}
\begin{split}
&\Lambda^{k}(z) = 0, \quad (\Lambda^{k})'(0) = I_{+}, \quad
(\Lambda^{k})'(z) = I_{-}, \quad\text{ for } k \in \Mstar, \\
&\Lambda^{k'}(-z) = 0, \quad (\Lambda^{k'})'(0) = I_{-}, \quad
(\Lambda^{k'})'(-z) = I_{+}, \;\;\; \text{ for } \;k' \not\in \Mstar.
\end{split}
\end{align}
By the KKT conditions in Appendix \ref{app:KKT_gap},   for $\ell \in \Mstar$ and $\ell' \not\in \Mstar$,
\begin{align*}
&\tilde{\gamma}^{\ell,\ell'}_{\ell} = z, \quad
\tilde{\gamma}^{\ell,\ell'}_{\ell'} = -z, \quad
\tilde{\gamma}^{\ell,\ell'}_{k} = 0\; \text{ for } \;k \not \in \{\ell,\ell'\}.\\
&r_A = \tilde{z}_{\ell,\ell'}, \quad \beta^A = \tilde{\gamma}^{\ell,\ell'}, \text{ where } A = (\Mstar \setminus \{\ell\}) \cup \{\ell'\}.
\end{align*}
Then by the same argument as for Example \ref{example:corr_normal_Gap},
$r_* =  \min_{A \in \mathcal{G}_m^{1}} r_{A}$, and 
the mixture $\Pro_{\widetilde{\Theta}^{0}}$ of $m(d-m)$ components is asymptotically efficient, which recovers Theorem 2 in \cite{song2016logarithmically} as a special case of this example.
\end{example}

\begin{remark}
Fix $\ell \in \Mstar$ and $\ell' \not\in \Mstar$. Consider the setup in Example \ref{example:gap_ind_hom}. Under $\Pro_{\tilde{\gamma}^{\ell,\ell'}_{\ell}}$, $X_{1,k}$ has the same distribution as under $\Pro$ if $k \not \in \{\ell,\ell'\}$,  while the distributions of $X_{1,\ell}$ and $X_{1,\ell'}$ are changed to the Siegmund exponential-tilting distributions of their original distributions under $\Pro$.
\end{remark}

\begin{remark}
A sufficient condition for \eqref{def:gap_example_homo} is as follows. Assume $\{X_{1,k}: k \in \Mstar\}$ are i.i.d. with distribution $\nu$ and $\{X_{1,k'}: k' \not\in \Mstar\}$ are i.i.d. with distribution $\tilde{\nu}$. If $\tilde{\nu}$ is the Siegmund exponential-tilting distributions of $\nu$, then \eqref{def:gap_example_homo} holds.
\end{remark}

\section{Sum-intersection Rule}\label{sec:sum_inter}
In this section, we consider the \textit{sum-intersection rule} proposed in \cite{song2019sequential}. Let $1 \leq L \leq d-1$ be a fixed integer, and suppose that all coordinates of the random walk have negative expectation, that is,  $\mathbb{E}[X_{1,k}] < 0$ for $k \in [d]$. Consider the first time the sum of the $L$ smallest (in absolute value) coordinates of the random walk exceeds $b>0$. We are interested in   estimating the probability that upon stopping at least $L$ coordinates are positive. When $L = 1$, this reduces to the multidimensional Siegmund problem in Section \ref{sec:multi_Siegmund} with $\ell=u=1$. Thus, we focus on the case that $L \geq 2$.

To formalize this, denote by $\mathcal{S}_L  :=  
\{A \subset [d]: |A| \geq L\}$ the collection of all subsets with size \textit{at least} $L$. For each $A \in \mathcal{S}_L$, denote by $\tau_b^{A}$ the first time the random walk hits the set $b W^A$, where we define
\begin{align}
\begin{split}
W^{A}  :=  \Big\{x \in \bR^d:\;\;& x_k > 0\; \text{ for } \;k \in A, \;\; x_{k'}< 0\; \text{ for } \;k' \in A^c,  \;\;  \\
&  
 \sum_{k \in B} |x_k| > 1\; \text{ for every } \;B \in \mathcal{G}_L \Big\},
 \end{split}
\label{def:SI_W_A}
\end{align}
and as before  $\mathcal{G}_L = \{B \subset [d]: |B| = L\}$ denotes the collection of subsets of size $L$. We are interested in  estimating
$$\Pro\left( \min_{A \in \mathcal{S}_L} \tau_b^A <  \min_{A \notin \mathcal{S}_L} \tau_b^A \right)$$
To connect with the problem formulation in Section \ref{sec:prob_wrong_exit}, we let
$W^{0} = \cup_{A \not\in \mathcal{S}_L} W^{A}$ 
and index $\{W^{A}: A \in \mathcal{S}_L\}$ with $[J]$, where $J := 2^{d} - \sum_{i=0}^{L-1}\binom{d}{i}$. Observe that $W^{0}$ is an open, yet non-convex set, while  $W^{j}$ is open and convex for every $j \in [J]$.  Moreover, conditions \eqref{assumptions_directions_of_mean} and  \eqref{assumption_on_regions} always hold in this setup, and we additionally assume that condition \eqref{assumption_on_X} holds.

As before, we first characterize $\{r_A, \beta^{A}: A \in \mathcal{S}_L\}$ and lower bound $v_A(\gamma)$ 
for $\gamma \in \operatorname{dom}(\Lambda)$, following  \eqref{def:r_j_alt} and \eqref{def:lower_bound_V_gamma}.

\begin{lemma}\label{lemma:sum_intersection_opt_probs}
For each $A \in \mathcal{S}_L$, $r_A$ and $\beta^A$ are the optimal value and solution of the following optimization problem:
\begin{align*}
& \max_{\theta \in \bR^d}    \min_{1 \leq \ell \leq L} \;\; \frac{1}{\ell} \sum_{i = L - \ell + 1}^{d} |\theta|_{(i)}, \quad \textup{ subject to }  \\
& \Lambda(\theta) \leq 0, \quad \theta_k \geq 0 \;\text{ for }\; k \in A, \quad \theta_{k'} \leq 0 \;\text{ for }\; k' \in A^c,
\end{align*}
where  $|\theta|_{(1)} \geq |\theta|_{(2)} \geq \ldots \geq |\theta|_{(d)}$ denotes the decreasing rearrangement of the absolute values  $\{|\theta_i|: i \in [d]\}$ for any $\theta \in \bR^d$. 
Further, for each $\gamma \in \bR^d$ such that $\Lambda(\gamma) \leq 0$, $v_A(\gamma)$ is lower bounded by
\begin{align*}
& \max_{\theta \in \bR^d} \;\;   \min_{1 \leq \ell \leq L} \;\; \frac{1}{\ell} \sum_{i = L - \ell + 1}^{d} |\theta|_{(i)}, \quad \textup{ subject to }  \\
& \Lambda(\theta - \gamma) \leq 0, \quad \theta_k \geq 0\; \text{ for } \;k \in A, \quad \theta_{k'} \leq 0\; \text{ for } \;k' \in A^c.
\end{align*}
\end{lemma}
\begin{proof}
    See Appendix \ref{app:sum_intersection}.
\end{proof}
\begin{remark}\label{remark:sum-intersection_convex}
As shown in Appendix \ref{app:sum_intersection} (see \eqref{opt:sum_inter_equiv} and \eqref{opt:sum_inter_equiv2}), the above optimization problems are both convex and can be formulated with $\binom{d}{L}+d$ decision variables with one non-linear convex constraint and $\binom{d}{L}+2d$ linear constraints.
\end{remark}

We follow the strategy in Section \ref{subsec:strategy}.
Since each coordinate of the random walk $\{S_n\}$ has a negative drift, on the rare event that
it exists from $b W^{A}$ for some $A \in \mathcal{S}_L$, intuitively the number of coordinates having negative values is mostly likely to be as small as possible, which is $L$. Thus, it is reasonable to set $\mathcal{C} = \mathcal{G}_L$, and include the optimal tiltings $\{\beta^{A}: A\in \mathcal{G}_L\}$, as defined in Definition \ref{def:optimal_beta_j}, in a candidate set $\Theta$, for which the associated mixture measure $\Pro_{\Theta}$ serves as a potential proposal.

Next, we show that a mixture, containing $\{\beta^{A}: A \in \mathcal{G}_L \}$ and some additional components, is asymptotically efficient. First, for $A \in \mathcal{G}_L$, define $z_A$ and $\gamma^A$ to be the optimal value and solution to the following optimization problem:
\begin{align}\label{opt:sum-inter_L}
\begin{split}
& \max_{\theta \in \bR^d} \;\; |\theta|_{(L)}, \;\; \textup{ subject to }  \Lambda(\theta) \leq 0, \quad \theta_k \geq 0 \;\text{ for }\; k \in A, \quad \theta_{k'} = 0 \;\text{ for }\; k' \in A^c,
\end{split}
\end{align}
where  $|\theta|_{(1)} \geq |\theta|_{(2)} \geq \ldots \geq |\theta|_{(d)}$ denotes the decreasing rearrangement of the absolute values $\{|\theta_i|: i \in [d]\}$ for any $\theta \in \bR^d$. In particular, a feasible solution contains at most $L$ non-zero coordinates.

Second,  for $B \in \mathcal{G}_{L+1}$, define $s_B$ and $\tilde{\gamma}^B$ to be the optimal value of the following optimization problem:
\begin{align}\label{opt:sum-inter_L_1}
\begin{split}
& \max_{\theta \in \bR^d} \min_{1 \leq \ell \leq L} \;\; \frac{1}{\ell} \sum_{i = L - \ell + 1}^{L+1} |\theta|_{(i)}, \quad \textup{ subject to } \\
& \Lambda(\theta) \leq 0, \quad \theta_k \geq 0 \;\text{ for }\; k \in B, \quad \theta_{k'} = 0 \;\text{ for }\; k' \in B^c.
\end{split}
\end{align}
In particular, a feasible solution contains at most $L+1$ non-zero coordinates.

Fix any $A \subset [d]$ such that $|A| \geq L+1$. Let $\tilde{A} \subset A$ such that $|\tilde{A}| = L$, and let $\tilde{k} \in A \setminus \tilde{A}$. For the second optimization problem in Lemma \ref{lemma:sum_intersection_opt_probs} with $\gamma = \gamma^{\tilde{A}}$, we have that $\gamma^{\tilde{A}} + \tilde{\gamma}^{\tilde{A}\cup \{\tilde{k}\}}$ is a feasible solution, and in the proof of Theorem \ref{thm:sum_intersection_rule}, we show that this implies
\begin{align*}
    v_{A}(\gamma^{\tilde{A}}) \geq   \min_{1 \leq \ell \leq L} \;\; \frac{1}{\ell} \sum_{i = L - \ell + 1}^{d} \left|\gamma^{\tilde{A}} + \tilde{\gamma}^{\tilde{A}\cup \{\tilde{k}\}}\right|_{(i)} \geq z_{\tilde{A}} + s_{\tilde{A} \cup \{\tilde{k}\}}.
\end{align*}
In other words, $z_{\tilde{A}} + s_{\tilde{A} \cup \{\tilde{k}\}}$ is a lower bound for $ v_{A}(\gamma^{\tilde{A}}) $ for all $A \subset [d]$ such that $\tilde{A} \cup \{\tilde{k}\} \subset A$, and in this sense $\gamma^{\tilde{A}}$ covers for all $A \subset [d]$ that contains $\tilde{A}$.

\begin{theorem}\label{thm:sum_intersection_rule}
Assume that 
\begin{align*}
(H\text{-}SI):\;\;
 \min_{A \in \mathcal{G}_L, k \in A^c}   z_A + s_{A \cup \{k\}} \geq \min_{A \in \mathcal{G}_L} 2 r_A.
\end{align*}
Then $r_* = \min_{A \in \mathcal{G}_L} r_A$. Further, define 
$$
\Theta^{(SI)}  :=   \{\beta^{A}:\; A \in \mathcal{G}_L\} \; \cup \; \{\gamma^{A}: \; A \in \mathcal{G}_L\}.
$$
Then $\Pro_{\Theta^{(SI)}}$ is asymptotically efficient. 
\end{theorem}

\begin{proof}
See Appendix \ref{app:proof_theorem_SI}.
\end{proof}

The above theorem suggests solving $\binom{d}{L}$ optimization problems to compute $\{r_A, \beta^{A}: A \in \mathcal{G}_L\}$, in addition to $\binom{d}{L} + \binom{d}{L+1}$ optimizations problems to obtain 
$\{z_A,  {\gamma}^{A}:  A \in \mathcal{G}_L\}$ and
$\{s_B:  B \in \mathcal{G}_{L+1}\}$. Then, if the condition $(H\text{-}SI)$ holds, a mixture of $2 \binom{d}{L}$ components is asymptotically efficient. As before, we can weaken the condition by adding more components to the mixture.

\begin{example}[correlated normals---sum-intersection rule]\label{example:SI_corr}
Assume that $X_1$ has the multivariate normal distribution with mean $-\frac{1}{2}\mathbbm{1}_d$ and covariance matrix $\Sigma_\rho$  defined in~\eqref{def:cor_normal}, whose cumulant generating function is given by \eqref{example:cor_normal_Lambda}.  For $A \in \mathcal{G}_L$ and $B \in \mathcal{G}_{L+1}$, due to exchangability,
 $\gamma^{A}_k, k \in A$ are identical, and 
 so are $\tilde{\gamma}^{B}_k, k \in B$. Thus, the optimal value  of the optimization problem in \eqref{opt:sum-inter_L} is equal to 
 \begin{align*}
     \max_{t \geq 0}\; t, \text{ such that } -\frac{1}{2} L t + \frac{1}{2}  L(1+(L-1)\rho) t^2 \leq 0,
 \end{align*}
 while the optimal value of \eqref{opt:sum-inter_L_1}  to
 \begin{align*}
\max_{t \geq 0} \;\frac{L+1}{L} t, \text{ such that } -\frac{1}{2} (L+1) t + \frac{1}{2} (L+1)(1+L\rho) t^2 \leq 0.
 \end{align*}
As a result, the optimal values  of the optimization problems  in \eqref{opt:sum-inter_L} and \eqref{opt:sum-inter_L_1}  are, respectively, 
$$
z_A = \frac{1}{\rho L + (1-\rho)}, \quad
s_B =
\frac{1}{\rho (L+1) + (1-\rho)} \times \frac{L+1}{L}.
$$

Let $d = 50$. For $L = 2$ (resp. $L=3$), by numerical computation, if $\rho \leq 0.23$ (resp. $\rho \leq 0.14$), condition $(H\text{-}SI)$ in Theorem \ref{thm:sum_intersection_rule} holds, and thus $\Pro_{\Theta^{(SI)}}$ is asymptotically efficient. Note that for $L=2$ and $L=3$, the size of $\Theta^{(SI)}$ is $2,450$ and $39,200$ respectively.

Next, we estimate the wrong exit probability under $\Pro_{\Theta^{(SI)}}$ for $d=50, L \in \{2,3\}$ and $\rho = 0.1$ with $N = 10^6$ repetitions. In Figure \ref{fig:SI_corr_normal}, we plot the estimated relative error against $-\log_{10}$ of the estimated probability. We observe that for probabilities as small as $10^{-10}$, the estimated relative error is below $1.6\%$ for $L=2$ and $3.5\%$ for $L=3$. In Appendix \ref{app:SI_corrnormal_more_sim}, we present addition simulation results for $d=100,L=2$ and $\rho \in \{0,0.05,0.1,0.2\}$.

\begin{figure}[!t]
    \centering
    \includegraphics[width=\textwidth]{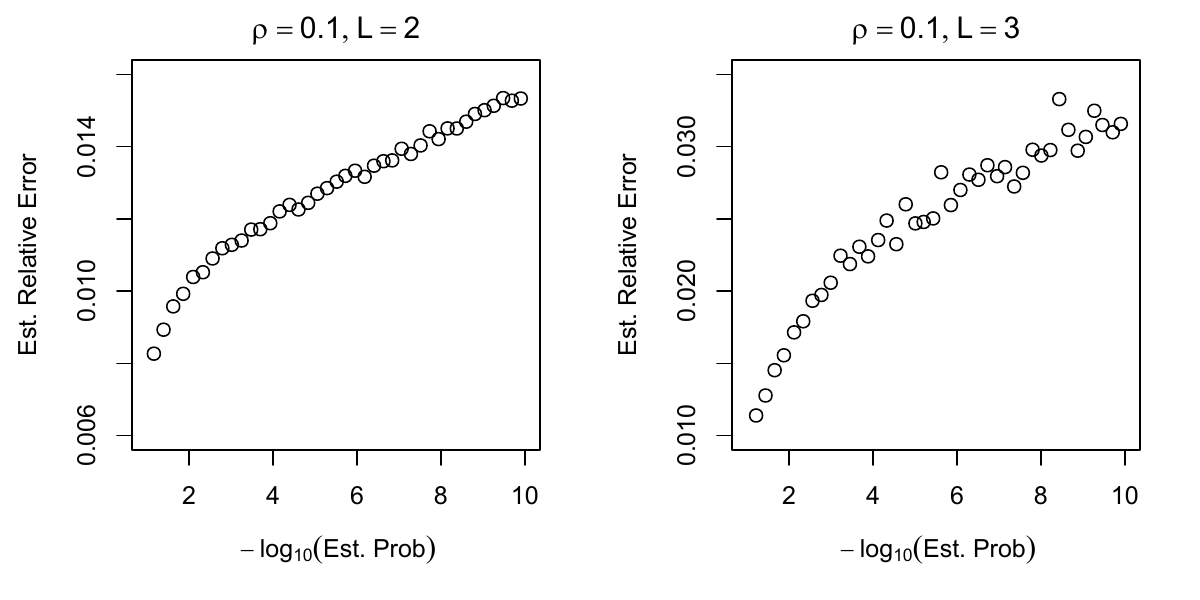}
    \caption{{Sum-intersection rule with $d=50$ and $L \in \{2,3\}$ under the setup of Example \ref{example:SI_corr}. The x-axis is $-\log_{10}$ of the estimated probability, and the y-axis is the estimated relative error. The results are based on $10^6$ repetitions.}}
    \label{fig:SI_corr_normal}
\end{figure}

\end{example}

\section{Conclusions}\label{sec:conclusion}
We develop a strategy for efficiently estimating rare-event probabilities involving many disjoint exit regions. By combining optimal tilts for a small number of regions with carefully selected additional components, the method remains practical even when the number of regions grows combinatorially or even exponentially with dimension $d$. Its effectiveness is demonstrated through several sequential multiple testing problems.

Future directions include extending the random walk model to Markov additive processes, as in \cite{collamore2002}, which may enable the estimation of error probabilities in sequential testing procedures with controlled sensing or composite hypotheses. Another direction is to develop computationally feasible proposals in settings where any asymptotically efficient mixture of exponential tilts provably requires too many components.

\bibliographystyle{imsart-nameyear}
\bibliography{sim}

\begin{thebibliography}{39}

\bibitem[\protect\citeauthoryear{Asmussen}{2003}]{alma9917764973405158}
\begin{bbook}[author]
\bauthor{\bsnm{Asmussen},~\bfnm{Søren.}\binits{S.}}
(\byear{2003}).
\btitle{Applied probability and queues},
\bedition{2nd ed.} ed.
\bseries{Applications of mathematics ; 51}.
\bpublisher{Springer}, \baddress{New York}.
\end{bbook}
\endbibitem

\bibitem[\protect\citeauthoryear{Asmussen and
  Glynn}{2007}]{asmussen2007stochastic}
\begin{bbook}[author]
\bauthor{\bsnm{Asmussen},~\bfnm{S{\o}ren}\binits{S.}} \AND
  \bauthor{\bsnm{Glynn},~\bfnm{Peter~W}\binits{P.~W.}}
(\byear{2007}).
\btitle{Stochastic simulation: algorithms and analysis}
\bvolume{57}.
\bpublisher{Springer}.
\end{bbook}
\endbibitem

\bibitem[\protect\citeauthoryear{Asmussen and
  Rubinstein}{1995}]{59428b60bd3211dc88d5000ea68e967b}
\begin{binbook}[author]
\bauthor{\bsnm{Asmussen},~\bfnm{S{\o}ren}\binits{S.}} \AND
  \bauthor{\bsnm{Rubinstein},~\bfnm{{R. Y. }}\binits{R.}}
(\byear{1995}).
\btitle{Steady state rare events simulation in queueing models and its
  complexity properties}.
In \bbooktitle{Advances in queueing theory, methods, and open problems}.
\bseries{Probability and stochastics series}
\bpages{429--466}.
\bpublisher{CRC Press}.
\end{binbook}
\endbibitem

\bibitem[\protect\citeauthoryear{Bertsekas, Nedic and
  Ozdaglar}{2003}]{bertsekas2003convex}
\begin{bbook}[author]
\bauthor{\bsnm{Bertsekas},~\bfnm{Dimitri}\binits{D.}},
  \bauthor{\bsnm{Nedic},~\bfnm{Angelia}\binits{A.}} \AND
  \bauthor{\bsnm{Ozdaglar},~\bfnm{Asuman}\binits{A.}}
(\byear{2003}).
\btitle{Convex analysis and optimization}
\bvolume{1}.
\bpublisher{Athena Scientific}.
\end{bbook}
\endbibitem

\bibitem[\protect\citeauthoryear{Blanchet, Glynn and
  Liu}{2007}]{blanchet2007fluid}
\begin{barticle}[author]
\bauthor{\bsnm{Blanchet},~\bfnm{Jos{\'e}}\binits{J.}},
  \bauthor{\bsnm{Glynn},~\bfnm{Peter}\binits{P.}} \AND
  \bauthor{\bsnm{Liu},~\bfnm{Jingchen~C}\binits{J.~C.}}
(\byear{2007}).
\btitle{Fluid heuristics, Lyapunov bounds and efficient importance sampling for
  a heavy-tailed G/G/1 queue}.
\bjournal{Queueing Systems}
\bvolume{57}
\bpages{99--113}.
\end{barticle}
\endbibitem

\bibitem[\protect\citeauthoryear{Blanchet and
  Glynn}{2008}]{blanchet2008efficient}
\begin{barticle}[author]
\bauthor{\bsnm{Blanchet},~\bfnm{Jose}\binits{J.}} \AND
  \bauthor{\bsnm{Glynn},~\bfnm{Peter}\binits{P.}}
(\byear{2008}).
\btitle{{Efficient rare-event simulation for the maximum of heavy-tailed random
  walks}}.
\bjournal{The Annals of Applied Probability}
\bvolume{18}
\bpages{1351 -- 1378}.
\bdoi{10.1214/07-AAP485}
\end{barticle}
\endbibitem

\bibitem[\protect\citeauthoryear{Blanchet and Lam}{2012}]{blanchet2012state}
\begin{barticle}[author]
\bauthor{\bsnm{Blanchet},~\bfnm{Jose}\binits{J.}} \AND
  \bauthor{\bsnm{Lam},~\bfnm{Henry}\binits{H.}}
(\byear{2012}).
\btitle{State-dependent importance sampling for rare-event simulation: An
  overview and recent advances}.
\bjournal{Surveys in Operations Research and Management Science}
\bvolume{17}
\bpages{38--59}.
\end{barticle}
\endbibitem

\bibitem[\protect\citeauthoryear{Blanchet and Liu}{2008}]{blanchet2008state}
\begin{barticle}[author]
\bauthor{\bsnm{Blanchet},~\bfnm{Jose~H}\binits{J.~H.}} \AND
  \bauthor{\bsnm{Liu},~\bfnm{Jingchen}\binits{J.}}
(\byear{2008}).
\btitle{State-dependent importance sampling for regularly varying random
  walks}.
\bjournal{Advances in Applied Probability}
\bvolume{40}
\bpages{1104--1128}.
\end{barticle}
\endbibitem

\bibitem[\protect\citeauthoryear{Blanchet and
  Liu}{2010}]{blanchet2010efficient}
\begin{barticle}[author]
\bauthor{\bsnm{Blanchet},~\bfnm{Jose}\binits{J.}} \AND
  \bauthor{\bsnm{Liu},~\bfnm{Jingchen}\binits{J.}}
(\byear{2010}).
\btitle{Efficient importance sampling in ruin problems for multidimensional
  regularly varying random walks}.
\bjournal{Journal of Applied Probability}
\bvolume{47}
\bpages{301--322}.
\end{barticle}
\endbibitem

\bibitem[\protect\citeauthoryear{Borovkov and
  Mogulskii}{2012}]{borovkov2012large}
\begin{barticle}[author]
\bauthor{\bsnm{Borovkov},~\bfnm{AA}\binits{A.}} \AND
  \bauthor{\bsnm{Mogulskii},~\bfnm{AA}\binits{A.}}
(\byear{2012}).
\btitle{Large deviation principles for random walk trajectories. I}.
\bjournal{Theory of Probability \& Its Applications}
\bvolume{56}
\bpages{538--561}.
\end{barticle}
\endbibitem

\bibitem[\protect\citeauthoryear{Brown}{1986}]{brown1986fundamentals}
\begin{bbook}[author]
\bauthor{\bsnm{Brown},~\bfnm{Lawrence~D.}\binits{L.~D.}}
(\byear{1986}).
\btitle{Fundamentals of Statistical Exponential Families: With Applications in
  Statistical Decision Theory}.
\bseries{Lecture Notes–Monograph Series, vol. 9}.
\bpublisher{Institute of Mathematical Statistics}, \baddress{Hayward, CA}.
\end{bbook}
\endbibitem

\bibitem[\protect\citeauthoryear{Bucklew}{2004}]{alma992931930105151}
\begin{bbook}[author]
\bauthor{\bsnm{Bucklew},~\bfnm{James~A.}\binits{J.~A.}}
(\byear{2004}).
\btitle{Introduction to rare event simulation}.
\bseries{Springer series in statistics}.
\bpublisher{Springer}, \baddress{New York}.
\end{bbook}
\endbibitem

\bibitem[\protect\citeauthoryear{Bucklew, Ney and
  Sadowsky}{1990}]{bucklew1990monte}
\begin{barticle}[author]
\bauthor{\bsnm{Bucklew},~\bfnm{James~A}\binits{J.~A.}},
  \bauthor{\bsnm{Ney},~\bfnm{Peter}\binits{P.}} \AND
  \bauthor{\bsnm{Sadowsky},~\bfnm{John~S}\binits{J.~S.}}
(\byear{1990}).
\btitle{Monte Carlo simulation and large deviations theory for uniformly
  recurrent Markov chains}.
\bjournal{Journal of Applied Probability}
\bvolume{27}
\bpages{44--59}.
\end{barticle}
\endbibitem

\bibitem[\protect\citeauthoryear{Collamore}{1996a}]{collamore1996}
\begin{barticle}[author]
\bauthor{\bsnm{Collamore},~\bfnm{Jeffrey~F.}\binits{J.~F.}}
(\byear{1996}a).
\btitle{Hitting probabilities and large deviations}.
\bjournal{Ann. Probab.}
\bvolume{24}
\bpages{2065--2078}.
\bdoi{10.1214/aop/1041903218}
\end{barticle}
\endbibitem

\bibitem[\protect\citeauthoryear{Collamore}{1996b}]{collamore_thesis}
\begin{bphdthesis}[author]
\bauthor{\bsnm{Collamore},~\bfnm{Jeffrey~F.}\binits{J.~F.}}
(\byear{1996}b).
\btitle{Large deviation techniques for the study of the hitting probabilities
  of rare sets.},
\btype{PhD thesis}.
\end{bphdthesis}
\endbibitem

\bibitem[\protect\citeauthoryear{Collamore}{2002}]{collamore2002}
\begin{barticle}[author]
\bauthor{\bsnm{Collamore},~\bfnm{J.~F.}\binits{J.~F.}}
(\byear{2002}).
\btitle{Importance Sampling Techniques for the Multidimensional Ruin Problem
  for General Markov Additive Sequences of Random Vectors}.
\bjournal{Ann. Appl. Probab.}
\bvolume{12}
\bpages{382--421}.
\bdoi{10.1214/aoap/1015961169}
\end{barticle}
\endbibitem

\bibitem[\protect\citeauthoryear{De and Baron}{2012a}]{de2012sequential}
\begin{barticle}[author]
\bauthor{\bsnm{De},~\bfnm{Shyamal~K}\binits{S.~K.}} \AND
  \bauthor{\bsnm{Baron},~\bfnm{Michael}\binits{M.}}
(\byear{2012}a).
\btitle{Sequential Bonferroni methods for multiple hypothesis testing with
  strong control of family-wise error rates I and II}.
\bjournal{Sequential Analysis}
\bvolume{31}
\bpages{238--262}.
\end{barticle}
\endbibitem

\bibitem[\protect\citeauthoryear{De and Baron}{2012b}]{de2012step}
\begin{barticle}[author]
\bauthor{\bsnm{De},~\bfnm{Shyamal~K}\binits{S.~K.}} \AND
  \bauthor{\bsnm{Baron},~\bfnm{Michael}\binits{M.}}
(\byear{2012}b).
\btitle{Step-up and step-down methods for testing multiple hypotheses in
  sequential experiments}.
\bjournal{Journal of Statistical Planning and Inference}
\bvolume{142}
\bpages{2059--2070}.
\end{barticle}
\endbibitem

\bibitem[\protect\citeauthoryear{Dembo and Zeitouni}{2009}]{dembo2009large}
\begin{bbook}[author]
\bauthor{\bsnm{Dembo},~\bfnm{Amir}\binits{A.}} \AND
  \bauthor{\bsnm{Zeitouni},~\bfnm{Ofer}\binits{O.}}
(\byear{2009}).
\btitle{Large deviations techniques and applications}
\bvolume{38}.
\bpublisher{Springer Science \& Business Media}.
\end{bbook}
\endbibitem

\bibitem[\protect\citeauthoryear{Dupuis, Sezer and
  Wang}{2007}]{DupuisSezerWang2007}
\begin{barticle}[author]
\bauthor{\bsnm{Dupuis},~\bfnm{Paul}\binits{P.}},
  \bauthor{\bsnm{Sezer},~\bfnm{Ali~Devin}\binits{A.~D.}} \AND
  \bauthor{\bsnm{Wang},~\bfnm{Hui}\binits{H.}}
(\byear{2007}).
\btitle{Dynamic Importance Sampling for Queueing Networks}.
\bjournal{The Annals of Applied Probability}
\bvolume{17}
\bpages{1306--1346}.
\end{barticle}
\endbibitem

\bibitem[\protect\citeauthoryear{Dupuis and Wang}{2007}]{DupuisWang2007}
\begin{barticle}[author]
\bauthor{\bsnm{Dupuis},~\bfnm{Paul}\binits{P.}} \AND
  \bauthor{\bsnm{Wang},~\bfnm{Hui}\binits{H.}}
(\byear{2007}).
\btitle{Subsolutions of an Isaacs Equation and Efficient Schemes for Importance
  Sampling}.
\bjournal{Mathematics of Operations Research}
\bvolume{32}
\bpages{723--757}.
\end{barticle}
\endbibitem

\bibitem[\protect\citeauthoryear{Glasserman}{2004}]{glasserman2004monte}
\begin{bbook}[author]
\bauthor{\bsnm{Glasserman},~\bfnm{Paul}\binits{P.}}
(\byear{2004}).
\btitle{Monte Carlo methods in financial engineering}
\bvolume{53}.
\bpublisher{Springer}.
\end{bbook}
\endbibitem

\bibitem[\protect\citeauthoryear{Glasserman and
  Kou}{1995}]{glasserman1995analysis}
\begin{barticle}[author]
\bauthor{\bsnm{Glasserman},~\bfnm{Paul}\binits{P.}} \AND
  \bauthor{\bsnm{Kou},~\bfnm{Shing-Gang}\binits{S.-G.}}
(\byear{1995}).
\btitle{Analysis of an importance sampling estimator for tandem queues}.
\bjournal{ACM Transactions on Modeling and Computer Simulation (TOMACS)}
\bvolume{5}
\bpages{22--42}.
\end{barticle}
\endbibitem

\bibitem[\protect\citeauthoryear{Glasserman and Wang}{1997}]{glasserman1997}
\begin{barticle}[author]
\bauthor{\bsnm{Glasserman},~\bfnm{Paul}\binits{P.}} \AND
  \bauthor{\bsnm{Wang},~\bfnm{Yashan}\binits{Y.}}
(\byear{1997}).
\btitle{Counterexamples in importance sampling for large deviations
  probabilities}.
\bjournal{Ann. Appl. Probab.}
\bvolume{7}
\bpages{731--746}.
\bdoi{10.1214/aoap/1034801251}
\end{barticle}
\endbibitem

\bibitem[\protect\citeauthoryear{He and Bartroff}{2021}]{he2021asymptotically}
\begin{barticle}[author]
\bauthor{\bsnm{He},~\bfnm{Xinrui}\binits{X.}} \AND
  \bauthor{\bsnm{Bartroff},~\bfnm{Jay}\binits{J.}}
(\byear{2021}).
\btitle{Asymptotically optimal sequential FDR and pFDR control with (or
  without) prior information on the number of signals}.
\bjournal{Journal of Statistical Planning and Inference}
\bvolume{210}
\bpages{87--99}.
\end{barticle}
\endbibitem

\bibitem[\protect\citeauthoryear{Hult and Lindskog}{2006}]{hult2006heavy}
\begin{btechreport}[author]
\bauthor{\bsnm{Hult},~\bfnm{Henrik}\binits{H.}} \AND
  \bauthor{\bsnm{Lindskog},~\bfnm{Filip}\binits{F.}}
(\byear{2006}).
\btitle{Heavy-tailed insurance portfolios: buffer capital and ruin
  probabilities}
\btype{Technical Report},
\bpublisher{Cornell University Operations Research and Industrial Engineering}.
\end{btechreport}
\endbibitem

\bibitem[\protect\citeauthoryear{Juneja and Shahabuddin}{2006}]{juneja2006rare}
\begin{barticle}[author]
\bauthor{\bsnm{Juneja},~\bfnm{Sandeep}\binits{S.}} \AND
  \bauthor{\bsnm{Shahabuddin},~\bfnm{Perwez}\binits{P.}}
(\byear{2006}).
\btitle{Rare-event simulation techniques: An introduction and recent advances}.
\bjournal{Handbooks in operations research and management science}
\bvolume{13}
\bpages{291--350}.
\end{barticle}
\endbibitem

\bibitem[\protect\citeauthoryear{Lehtonen and
  Nyrhinen}{1992}]{lehtonen1992asymptotically}
\begin{barticle}[author]
\bauthor{\bsnm{Lehtonen},~\bfnm{Tapani}\binits{T.}} \AND
  \bauthor{\bsnm{Nyrhinen},~\bfnm{Harri}\binits{H.}}
(\byear{1992}).
\btitle{On asymptotically efficient simulation of ruin probabilities in a
  Markovian environment}.
\bjournal{Scandinavian Actuarial Journal}
\bvolume{1992}
\bpages{60--75}.
\end{barticle}
\endbibitem

\bibitem[\protect\citeauthoryear{Ney and Nummelin}{1987a}]{ney1987markovI}
\begin{barticle}[author]
\bauthor{\bsnm{Ney},~\bfnm{Peter}\binits{P.}} \AND
  \bauthor{\bsnm{Nummelin},~\bfnm{Esa}\binits{E.}}
(\byear{1987}a).
\btitle{Markov additive processes I. Eigenvalue properties and limit theorems}.
\bjournal{The Annals of Probability}
\bpages{561--592}.
\end{barticle}
\endbibitem

\bibitem[\protect\citeauthoryear{Ney and Nummelin}{1987b}]{ney1987markovII}
\begin{barticle}[author]
\bauthor{\bsnm{Ney},~\bfnm{Peter}\binits{P.}} \AND
  \bauthor{\bsnm{Nummelin},~\bfnm{Esa}\binits{E.}}
(\byear{1987}b).
\btitle{Markov additive processes II. Large deviations}.
\bjournal{The Annals of Probability}
\bpages{593--609}.
\end{barticle}
\endbibitem

\bibitem[\protect\citeauthoryear{Rudin}{1976}]{rudin1976principles}
\begin{bbook}[author]
\bauthor{\bsnm{Rudin},~\bfnm{Walter}\binits{W.}}
(\byear{1976}).
\btitle{Principles of Mathematical Analysis},
\bedition{3rd} ed.
\bpublisher{McGraw-Hill}, \baddress{New York}.
\end{bbook}
\endbibitem

\bibitem[\protect\citeauthoryear{Sadowsky}{1996}]{sadowsky1996monte}
\begin{barticle}[author]
\bauthor{\bsnm{Sadowsky},~\bfnm{John~S}\binits{J.~S.}}
(\byear{1996}).
\btitle{On Monte Carlo estimation of large deviations probabilities}.
\bjournal{The Annals of Applied Probability}
\bvolume{6}
\bpages{399--422}.
\end{barticle}
\endbibitem

\bibitem[\protect\citeauthoryear{Sadowsky and
  Bucklew}{1990}]{sadowsky1990large}
\begin{barticle}[author]
\bauthor{\bsnm{Sadowsky},~\bfnm{John~S}\binits{J.~S.}} \AND
  \bauthor{\bsnm{Bucklew},~\bfnm{James~A}\binits{J.~A.}}
(\byear{1990}).
\btitle{On large deviations theory and asymptotically efficient Monte Carlo
  estimation}.
\bjournal{IEEE transactions on Information Theory}
\bvolume{36}
\bpages{579--588}.
\end{barticle}
\endbibitem

\bibitem[\protect\citeauthoryear{Siegmund}{1976}]{siegmund1976importance}
\begin{barticle}[author]
\bauthor{\bsnm{Siegmund},~\bfnm{David}\binits{D.}}
(\byear{1976}).
\btitle{Importance sampling in the Monte Carlo study of sequential tests}.
\bjournal{The Annals of Statistics}
\bpages{673--684}.
\end{barticle}
\endbibitem

\bibitem[\protect\citeauthoryear{Sion}{1958}]{sion1958minimax}
\begin{barticle}[author]
\bauthor{\bsnm{Sion},~\bfnm{Maurice}\binits{M.}}
(\byear{1958}).
\btitle{On general minimax theorems}.
\bjournal{Pacific Journal of Mathematics}
\bvolume{8}
\bpages{171--176}.
\bdoi{10.2140/pjm.1958.8.171}
\end{barticle}
\endbibitem

\bibitem[\protect\citeauthoryear{Song and
  Fellouris}{2016}]{song2016logarithmically}
\begin{binproceedings}[author]
\bauthor{\bsnm{Song},~\bfnm{Yanglei}\binits{Y.}} \AND
  \bauthor{\bsnm{Fellouris},~\bfnm{Georgios}\binits{G.}}
(\byear{2016}).
\btitle{Logarithmically efficient simulation for misclassification
  probabilities in sequential multiple testing}.
In \bbooktitle{2016 Winter Simulation Conference (WSC)}
\bpages{314--325}.
\bpublisher{IEEE}.
\end{binproceedings}
\endbibitem

\bibitem[\protect\citeauthoryear{Song and
  Fellouris}{2017}]{song2017asymptotically}
\begin{barticle}[author]
\bauthor{\bsnm{Song},~\bfnm{Yanglei}\binits{Y.}} \AND
  \bauthor{\bsnm{Fellouris},~\bfnm{Georgios}\binits{G.}}
(\byear{2017}).
\btitle{Asymptotically optimal, sequential, multiple testing procedures with
  prior information on the number of signals}.
\end{barticle}
\endbibitem

\bibitem[\protect\citeauthoryear{Song and Fellouris}{2019}]{song2019sequential}
\begin{barticle}[author]
\bauthor{\bsnm{Song},~\bfnm{Yanglei}\binits{Y.}} \AND
  \bauthor{\bsnm{Fellouris},~\bfnm{Georgios}\binits{G.}}
(\byear{2019}).
\btitle{Sequential multiple testing with generalized error control: An
  asymptotic optimality theory}.
\end{barticle}
\endbibitem

\bibitem[\protect\citeauthoryear{Wald}{1947}]{alma993552393405158}
\begin{bbook}[author]
\bauthor{\bsnm{Wald},~\bfnm{Abraham}\binits{A.}}
(\byear{1947}).
\btitle{Sequential analysis}.
\bseries{Wiley mathematical statistics series}.
\end{bbook}
\endbibitem

\end{thebibliography}

\begin{appendix}

\section{Proofs and Discussion for Section \ref{sec:preliminary}}
\label{app:collamore_results}

\subsection{A weaker version of Theorem \ref{thm:Collomore_prob}} \label{app:discuss_collomore_prob}

In this appendix, we present a self-contained, elementary proof of a weaker version of Theorem \ref{thm:Collomore_prob}, under the stronger assumption \eqref{assumption:stronger}.


\begin{remark}
    \label{aux:remark_semi}
Since $\Lambda$ is lower semi-continuous \cite[Theorem 1.13]{brown1986fundamentals}, 
we have $\mathcal{L}_0\Lambda$ is closed, and thus condition \eqref{assumption:stronger} implies that $\mathcal{L}_0\Lambda$ is compact.
\end{remark}

\begin{theorem}[Weaker version of Theorem \ref{thm:Collomore_prob}]\label{app:thrm_weaker_prob}
    Let $W \subset \mathbb{R}^d$ be a convex open set. 
 Assume  that \eqref{assumption:stronger} holds, and that 
$\operatorname{cone}(B(\Exp[X_1], \delta)) \cap W$ is empty for some $\delta > 0$.  Then  
\begin{align*}
\lim_{b \to \infty}    \frac{1}{b} \log \Pro(\tau_b(W) < \infty)  =- r_W.
\end{align*}
\end{theorem}
\begin{proof}
\noindent \textbf{Upper bound.} For each $\theta \in \mathcal{L}_0\Lambda$, by a change of measure, we have
\begin{align*}
 \Pro(\tau_b(W) < \infty) &= \Exp_{\theta}\left[\exp(- \theta \cdot S_{\tau_b(W)} + \tau_b(W) \Lambda(\theta) ); \tau_b(W) < \infty \right]  \\
& \leq \Exp_{\theta}\left[\exp(- \theta \cdot S_{\tau_b(W)}  ); \tau_b(W) < \infty \right] \\
&\leq \exp(- \inf_{x \in W}  \theta \cdot (bx)  ),
\end{align*}
where the first inequality is because $\Lambda(\theta) \leq 0$, and the second is because on the event $\{\tau_b(W) < \infty\}$, we have $S_{\tau_b(W)} \in bW$.  
Optimizing the upper bound over $\theta \in \mathcal{L}_0\Lambda$,   we have
\begin{align*}
   \frac{1}{b} \log \Pro(\tau_b(W) < \infty) \leq - \sup_{\theta \in \mathcal{L}_0\Lambda} (\inf_{x \in W} \theta \cdot x).
\end{align*}
Since $\mathcal{L}_0 \Lambda$ is compact due to condition \eqref{assumption:stronger} (see Remark \ref{aux:remark_semi}),  by Sion's minimax theorem \citep{sion1958minimax}  we have
\begin{align*}
    r_W = \inf_{x \in W} \sup_{\theta \in \mathcal{L}_0\Lambda} \theta \cdot x =  \sup_{\theta \in \mathcal{L}_0\Lambda} \inf_{x \in W} \theta \cdot x,
\end{align*}
which establishes the upper bound: $b^{-1} \log \Pro(\tau_b(W) < \infty) \leq -r_W$.\\

\noindent \textbf{Lower bound.} Fix an arbitrary $x \in W$. We recall from \eqref{def:rate_function} that:
\begin{align*}
  I(x)=  \sup_{\theta \in \mathcal{L}_0\Lambda} \theta \cdot x.
\end{align*}
Since  $\mathcal{L}_0\Lambda$ is compact due to condition \eqref{assumption:stronger} (Remark \ref{aux:remark_semi}),
this optimization problem 
must have a solution, denoted as $\theta_x \in \mathcal{L}_0\Lambda$.

Due to condition \eqref{assumption:stronger}, $\operatorname{dom}(\Lambda)$ is open, and thus $\nabla \Lambda(\mathbf{0}) = \Exp[X_1]$ (see, e.g., Theorem 2.2 in \cite{brown1986fundamentals}). Since $W \cap \operatorname{cone}(B(\Exp[X_1], \delta)) = \emptyset$ for some $\delta > 0$, we have that $\Exp[X_1] \neq \mathbf{0}$, and that $\mathbf{0} \not \in W$, which also implies that $x \neq \mathbf{0}$ since $x \in W$. As a result, there exists some $\tilde{\theta} \in \bR^{d}$ such that $\Lambda(\tilde{\theta}) < 0$, which  implies that Slater’s Condition \cite[Asumption 6.4.3]{bertsekas2003convex} holds, and thus the KKT condition is necessary for $\theta_x$. Since $x \neq \mathbf{0}$, we must have
\begin{align} \label{aux:for_remark}
    I(x) = \theta_x \cdot x, \quad \Lambda(\theta_x) = 0, \quad \text{ and } \quad \rho  \nabla  \Lambda(\theta_x) =  x  \text{ for some } \rho > 0.
\end{align}

Since $W$ is open, for any sufficiently small $\epsilon > 0$, $B(x,\epsilon) \subset W$.  Fix such an $\epsilon$.
By a change of measure from $\Pro$ to $\Pro_{\theta_x}$,  we have
\begin{align*}
 \Pro(\tau_b(W) < \infty)& \geq \Pro\left(\tau_b(B(x,\epsilon) < \infty \right) \\
 &= \Exp_{\theta_x}\left[\exp\left(- \theta_x \cdot S_{\tau_b(B(x,\epsilon))} + \tau_b(B(x,\epsilon)) \Lambda(\theta_x)\right);\; \tau_b(B(x,\epsilon))< \infty \right]  \\
& \geq   - b (\theta_x \cdot x + \epsilon \|\theta_x\|) \; \Pro_{\theta_x}\left(\tau_b(B(x,\epsilon)) < \infty \right),
\end{align*}
where the final inequality holds  because $\Lambda(\theta_x) = 0$ and  
$S_{\tau_b(B(x,\epsilon))} \in b B(x,\epsilon)$ on the event 
$\{\tau_b(B(x,\epsilon)) < \infty\}$.  Since $\operatorname{dom}(\Lambda)$ is open,
$$\Exp_{\theta_x}[X_1] = \nabla \Lambda(\theta_x) = \frac{x}{\rho}.
$$
Due to the Strong Law of Large Numbers, $\lim_{b \to \infty} \Pro_{\theta_x}\left(\tau_b(B(x,\epsilon)) < \infty \right) = 1$, and thus
\begin{align*}
      \liminf_{b \to \infty} \frac{1}{b} \log \Pro(\tau_b(W) < \infty) \geq 
       - I(x) - \epsilon \|\theta_x\|.
\end{align*}
Since $\epsilon >0$ can be arbitrarily small,  optimizing the lower bound over $x \in W$ we have
\begin{align*}
     \liminf_{b \to \infty} \frac{1}{b} \log \Pro(\tau_b(W) < \infty) \geq 
       - \inf_{x \in W} I(x)  = - r_W.
\end{align*}
The proof is complete.
\end{proof}

\begin{remark}
    \label{aux:remark_stronger}
    Due to \eqref{aux:for_remark}, by  the same argument as for Lemma 2.2.31(b) in \cite{dembo2009large}, for $x \in \mathcal{V}$,
$$
\Lambda^{*}(x/\rho) = \theta_x \cdot (x/\rho) - \Lambda(\theta_x),
$$  
which implies that $\Lambda^{*}(x/\rho) < \infty$. Thus, $W \subset \operatorname{cone}(\operatorname{dom}(\Lambda^*))$. Since $W$ is open, condition \eqref{assumption:stronger} implies that $W \cap \operatorname{ri}(\operatorname{cone}(\operatorname{dom}(\Lambda^*))) \neq \emptyset$. 
\end{remark}

\subsection{Discussions for Theorem \ref{thrm:Collamore_simulation}}
\label{app:collomore_simulation}

We first provide more details about the proof of Theorem \ref{thrm:Collamore_simulation} based on \cite{collamore2002}.

\begin{proof}[Proof of Theorem \ref{thrm:Collamore_simulation}]
The proof is due to \cite{collamore2002}; we provide further details below. In \cite{collamore2002}, the process $\{(X_n, S_n) : n \geq 0\}$ is modeled as a Markov additive process, where $S_n := \sum_{i=1}^{n} \xi_i$ and $\{(X_n, \xi_n) : n \geq 0\}$ is a time-homogeneous Markov process taking values in $\mathbb{S} \times \mathbb{R}^d$, with transition kernel
\begin{align*}
\mathcal{P}(x, E \times \Gamma) := \Pro\left((X_{n+1}, \xi_{n+1}) \in E \times \Gamma \mid X_n = x\right), \quad x \in \mathbb{S},; E \in \mathcal{S},; \Gamma \in \mathcal{R}^d,
\end{align*}
where $\mathcal{S}$ is a $\sigma$-algebra on $\mathbb{S}$, and $\mathcal{R}^d$ is the Borel $\sigma$-algebra on $\mathbb{R}^d$. In the special case of random walks considered in our setting, the state space $\mathbb{S}$ can be taken as a singleton. Hence, the dependence on ${X_n}$ and the structure of $\mathbb{S}$ in \cite{collamore2002} can be disregarded.

For $\theta \in \operatorname{dom}(\Lambda)$, we identify the proposal $\Qro$ in  \cite{collamore2002} with $\Pro_{\theta}$, and then $\mathcal{K}_{\Qro}$ in Equation (2.4) of \cite{collamore2002} is a measure on $\bR^d$ such that for  $v \in \bR^d$,
$$
\frac{d\mathcal{K}_{\Qro}}{d \mu}(v) = 
\frac{d\mu}{d \mu_{\theta}}(v) =
 \exp\left(- \theta \cdot v + \Lambda(\theta) \right).
$$
Further, $\Lambda_{\mathcal{K}_{\Qro}}$ defined in Section 2.2 of \cite{collamore2002}  is defined for $\alpha \in \mathbb{R}^d$ by
\begin{align*}
    \Lambda_{\mathcal{K}_{\Qro}}(\alpha) = &\log\left(
\int_{\bR^d} \exp(\alpha \cdot v) d\mathcal{K}_{\Qro} (v)
\right) \\
=& \log\left(
\int_{\bR^d} \exp(\alpha \cdot v)  \exp\left(- \theta \cdot v  + \Lambda(\theta) \right) d\mu(v)
\right)  =  \Lambda(\theta) + \Lambda(\alpha - \theta).
\end{align*}
As a result, the rate function $I_{\mathcal{K}{\Qro}}$ defined in Section 3.1 of \cite{collamore2002} coincides with $\mathcal{V}_{\theta}(\cdot)$ as given in \eqref{def:rate_theta}.

As discussed around equation (3.6) and in Section 3.2 of \cite{collamore2002}, the assumptions $\Re$, (H1), (H2), and (H3) are satisfied in the case of random walks under the conditions stated in Theorem \ref{thrm:Collamore_simulation}. Thus, parts (i) and (ii) of Theorem \ref{thrm:Collamore_simulation} follow  from Lemma 3.2 and Theorem 3.3 of \cite{collamore2002}, respectively, while part (iii) follows from Theorem 3.1 and the subsequent discussion.
\end{proof}

We next establish a \textit{weaker} version of part~(iii) of Theorem \ref{thrm:Collamore_simulation}. This result offers insight into the function $\mathcal{V}_{\theta}(\cdot)$ defined in~\eqref{def:rate_theta} and the two-argument function $v(\cdot, \cdot)$ defined in~\eqref{def:V_W_theta}.

\begin{theorem}[Weaker version of Part (iii) of Theorem \ref{thrm:Collamore_simulation}]\label{app:thrm_weaker_simulation}
Let $W \subset \mathbb{R}^d$ be a convex open set. 
 Assume  that \eqref{assumption:stronger} holds.
  For $\theta \in \operatorname{dom}(\Lambda)$, if  $v(W;\theta) > 0$, then
\begin{align*}
    \frac{1}{b} \log \Exp_{\theta}\left[ 
\mathbb{L}_{\theta}^{-2}(\tau_b(W))
\right] \leq   - v(W;\theta).
\end{align*}
\end{theorem}
\begin{proof}
Fix arbitrary $\theta \in \operatorname{dom}(\Lambda)$ such that  $v(W;\theta) > 0$. 
Since $v(W;\theta) > 0$, it implies that the following subset is non-empty:
$$
\widetilde{D} := \{\alpha \in \bR^d: \Lambda(\theta) + \Lambda(\alpha - \theta) \leq 0\}.
$$
For each $\alpha \in \widetilde{D}$,
by a change of measure from $\Pro_{\theta}$ to $\Pro_{ \alpha - \theta}$, we have
\begin{align*}
    \Exp_{\theta}\left[ 
\mathbb{L}_{\theta}^{-2}(\tau_b(W))
\right] &=  \Exp_{\alpha -\theta }\left[ 
\exp\left(-\alpha  \cdot S_{\tau_b(W)} + \tau_b(W)(\Lambda(\theta) +\Lambda(\alpha - \theta))\right)
\mathbbm{1}\{\tau_b < \infty\} 
\right] \\
&\leq \Exp_{\alpha -\theta }\left[ 
\exp\left(-\alpha  \cdot S_{\tau_b(W)}  \right)
\mathbbm{1}\{\tau_b(W) < \infty\} 
\right] \\
&\leq \exp(-b \inf_{x \in W} (\alpha \cdot x)),
\end{align*}
where, in the final step, we use the fact that on the event $\{\tau_b(W) < \infty\}$, we have $S_{\tau_b(W)} \in bW$. Optimizing the upper bound over $\alpha \in \widetilde{D}$, we have
\begin{align*}
    \frac{1}{b}\log\Exp_{\theta}\left[ 
\mathbb{L}_{\theta}^{-2}(\tau_b(W))
\right]  \leq - \sup_{\alpha \in \tilde{D}} \left(\inf_{x \in W} (\alpha \cdot x)\right).
\end{align*}
Finally, by the convexity of $\Lambda(\cdot)$,
$$
\widetilde{D} \subset \{\alpha \in \bR^{d}: \Lambda(\alpha/2) \leq 0\}.
$$
Since $\{\tilde{\theta} \in \bR^{d}: \Lambda(\tilde{\theta}) \leq 0\}$ is bounded, 
$\widetilde{D}$ is bounded. Since $\Lambda$ is lower-semicontinuous (see Remark \ref{aux:remark_semi}), $\widetilde{D}$ is compact.  Then, by Sion's minimax theorem \citep{sion1958minimax}, we have
\begin{align*}
    v(W;\theta) = \inf_{x \in \overline{W}} \sup_{\alpha \in \widetilde{D}} \alpha \cdot x =   \sup_{\alpha \in \widetilde{D}} \inf_{x \in \overline{W}} \alpha \cdot x.
\end{align*}
The proof is then complete since $\inf_{x \in W} (\alpha \cdot x) = \inf_{x \in \overline{W}} (\alpha \cdot x)$.
\end{proof}

\section{Proofs regarding the exponential decay of the wrong exit probability}

\subsection{Supporting lemmas}
The following lemma is proved in Lemma 3.2 of \cite{collamore2002}. Recall the definition of the rate function $I(\cdot)$ in \eqref{def:rate_function} and $r_j$ in \eqref{def:rj_rstar}.

\begin{lemma}\label{app_lemma:upper_bound_var_single_j}
Fix $j \in [J]$.  Assume conditions \eqref{assumptions_directions_of_mean},  \eqref{assumption_on_X} and \eqref{assumption_on_regions} hold. 
\begin{enumerate}[label=(\roman*)]
    \item 
There exists a unique $\beta^j \in \bR^d$ such that $\Lambda(\beta^j) = 0$,
\begin{align*}
    W^j \subset \{v\in \bR^{d}: \beta^j \cdot v \geq r_j\},\;\; \text{ and } \;\; \mathcal{L}_{r_j} I \subset \{v \in \bR^{d}:\; \beta^j \cdot v \leq r_j\}.
\end{align*}

\item There exists a unique $x^j \in \partial W^{j}$ such that $I(x^j) = r_j$.

\item For $\beta^j$ and $x^j$ in parts (i) and (ii), we have
\begin{align*}
    \inf_{x \in \overline{W^j}} \sup_{\theta \in \mathcal{L}_0 \Lambda} \theta \cdot x = r_j = \beta^j \cdot x^j = \sup_{\theta \in \mathcal{L}_0 \Lambda} \inf_{x \in \overline{W^j}} \theta \cdot x,
\end{align*}
and for some $\rho > 0$, $x^j = \rho \nabla \Lambda(\beta^j)$.
\end{enumerate}
\end{lemma}

\begin{proof}
For part (i), the existence is proved in part (i) of Lemma 3.2 in \cite{collamore2002}. Further, due to \eqref{assumption_on_X}, $\operatorname{int}\left(\operatorname{cone}\left( \operatorname{dom}(\Lambda^*)\right)\right)$ is non-emprty. Then the uniqueness follows from part (iv) of Lemma 3.2 in \cite{collamore2002}.

Parts (ii) and the majority of part (iii) follow from Lemma 3.2 in \cite{collamore2002}, with the exception of the identity
$\sup_{\theta \in \mathcal{L}_0 \Lambda} \inf_{x \in \overline{W^j}} \theta \cdot x = r_j$, 
which we now justify. By the max-min inequality,
\begin{align*}
    \sup_{\theta \in \mathcal{L}_0 \Lambda} \inf_{x \in \overline{W^j}} \theta \cdot x
\leq  \inf_{x \in \overline{W^j}} \sup_{\theta \in \mathcal{L}_0 \Lambda} \theta \cdot x  = r_j.
\end{align*}
Further, since $\beta^j \in \mathcal{L}_0 \Lambda$ by part (i), we have
\begin{align*}
  \sup_{\theta \in \mathcal{L}_0 \Lambda} \inf_{x \in \overline{W^j}} \theta \cdot x
\geq   \inf_{x \in \overline{W^j}} \beta^j \cdot x \geq r_j.
\end{align*}
where the final inequality is also due to part (i). Then, the proof is complete.
\end{proof}

The following lemma is crucial in the proof of Theorem \ref{thm:prob_rate}. Recall that $\textbf{0}$ is the origin in $\bR^d$.

\begin{lemma}\label{lemma:existence_epsilon_lower_bound}
Let $j \in [J]$ and $v \in W^j$. Assume that 
\eqref{assumptions_directions_of_mean} and   \eqref{assumption_on_regions} hold.  
Then there exists $\epsilon > 0$ such that $ B(v, \epsilon) \subset W^j$ and for $j' = 0,1,\ldots,J$ with $j' \neq j$,
\begin{align}
\label{aux_tmp}
B(tv, \epsilon)\cap W^{j'} = \emptyset\; \text{ for } \;0 \leq t \leq 1.
\end{align}
\end{lemma}
\begin{proof}
Since $J$ is finite, it suffices to fix some $j' \neq j$ and show that there exists $\epsilon > 0$ such that  \eqref{aux_tmp} holds.

Since $W^j$ is an open subset, there exists some $\tilde{\epsilon} >0$ such that $B(v,\tilde{\epsilon}) \subset W^j$. Define
\begin{align*}
    c  :=  \delta/ (2\|v\|), \quad  \epsilon :=  \min\left\{c \tilde{\epsilon}, \delta/2\right\}.
\end{align*}
Due to assumption \eqref{assumptions_directions_of_mean}, $\|v\| > \delta$, and thus $c \in (0,1/2)$. Further, by definition, $\epsilon < \tilde{\epsilon}$, which implies that $B(v,\epsilon) \subset W^j$.
  
Now, we show \eqref{aux_tmp} holds. For $c \leq t \leq 1$, note that by definition, $\epsilon \leq t \tilde{\epsilon}$, and hence
$$ 
B(tv, \epsilon) \subset  B(tv, t \tilde{\epsilon}) = t B(v, \tilde{\epsilon}) \subset \operatorname{cone}(W^j).
$$
Since $\operatorname{cone}(W^j)$ and $\operatorname{cone}(W^{j'})$ are disjoint by \eqref{assumption_on_regions}, we have 
$B(tv, \epsilon) \cap  W^{j'}  = \emptyset$ for $c \leq t \leq 1$.

For $0 \leq t < c$, since by definition
$\|tv\| + \epsilon < c \|v\| + \delta/2 = \delta/2 + \delta/2 = \delta$, 
we have
$$
B(tv, \epsilon)  \subset {B(\textbf{0},\delta)}.
$$
Since by assumption \eqref{assumptions_directions_of_mean} $ {B(\textbf{0},\delta)}\cap W^{j'} =\emptyset$, we have
 $B(tv, \epsilon) \cap W^{j'} = \emptyset$. The proof is complete.
\end{proof}

The following lemma is from  Lemma 1.2.15  in \cite{dembo2009large}, which we present for completeness.

\begin{lemma}\label{lemma:log_max_sum}
Let $J \geq 1$ be a fixed integer, and $\{z_b^j: b>0, j \in [J] \}$ be a collection of positive numbers. Then
\begin{align*}
    \limsup_{b \to \infty} \frac{1}{b} \log\left(\sum_{j \in [J]} z_b^j \right) = 
       \max_{j \in [J]} \limsup_{b \to \infty} \frac{1}{b} \log\left(z_b^j\right).
\end{align*}
The same holds if $\limsup$ is replaced by $\liminf$.
\end{lemma}

\subsection{Proof of Theorem \ref{thm:prob_rate}}
\label{app:lower_bound_exit}
Recall that for each $j \in [J]$, $\tau_b^j$ is the first time that $\{S_n\}$ hits $b W^j$, and that $\Lambda^*(\cdot)$, $I(\cdot)$ and $r_j, r_*$ are defined in Section \ref{sec:prob_wrong_exit}.

\begin{proof}[Proof of Theorem \ref{thm:prob_rate}]
Due to condition \eqref{assumptions_directions_of_mean}, we have that
\begin{align*}
       \Pro\left(\tau_b^* < \tau_b^{0} \right) 
\;=\;\sum_{j \in [J]}   \Pro\left(\tau_b^* < \tau_b^{0}, \; \tau_b^* = \tau_b^j \right) .
\end{align*}
Since $J$ is fixed, by Lemma \ref{lemma:log_max_sum}, 
if the first claim holds for each $j \in [J]$, then 
it follows that
\begin{align*}
    \lim_{b \to \infty} \frac{1}{b} \log \Pro\left(\tau_b^* < \infty \right)  = \max_{j \in [J] } (-r_j) = -r_*.
\end{align*}
Thus, it is sufficient to establish the result for a fixed $j \in [J]$. Accordingly, fix such a $j \in [J]$.

By Theorem \ref{thm:Collomore_prob},   we have that 
\begin{align*}
    \lim_{b \to \infty} \frac{1}{b}\log\Pro\left(\tau_b^j < \infty \right) =  -\inf_{x \in W^j} I(x) = -r_j,
\end{align*}
which  implies that
\begin{align*}
\limsup_{b \to \infty}    \frac{1}{b} \log \Pro(\tau^{*}_b < \tau^{0}_b,\;\; \tau_b^{*} = \tau_b^{j} )\leq \lim_{b \to \infty} \frac{1}{b}\log\Pro\left(\tau_b^j < \infty \right) = -r_j.
\end{align*}

Thus, to establish the first claim, it suffices to show that
\begin{align}\label{aux_lower_bound}
\liminf_{b \to \infty}    \frac{1}{b} \log \Pro(\tau^{*}_b < \tau^{0}_b,\;\; \tau_b^{*} = \tau_b^{j} ) \geq -r_j,
\end{align}
on which we now focus.



Let $\mathbb{C}^d[0,1]$ be the space of continuous functions from $[0,1]$ to $\bR^d$, which is endowed with the uniform norm $\|\cdot\|_{\infty}$, that is, for $f,g \in \mathbb{C}^d[0,1]$, $\|f-g\|_{\infty} = \sup_{0\leq t \leq 1} \|f(t) - g(t)\|$. For $f \in \mathbb{C}^d[0,1]$ and $\epsilon > 0$, define the $\epsilon$-neighbourhood of $f$ as
\begin{align*}
    \mathbb{B}_{\infty}(f, \epsilon)  :=  \{ g \in  \mathbb{C}^d[0,1]: \|g-f\|_{\infty} < \epsilon\}.
\end{align*}

Fix an arbitrary $v \in W^j$. By Lemma \ref{lemma:existence_epsilon_lower_bound}, there exists a small enough $\epsilon > 0$ such that
 for each $0 \leq j' \leq J$ with $j' \neq j$,
\begin{align}\label{set_property}
         B(v, 2\epsilon) \subset W^j, \;\;\text{ and } \;\;
        B(tv, 2\epsilon)\cap W^{j'} = \emptyset\; \text{ for } \;0 \leq t \leq 1,
\end{align}
where recall that $B(v, 2\epsilon) \subset \mathbb{R}^d$ denotes the open ball centered at $v$ with radius $2\epsilon$. Denote by $p_{v} \in \mathbb{C}^d[0,1]$ the path from the origin to $v$, that is
$$
p_v(t) = t v,\; \text{ for } \;t \in [0,1].
$$

Fix an arbitrary $\tau > 0$, and denote by $N_b  :=  \lfloor b /\tau \rfloor$, where $\lfloor z \rfloor$ is the integer part of $z > 0$. For $b \geq \tau$, define $s^{(b)} \in \mathbb{C}^d[0,1]$ as follows: for $0 \leq t \leq 1$,
$$
s^{(b)}(t)  :=  \frac{1}{N_b} \left( S_{\lfloor N_b t  \rfloor} + \left(N_b t - \lfloor N_b t \rfloor \right)  X_{\lfloor N_b t \rfloor + 1} \right).
$$
That is, the process $\{s^{(b)}:t \in [0,1]\}$ is constructed by linearly interpolating the random walk $\{S_n : n = 0,1,\ldots, N_b\}$, followed by rescaling in both time and space. In particular,
$s^{(b)}(k/N_b) = S_{k}/N_b$ for $k = 0,1,\ldots,N_b$ and the sample-path of $s^{(b)}$ is piece-wise linear.  Note that for any $f \in \mathbb{B}_{\infty}(p_v, \epsilon)$, by the triangle inequality,
$$
\left\| \tau f -  (b/N_b) p_v\right\|_{\infty} \leq \tau \epsilon + 
\left|\tau - (b/N_b)\right| \|p_v\|_{\infty}.
$$
By definition, for $b > \tau$, we have
$$\tau \leq \frac{b}{N_b} \leq \frac{b\tau}{b - \tau},
$$
which implies that if  $b > \tau ( 1+\epsilon^{-1}\|p_v\|_{\infty})$, then  for any $f \in \mathbb{B}_{\infty}(p_v, \epsilon)$,
$$
\left\| \tau f -  (b/N_b) p_v\right\|_{\infty}  < 2(b/N_b)\epsilon, 
$$
and thus
$$
\tau \mathbb{B}_{\infty}(p_v, \epsilon) \subset (b/N_b) \mathbb{B}_{\infty}(p_v, 2\epsilon).
$$
As a result,  $\{s^{(b)} \in \tau \mathbb{B}_{\infty}(p_v, \epsilon)\}\; \subset \;
\{s^{(b)} \in  (b/N_b) \mathbb{B}_{\infty}(p_v, 2\epsilon)\}$.
Then, by definition,  on the event $\{s^{(b)} \in \tau \mathbb{B}_{\infty}(p_v, \epsilon)\}$,  for $k = 0,1,\ldots, N_b$,
$$
\left\|\frac{1}{N_b} S_{k} - \frac{b}{N_b} p_v\left(\frac{k}{N_b}\right) \right\| \leq 2\frac{b}{N_b}  \epsilon,
$$
which is equivalent to 
$\left\|b^{-1} S_{k} - (k/N_b) v \right\| \leq 2 \epsilon$. Thus,
 due  to \eqref{set_property}, on the event $\{s^{(b)} \in \tau \mathbb{B}_{\infty}(p_v, \epsilon)\}$, we have
 $$
S_{N_b} \in b W^j, \quad S_{k} \not \in b W^{j'}\; \text{ for } \;k = 0,1,\ldots, N_b, \text{ and } 0 \leq j' \leq J \text{ with } j' \neq j,
$$
that is, in the first $N_b$ steps, the random walk $\{S_n\}$ has hit $bW^j$, but has not hit $b W^{j'}$ for each $j' \neq j$. Thus, for large enough $b$, we have 
\begin{align*}
\Pro(\tau^{*}_b < \tau^{0}_b,\;\; \tau_b^{*} = \tau_b^{j} )  \geq  \Pro\left( 
s^{(b)} \in \tau \mathbb{B}_{\infty}(p_v, \epsilon)
\right).
\end{align*}
Next, we apply the path-wise large deviations results to lower bound the term on the right-hand side. We note that the path $\tau p_v$ is differentiable on $[0,1]$ with gradient $\tau v$ for each $0 \leq t \leq 1$. Thus, by \citet[Theorem 2.3]{borovkov2012large}, we have
\begin{align*}
\liminf_{b \to \infty} \frac{1}{N_b} \log    \Pro\left( 
s^{(b)} \in \tau \mathbb{B}_{\infty}(p_v, \epsilon)
\right) \geq - \int_0^1 \Lambda^*(\tau v)\; dt =- \Lambda^*(\tau v),
\end{align*}
which implies that
$\liminf_{b \to \infty} \frac{1}{b} \log \Pro(\tau^{*}_b < \tau^{0}_b,\;\; \tau_b^{*} = \tau_b^{j} ) \geq -\tau^{-1} \Lambda^*(\tau v)$.

The rest of the arguments are the same as the proof of \citet[Theorem 2.1]{collamore1996}. Specifically, let $\mathcal{D} = \operatorname{ri}(\operatorname{cone}(\operatorname{dom}(\Lambda^*))$, which by assumption \eqref{assumption_on_X} has a non-empty intersection with $W^j$. For any  $v \in W^j \cap \mathcal{D}$,
by the proof of \citet[Theorem 2.1]{collamore1996},
$$
\inf_{\tau > 0} \tau^{-1} \Lambda^*(\tau v) = I(v), \quad \text{ and }
\inf_{v \in W^j \cap \mathcal{D}} I(v) = \inf_{v \in W^j} I(v).
$$
As a result, since $\tau > 0$ and $v \in W^j$ are arbitrary,
\begin{align*}
       \liminf_{b \to \infty} \frac{1}{b} \log \Pro(\tau^{*}_b < \tau^{0}_b,\;\; \tau_b^{*} = \tau_b^{j} ) \geq - \inf_{v \in W^j \cap \mathcal{D}} \inf_{\tau > 0} \tau^{-1} \Lambda^*(\tau v) = - \inf_{v \in W^j} I(v),
\end{align*}
which completes the proof of \eqref{aux_lower_bound}, and thus establishes the first claim.

\medskip

\noindent 
\textbf{Next, we prove the final claim that $0 < r_j < \infty$.} By Lemma \ref{app_lemma:upper_bound_var_single_j},  there exists a pair of 
 $x^j \in \overline{W^j}$  and $\beta^j \in \partial \mathcal{L}_0(\Lambda)$ such that 
$$
r_j = \beta^j \cdot x^j = I(x^j).
$$
The first equality immediately implies that $r_j < \infty$. 
Further, note that
$r_j = I(x^j)$. Since $x^j \in \overline{W^j}$, due to condition \eqref{assumptions_directions_of_mean}, we have
$x^j \in (\operatorname{cone}(B(\Exp[X_1],\delta))^c$ and $x^j \neq \textbf{0}$, where $\textbf{0}$ is the origin in $\bR^d$. Thus by \citet[Lemma 3.5]{collamore_thesis}, to show $r_j > 0$, it suffices to verify that $\Lambda^*(\textbf{0}) \neq 0$.

Due to condition \eqref{assumptions_directions_of_mean} and \eqref{assumption_on_X}, we have $\Exp[X_1] \neq \textbf{0}$ and $\textbf{0} \in \operatorname{int}(\operatorname{dom}(\Lambda))$. Since $\Lambda$ is infinitely differentiable on  $\operatorname{dom}(\Lambda)$ and $\nabla \Lambda(\textbf{0}) = \Exp[X_1] \neq \textbf{0}$ (see, e.g., Theorem 2.2 in \cite{brown1986fundamentals}), we have that for some $\theta \in \bR^d$, $\Lambda(\theta) < 0$, and thus $\Lambda^*(\textbf{0}) = \sup_{\theta \in \bR^d} (-\Lambda(\theta)) > 0$.    The proof is complete.
\end{proof}

\section{Proofs regarding importance sampling}

\subsection{Proof of Lemma \ref{lemma:upper_bound_var_single_j}}\label{app:proof_var_single_j}

\begin{proof}[Proof of Lemma \ref{lemma:upper_bound_var_single_j}]
Fix $j \in [J]$. 
Note that by definition, for each $\theta \in \operatorname{dom}(\Lambda)$,
\begin{align*}
    \widehat{\mathbb{W}}_{b}^{j}(\theta)  
= 
\mathbb{L}_{\theta}^{-1}(\tau_b^j) \, \mathbbm{1}\left\{
\tau_b^* < \tau_b^0,\; \tau_b^* = \tau_b^j \right\} \leq \mathbb{L}_{\theta}^{-1}(\tau_b^j),
\end{align*}
which implies that
\begin{align*}
    \Exp_{\theta}\left[ 
\left(\widehat{\mathbb{W}}_{b}^{j}(\theta)\right)^2
\right] \leq
\Exp_{\theta}\left[ 
\left(\mathbb{L}_{\theta}^{-1}(\tau_b^j)\right)^2
\right].
\end{align*}
As a result, the second claim follows immediately from part (iii) of Theorem \ref{thrm:Collamore_simulation}.

For the first claim, applying part (ii) of Theorem \ref{thrm:Collamore_simulation} with $W = W^j$, we have
\begin{align*}
    \limsup_{b \to \infty} \frac{1}{b} \log \Exp_{\beta^j}\left[ 
\left(\widehat{\mathbb{W}}_{b}^{j}(\beta^j)\right)^2
\right]  \leq -2 \inf_{x \in W^j} I(x) = -2r_j.
\end{align*}
Finally, by Jensen's inequality and Theorem \ref{thm:prob_rate},
\begin{align*}
    \liminf_{b \to \infty} \frac{1}{b} \log \Exp_{\beta^j}\left[ 
\left(\widehat{\mathbb{W}}_{b}^{j}(\beta^j)\right)^2
\right]  \geq  2 \liminf_{b \to \infty}   \frac{1}{b} \log \Pro\left(\tau_b^* < \tau_b^0,\; \tau_b^* = \tau_b^j \right) = -2 r_j.
\end{align*}
The proof is complete.
\end{proof}

\subsection{Proof of Lemma \ref{lemma:full_mixture_efficient}}\label{app:proof_full_mixture}
\begin{proof}[Proof of Lemma \ref{lemma:full_mixture_efficient}]
The second part of the lemma follows directly from the first, so we focus on the proof of the latter.
We first establish the lower bound. By change of measure,
\begin{align*}
\Exp_{\Theta}\left[ \widehat{\mathbb{W}}_{b}(\Theta) \sum_{j \in \mathcal{C}} \mathbbm{1}\{\tau_b^* = \tau_b^j\}\right] = \sum_{j \in \mathcal{C}}\Pro\left(\tau_b^* < \tau_0^*, \tau_b^* = \tau^j \right).
\end{align*}
By Lemma \ref{lemma:log_max_sum} and  Theorem \ref{thm:prob_rate},
\begin{align*}
\lim_{b\to \infty} \frac{1}{b} \log  \Exp_{\Theta}\left[ \widehat{\mathbb{W}}_{b}(\Theta) \sum_{j \in \mathcal{C}} \mathbbm{1}\{\tau_b^* = \tau_b^j\}\right] = -\min_{j \in \mathcal{C}} r_j.
\end{align*}
Note that the indicators above correspond to disjoint events due to condition \eqref{assumption_on_regions}. Then by Jensen's inequality, 
\begin{align*}
        \liminf_{b\to \infty} \frac{1}{b} \log  \Exp_{\Theta}\left[ \left(\widehat{\mathbb{W}}_{b}(\Theta) \right)^2 \sum_{j \in \mathcal{C}} \mathbbm{1}\{\tau_b^* = \tau_b^j\}\right] \geq -2 \min_{j \in \mathcal{C}} r_j.
\end{align*}

Next, we establish a matching upper bound. By change of measure,
\begin{align*}
\Exp_{\Theta}\left[ \left(\widehat{\mathbb{W}}_{b}(\Theta)\right)^2 \sum_{j \in \mathcal{C}} \mathbbm{1}\{\tau_b^* = \tau_b^j\}\right] = \sum_{j \in \mathcal{C}}\Exp\left[\widehat{\mathbb{W}}_{b}(\Theta)\mathbbm{1}\{\tau_b^* < \tau_0^*, \tau_b^* = \tau^j\} \right].
\end{align*}
By Lemma \ref{lemma:log_max_sum} and further change of measure,
\begin{align*}
   & \limsup_{b\to \infty} \frac{1}{b} \log       \Exp_{\Theta}\left[ \left(\widehat{\mathbb{W}}_{b}(\Theta)\right)^2 \sum_{j \in \mathcal{C}} \mathbbm{1}\{\tau_b^* = \tau_b^j\}\right]  \\
   =& \max_{j \in \mathcal{C}} \limsup_{b\to \infty} \frac{1}{b} \log       \Exp_{\beta^j}\left[ \mathbb{L}_{\beta^j}^{-1}(\tau_b^j)\widehat{\mathbb{W}}_{b}(\Theta)  \mathbbm{1}\{\tau_b^* = \tau_b^j\}\right].
\end{align*}

By the definition in \eqref{def:unbiased_estimator_Theta}, for each $j \in \mathcal{C}$, since $\beta^j \in \Theta$,
\begin{align}\label{aux_Wb_j_bound}
  \widehat{\mathbb{W}}_{b}(\Theta) \mathbbm{1}\{\tau_b^* = \tau_b^j\} & \leq |\Theta| 
   \left(   
\mathbb{L}_{\beta^j}(\tau_b^j) \right)^{-1} \mathbbm{1}\{\tau_b^* < \tau_b^0, \tau_b^* = \tau_b^j\}.
\end{align}  
Since $|\Theta|$ is fixed and finite,
\begin{align*}
   & \limsup_{b\to \infty} \frac{1}{b} \log       \Exp_{\Theta}\left[ \left(\widehat{\mathbb{W}}_{b}(\Theta)\right)^2 \sum_{j \in \mathcal{C}} \mathbbm{1}\{\tau_b^* = \tau_b^j\}\right]  \\
   \leq & \max_{j \in \mathcal{C}} \limsup_{b\to \infty} \frac{1}{b} \log       \Exp_{\beta^j}\left[ \mathbb{L}_{\beta^j}^{-2}(\tau_b^j)  \mathbbm{1}\{\tau_b^* < \tau_b^0,  \tau_b^* = \tau_b^j\} \right] = -2 \min_{j \in \mathcal{C}} r_j.
\end{align*}
where the last equality is due to Lemma \ref{lemma:upper_bound_var_single_j}.
The proof is complete.
\end{proof}

\subsection{Proof of Theorem \ref{thm:main_strategy}}\label{app:proof_main_strategy}
\begin{proof}[Proof of Theorem \ref{thm:main_strategy}]
By the same argument as in Lemma \ref{lemma:full_mixture_efficient}, in particular, the Jensen's inequality,
\begin{align*}
      \liminf_{b\to \infty} \frac{1}{b} \log \Exp_{\Theta}\left[ \left(\widehat{\mathbb{W}}_{b}(\Theta) \right)^2  \right] \geq -2 r_*.
\end{align*}
Next, we establish a matching upper bound.
By Lemma \ref{lemma:full_mixture_efficient},
\begin{align*}
        \lim_{b\to \infty} \frac{1}{b} \log  \Exp_{\Theta}\left[ \left(\widehat{\mathbb{W}}_{b}(\Theta) \right)^2 \sum_{j \in \mathcal{C}} \mathbbm{1}\{\tau_b^* = \tau_b^j\}\right] = -2 \min_{j \in \mathcal{C}} r_j = -2 r_\mathcal{C}.
\end{align*}

 We claim that it suffices to show that for each $j \not \in \mathcal{C}$,
\begin{align}\label{aux_j_not_in_C}
        \limsup_{b\to \infty} \frac{1}{b} \log  \Exp_{\Theta}\left[ \left(\widehat{\mathbb{W}}_{b}(\Theta) \right)^2 \mathbbm{1}\{\tau_b^* = \tau_b^j\}\right]  \leq -2 r_\mathcal{C}.
\end{align}
Indeed, if \eqref{aux_j_not_in_C} holds, by Theorem \ref{thm:prob_rate} and Jensen's inequality, we have that for each $j \not \in \mathcal{C}$,
$$-r_{j} = \lim_{b \to \infty} \frac{1}{b} \log\Pro\left(\tau_b^* < \tau_b^0, \tau_b^* = \tau_b^j \right) \leq -r_{\mathcal{C}},
$$
which would imply $r_{\mathcal{C}}  = r_{*}$. Then the second claim follows immediately from Lemma \ref{lemma:log_max_sum}.

Now, we fix some $j \not \in \mathcal{C}$ and show \eqref{aux_j_not_in_C} holds. By assumption, there exists some $\gamma \in \Theta$ such that \eqref{def:key_condition} holds. As in the proof of Lemma \ref{lemma:full_mixture_efficient}, by change of measure,
\begin{align*}
    \Exp_{\Theta}\left[ \left(\widehat{\mathbb{W}}_{b}(\Theta) \right)^2 \mathbbm{1}\{\tau_b^* = \tau_b^j\}\right] =
     \Exp_{\gamma}\left[ 
     \mathbb{L}_{\gamma}^{-1}(\tau_b^j)
     \widehat{\mathbb{W}}_{b}(\Theta)  \mathbbm{1}\{\tau_b^* = \tau_b^j\}\right].
\end{align*}
Then by arguments similar to those  in \eqref{aux_Wb_j_bound}, and since $|\Theta|$ is finite,
\begin{align*}
        &\limsup_{b\to \infty} \frac{1}{b} \log  \Exp_{\Theta}\left[ \left(\widehat{\mathbb{W}}_{b}(\Theta) \right)^2 \mathbbm{1}\{\tau_b^* = \tau_b^j\}\right]  \\
        \leq 
        &     \limsup_{b\to \infty} \frac{1}{b} \log  \Exp_{\gamma}\left[ \mathbb{L}_{\gamma}^{-2}(\tau_b^j) \mathbbm{1}\{\tau_b^* < \tau_b^0, \tau_b^* = \tau_b^j\}\right]\leq \limsup_{b\to \infty} \frac{1}{b} \log  \Exp_{\gamma}\left[ \mathbb{L}_{\gamma}^{-2}(\tau_b^j) \right].
\end{align*}
Finally, by Theorem \ref{thm:prob_rate}, $r_{\mathcal{C}} > 0$. Thus, due to condition \eqref{def:key_condition}, by part (iii) of Theorem \ref{thrm:Collamore_simulation},  
\begin{align*}
    \limsup_{b\to \infty} \frac{1}{b} \log  \Exp_{\Theta}\left[ \left(\widehat{\mathbb{W}}_{b}(\Theta) \right)^2 \mathbbm{1}\{\tau_b^* = \tau_b^j\}\right]  \leq - 
    v(W^j; \gamma)
    \leq -2 r_{\mathcal{C}},
\end{align*}
which proves \eqref{aux_j_not_in_C}, and thus completes the proof.
\end{proof}

\subsection{Proof of Theorem \ref{thm:negative}}\label{app:neg_result}
Recall for $j \in [J]$, $\beta^j \in \partial \mathcal{L}_0 \Lambda$ and $x^j \in \partial W^j$ are defined in Definition \eqref{def:optimal_beta_j} and Lemma \ref{app_lemma:upper_bound_var_single_j}. By part (iii) of Lemma \ref{app_lemma:upper_bound_var_single_j}, there exists $\rho^j > 0$ such that 
\begin{align*}
    x^j = \rho^j \nabla \Lambda(\beta^j).
\end{align*}
Further, recall the definition of 
  $\Pro_{\beta^j}$ in \eqref{def:exp_tilting}. For $j \in [J]$, define the cumulant function of $X_1$ under $\Pro_{\beta^j}$ and the corresponding Legendre--Fenchel conjugate as follows:
\begin{align*}
\begin{split}
        &\Lambda^j(\theta) := \log\Exp_{\beta^j}[e^{\theta \cdot X_1}] = \Lambda(\theta + \beta^j),\quad \text{ for } \theta \in \bR^{d}, \\
    & \Lambda^{j,*}(x) := \sup_{\theta \in \bR^{d}} \theta \cdot x - \Lambda^j(\theta) = \Lambda^*(x)- \beta^j \cdot x,\quad \text{ for } x \in \bR^d.
    \end{split}
\end{align*}
 
We begin by establishing several technical lemmas.
\begin{lemma}
    \label{app_lemma:P_beta_j}
    Fix $j \in [J]$. 
    Assume that conditions \eqref{assumptions_directions_of_mean}, \eqref{assumption_on_X}, and \eqref{assumption_on_regions} hold.  
    \begin{enumerate}[label=(\roman*)]
    \item $\operatorname{dom}(\Lambda^j)$ is an open set, and $\operatorname{dom}(\Lambda^{j,*}) = \operatorname{dom}(\Lambda^{*})$. Further, $
\Exp_{\beta^j}[X_1] =  {x^j}/{\rho^j}$.
\item $\Lambda^{j,*}\left({x^j}/{\rho^j} \right) = 0$, and $\Lambda^{j,*}(\cdot)$   is continuously differentiable at ${x^j}/{\rho^j}$. 
\item $\Lambda^{j,*}(x) > 0$ for $x \neq ({x^j}/{\rho^j})$, and for any closed subset $F \subset \bR^{d}$ such that $x^j/\rho^j \not\in F$,  
\begin{align*}
    \inf_{x \in  F} \Lambda^{j,*}(x) > 0.
\end{align*}
\end{enumerate}
\end{lemma}
\begin{proof}
By definition, $\operatorname{dom}(\Lambda^j) = \operatorname{dom}(\Lambda) - \beta^j$, and
$\operatorname{dom}(\Lambda^{j,*}) = \operatorname{dom}(\Lambda^{*})$. Thus, due to \eqref{assumption_on_X}, $\operatorname{dom}(\Lambda^j)$ is an open set. Further, by Lemma \ref{app_lemma:upper_bound_var_single_j},
    $$
\Exp_{\beta^j}[X_1] = \nabla \Lambda^j(\mathbf{0}) = \nabla \Lambda(\beta^j) = \frac{x^j}{\rho^j},
$$
which proves (i).

By condition \eqref{assumption_on_X}, the convex hull of the support of $X_1$ under $\Pro$ is $\mathbb{R}^d$, and the same holds under $\Pro_{\beta^j}$ by mutual absolute continuity. As a result, under $\Pro_{\beta^j}$, the covariance matrix of $X_1$ is invertible.  Thus, by the implicit function theorem \cite[Theorem 9.24]{rudin1976principles}, there exists two open sets $\mathcal{U},\mathcal{V} \subset \bR^{d}$ such that
$\mathbf{0} \in \mathcal{U}$, $x^j/\rho^j \in \mathcal{V}$, and 
$$
\nabla \Lambda^j :\; \mathcal{U} \;\; \to \;\; \mathcal{V} \text{ is a one-to-one function}.
$$
Denote by $H:\mathcal{V} \to \mathcal{U}$ the inverse function of $\nabla \Lambda^j: \mathcal{U} \to \mathcal{V}$. Again by Theorem 9.24 in \cite{rudin1976principles}, $H(\cdot)$ is continuously differentiable. By  the same argument as for Lemma 2.2.31(b) in \cite{dembo2009large}, for $x \in \mathcal{V}$,
$$
\Lambda^{j,*}(x) = H(x) \cdot x - \Lambda^j(H(x)),
$$
which, since $H(x^j/\rho^j) = \mathbf{0}$,  implies that (ii) holds.

Now we focus on (iii). Let $x \in \bR^d$, and assume $\Lambda^{j,*}(x) = 0$. By definition of $\Lambda^{j,*}$, we have
$$
\theta \cdot x \leq \Lambda^j(\theta), \text{ for any } \theta \in \bR^d,
$$
and thus $x$ is a sub-gradient of $\Lambda^j(\cdot)$ at zero. Since $\nabla \Lambda^{j}(\mathbf{0}) = \Exp_{\beta^j}[X_1]$, we have $x = \Exp_{\beta^j}[X_1]$, which proves the first statement in (iii).

For the second statement of (iii), since $\Lambda^{j,*}$ is continuous differentiable on $\mathcal{V}$ and $\mathbf{0} \not \in F$, there exists a sufficient small $\delta' > 0$ such that $B(x^j/\rho^j,\delta') \cap F = \emptyset$ and
$$
\inf\{\Lambda^{j,*}(x) :  \|x - x^j/\rho^j\| = \delta'  \} > 0.
$$
Due to the convexity of $\Lambda^{j,*}$, for any $x \in F$,
\begin{align*}
\Lambda^{j,*}(x) \geq &\frac{\|x-x^j/\rho^j\|}{\delta'} \Lambda^{j,*}\left(\frac{\delta'}{\|x-x^j/\rho^j\|}   x + \left(1 - \frac{\delta'}{\|x-x^j/\rho^j\|}   \right) \frac{x^j}{\rho^j}\right) \\
\geq &\inf\{\Lambda^{j,*}(z) :  \|z - x^j/\rho^j\| = \delta'  \}.
\end{align*}
Then the proof is complete.
\end{proof}

\begin{lemma}\label{lemma:aux_prob1}
    Fix $j \in [J]$. 
    Assume that conditions \eqref{assumptions_directions_of_mean}, \eqref{assumption_on_X}, and \eqref{assumption_on_regions} hold.  Then 
    for any $\delta > 0$,
        \begin{align*}
              \lim_{b \to \infty} \frac{1}{b}  \log \Pro_{\beta^j}\left(\left|{\tau_b^j}/{b} - \rho^{j}\right| \leq \delta  \right) = 0.
        \end{align*}
\end{lemma}

\begin{proof}
Note that 
$$\Pro_{\beta^j}\left( \left|\tau_b^j/b - \rho^j\right| \leq  \delta \right) =
\Pro_{\beta^j}\left( \tau_b^j \leq  b(\rho^j+\delta) \right) 
-
\Pro_{\beta^j}\left( \tau_b^j < b(\rho^j-\delta) \right).
$$
It then suffices to show that
\begin{align}
    \limsup_{b \to \infty} \frac{1}{b} \log\left( \Pro_{\beta^j}\left( \tau_b^j < b(\rho^j-\delta) \right) \right) < 0, \label{aux:step_1} \\
        \liminf_{b \to \infty} \frac{1}{b} \log\left( \Pro_{\beta^j}\left( \tau_b^j \leq b(\rho^j+\delta) \right) \right) = 0. \label{aux:step_2}
\end{align}

\noindent \underline{\textit{Step 1: we show \eqref{aux:step_1} holds.}} Without loss of generality, we assume $\delta < \rho^j$. 
By Lemma \ref{app_lemma:upper_bound_var_single_j}, for each $v \in W^j$, $\beta^j \cdot v \geq r_j$, and then by definition and the union bound,
\begin{align*}
    &\Pro_{\beta^j}\left(\tau_b^j < b (\rho^j - \delta) \right) = 
    \Pro_{\beta^j}\left(\tau_b^j <  b (\rho^j - \delta),\; S_{\tau_b^j} \in b W^j \right) \\
    \leq &
    \Pro_{\beta^j}\left( 
    \max_{n  <  b(\rho^j-\delta)} \beta^j \cdot S_n \geq b r_j
    \right) \leq \sum_{n < b(\rho^j - \delta)}
    \Pro_{\beta^j}\left( 
 \sum_{i=1}^{n} \left(\beta^j \cdot X_i \right) \geq b r_j
    \right).
\end{align*}
Since $\Lambda(\beta^j) = 0$ due to Lemma \ref{app_lemma:upper_bound_var_single_j}, by the definition of $\Pro_{\beta^j}$ in \eqref{def:exp_tilting}, for $\epsilon \in \bR$
\begin{align*}
    \Exp_{\beta^j} \left[ e^{\epsilon \beta^j \cdot X_1} \right] =  \Lambda((1+\epsilon)\beta^j) := h(\epsilon).
\end{align*}
Now, due to \eqref{assumption_on_X}, $\operatorname{dom}(\Lambda)$ is open, and thus for $\epsilon \in \bR$ such that $|\epsilon|$ is sufficiently small, $(1+\epsilon) \beta^j \in \operatorname{dom}(\Lambda)$, and thus due to Lemma \ref{app_lemma:upper_bound_var_single_j},
$$
h'(0) = \nabla \Lambda(\beta^j) \cdot \beta^j = \left(\frac{1}{\rho^j} x^j \right) \cdot \beta^j = \frac{r_j}{\rho^j} > 0.
$$
By the Markov inequality, for $n <  b(\rho^j -\delta)$ and  $\epsilon > 0$
\begin{align*}
\Pro_{\beta^j}\left( 
 \sum_{i=1}^{n} \left(\beta^j \cdot X_i \right) \geq b r_j
    \right) &\leq \exp\left(n \Lambda((1+\epsilon)\beta^j) - b \epsilon r_j \right) = \exp\left( n h(\epsilon) - b \epsilon r_j\right) \\
    &\leq \exp\left(b\left((\rho^j-\epsilon) h(\epsilon) - \epsilon r_j \right) \right).
\end{align*}
Since $h(0) = 0$ and $h'(0) = r_j/(\rho^j)$, there exists $\tilde{\epsilon} > 0$ sufficiently small, such that 
\begin{align*}
    -\tilde{\delta} := (\rho^j-\epsilon) h(\tilde{\epsilon}) - \tilde{\epsilon} r_j < 0.
\end{align*}
As a result,
\begin{align*}
    &\Pro_{\beta^j}\left(\tau_b^j < b (\rho^j - \delta) \right) \leq b \exp\left( - b \tilde{\delta} \right),
\end{align*}
and thus proves that \eqref{aux:step_1} holds since $\log(b)/b \to 0$ as $b \to \infty$.

\medskip
\noindent \underline{\textit{Step 2: we show \eqref{aux:step_2} holds.}} 
Since $x^j \in \partial W^j$ and $W^{j}$ is open, there exist $\{v_n: n \geq 1\} \subset W^j$ and $\{\epsilon_n > 0: n \geq 1\}$  such that
\begin{align*}
    \lim_{n \to \infty} v_n = x^j, \quad \text{ and } \quad 
   B(v_n, \epsilon_n) \subset W^j.
\end{align*}

For any $b > 0$, define 
$n_b := \lfloor b   \rho^j \rfloor$, and by definition,
\begin{align}\label{aux:ratio}
 1 \leq \frac{b \rho^j}{n_b} \leq \frac{1}{1 - 1/(b\rho^j)}.
\end{align}
Fix $n \geq 1$. Due to \eqref{aux:ratio}, for sufficiently large $b > 0$,
$$
\frac{b \rho^j}{n_b} B(v_n, \epsilon_n) \supset B(v_n, \epsilon_n/2) \;\; \Leftrightarrow \;\; \frac{b }{n_b} B(v_n, \epsilon_n) \supset \frac{1}{\rho^j}B(v_n, \epsilon_n/2).
$$
As a result, for sufficiently large $b > 0$, we have
\begin{align*}
 \Pro_{\beta^j}\left(\tau_b^j \leq  b(\rho^j+\delta) \right) &\geq \Pro_{\beta^j}\left({S_{n_b}}\in  b B(v_n, \epsilon_n) \right)
 = \Pro_{\beta^j}\left(\frac{S_{n_b}}{n_b}\in \frac{b}{n_b} B(v_n, \epsilon_n) \right)\\
 &\geq \Pro_{\beta^j}\left(\frac{S_{n_b}}{n_b}\in \frac{1}{\rho^j} B(v_n, \epsilon_n/2) \right).
\end{align*}
Since $\operatorname{dom}(\Lambda^j)$ is open by Lemma \ref{app_lemma:P_beta_j}, by Cram\'er theorem \cite[Corollary 6.1.6]{dembo2009large}, $\{S_{n_b}/n_b: b > 0\}$ satisfies the large deviations principle with the rate function $\Lambda^{j,*}$ under $\Pro_{\beta^j}$, which, in particular, implies that for a fixed $n \geq 1$,
\begin{align*}
    \liminf_{b \to \infty} \frac{1}{n_b} \log \Pro_{\beta^j}(\tau_b^j \leq b (\rho^j + \delta))) \geq - \inf_{x \in (\rho^j)^{-1} B(v_n, \epsilon_n/2)} \Lambda^{j,*}(x) \geq - \Lambda^{j,*}(v_n/\rho^j).
\end{align*}
Note that the left-hand side does not depend on $n$. By  Lemma \ref{app_lemma:P_beta_j},  $\Lambda^{j,*}(\cdot)$ is continuous at $x^j/\rho^j$ and $\Lambda^{j,*}(x^j/\rho^j) = 0$. 
By sending $n \to \infty$, since $v_n/\rho^j$ approaches $x^j/\rho^j$,  we have
\begin{align*}
    \liminf_{b \to \infty} \frac{1}{n_b} \log \Pro_{\beta^j}(\tau_b^j \leq b (\rho^j + \delta)) \geq  - \Lambda^{j,*}(x^j/\rho^j) = 0.
\end{align*}
Then the proof for \eqref{aux:step_2} is complete, which also completes the proof of the lemma.
\end{proof}

\begin{lemma}\label{lemma:aux_prob2}
    Fix $j \in [J]$. 
    Assume that conditions \eqref{assumptions_directions_of_mean}, \eqref{assumption_on_X},  \eqref{assumption_on_regions}, and \eqref{assumption:negative_results} hold.  For $j' \in \{0,1,\ldots,J\}$ such that $j' \neq j$,
        \begin{align*}
          \limsup_{b \to \infty} \frac{1}{b}  \log \Pro_{\beta^j}\left(\tau_b^{j'} < \infty\right) < 0.
        \end{align*}  
\end{lemma}
\begin{proof}
In condition \eqref{assumption:negative_results}, without of loss generality, assume $W^{0}$ is \textit{non-empty}, that is, case (ii) holds; otherwise, the conclusion for $j' = 0$ is trivial.

By Lemma \ref{app_lemma:P_beta_j}, $\operatorname{dom}(\Lambda^j)$ is an open set and $\operatorname{dom}(\Lambda^{j,*}) = \operatorname{dom}(\Lambda^{*})$. Further, due to \eqref{assumption:negative_results}, since $\Exp_{\beta^j}[X_1] = x^j/\rho^j$ and $x^j \in \partial W^j$, there exists some $\delta > 0$ such that
$$
\text{cone}(B(\Exp_{\beta^j}[X_1], \delta)) \cap W^{j'} = \emptyset.
$$
Then due to \eqref{assumption_on_X} and case (ii) in \eqref{assumption:negative_results} if $j' = 0$, by Theorem \ref{thm:Collomore_prob}, we have
\begin{align*}
   \lim_{b \to \infty} \frac{1}{b}  \log \Pro_{\beta^j}\left(\tau_b^{j'} < \infty\right) = - \inf_{x \in W^{j'}} I^j(x), \quad \text{ where } I^j(x) := \sup_{\theta \in \mathcal{L}_0 \Lambda^j} \theta \cdot x.
\end{align*} 
Finally, by Lemma \ref{app_lemma:P_beta_j}, $\Lambda^{j,*}(\mathbf{0}) \neq 0$. Then by the same argument as in Theorem \ref{thm:prob_rate}, $\inf_{x \in W^{j'}} I^j(x) > 0$. The proof is complete.
\end{proof}

\begin{lemma}\label{lemma:aux_prob34}
    Fix $j \in [J]$. 
    Assume that conditions \eqref{assumptions_directions_of_mean}, \eqref{assumption_on_X}, and \eqref{assumption_on_regions} hold.  
    \begin{enumerate}[label=(\roman*)]
        \item For any $\delta >0$,
        \begin{align*}
          \limsup_{n \to \infty} \frac{1}{n}  \log \Pro_{\beta^j}\left( \frac{1}{n}\beta^j \cdot S_n -\frac{r_j}{\rho^j} \geq \delta  \right) < 0.
        \end{align*}
        \item For any $\theta \in \operatorname{dom}(\Lambda)$ such that $\theta \neq \beta^j$, there exists a sufficiently small $\delta > 0$ such that
        \begin{align*}
          \limsup_{n \to \infty} \frac{1}{n}  \log \Pro_{\beta^j}\left( \frac{1}{n}(\theta -\beta^j) \cdot S_n -\Lambda(\theta) \geq -\delta  \right) < 0.
        \end{align*}      
    \end{enumerate}
\end{lemma}
\begin{proof}
We first focus on (i). Fix $\delta > 0$ and define
$$
F_1 := \left\{x \in \bR^d: \beta^j \cdot x - r_j/\rho^j \geq \delta \right\}.
$$
Clearly, $F_1$ is closed, and since $\beta^j \cdot x^j = r_j$ due to Lemma \ref{app_lemma:upper_bound_var_single_j}, we have $x^j/\rho^j \not \in F_1$.

By Lemma \ref{app_lemma:P_beta_j},
$\operatorname{dom}(\Lambda^j)$ is open and $\inf_{x \in F_1} \Lambda^{j,*}(x) > 0$. Then by Cram\'er theorem \cite[Corollary 6.1.6]{dembo2009large}, $\{S_{n}/n: n\geq 1\}$ satisfies the large deviations principle with the rate function $\Lambda^{j,*}$ under $\Pro_{\beta^j}$, which implies that 
      \begin{align*}
          \limsup_{n \to \infty} \frac{1}{n}  \log \Pro_{\beta^j}\left( \frac{1}{n}\beta^j \cdot S_n -\frac{r_j}{\rho^j} \geq \delta  \right) \leq - \inf_{x \in F_1}\Lambda^{j,*}(x)  < 0,
        \end{align*}
and proves (i).

Now, we focus on (ii). For $\delta > 0$, define
$$
F_{2,\delta} := \left\{x \in \bR^d: (\theta - \beta^j) \cdot x - \Lambda(\theta) \geq -\delta \right\}.
$$
Clearly, $F_{2,\delta}$ is closed. By the same argument as above,
\begin{align*}
          \limsup_{n \to \infty} \frac{1}{n}  \log \Pro_{\beta^j}\left( \frac{1}{n}(\theta -\beta^j) \cdot S_n -\Lambda(\theta) \geq -\delta  \right) \leq - \inf_{x \in F_{2,\delta}} \Lambda^{j,*}(x).
\end{align*}     
Then it suffices to show that $\inf_{x \in F_{2,\delta}} \Lambda^{j,*}(x) > 0$ for sufficient small $\delta > 0$. Since $\Lambda(\beta^j) = 0$, by the strong convexity of $\Lambda$ in $\operatorname{dom}(\Lambda)$ \cite[Theorem 1.3]{brown1986fundamentals},  we have
\begin{align*}
  \epsilon := \Lambda(\theta) - (\theta - \beta^j) \cdot \frac{x^j}{\rho^j} = \Lambda(\theta) - \Lambda(\beta^j) - (\theta - \beta^j) \cdot \nabla \Lambda(\beta^j) > 0.
\end{align*}
Thus for $\delta \in (0,\epsilon)$,
 $x^j/\rho^j \not \in F_{2,\delta}$. Then by Lemma \ref{app_lemma:P_beta_j}, $\inf_{x \in F_{2,\delta}} \Lambda^{j,*}(x) > 0$ for $\delta \in (0,\epsilon)$. The proof is complete.
\end{proof}

For $j \in [J]$, $\epsilon, \delta > 0$, and $\Theta \subset \operatorname{dom}(\Lambda)$, define the following events:
\begin{align}\label{def:aux_A_epsilon}
\begin{split}
  & \mathcal{A}_{b}^{(j,1)}  :=   
 \{\tau_b^k = \infty,\; \text{ for } \;k =0,1,\ldots,J \text{ and } k \neq j\},\\
   &\mathcal{A}_{b,\epsilon}^{(j,2)}  :=   \left\{\left|{\tau_b^j}/{b} - \rho^{j}\right| \leq \epsilon \right\},\\
& \mathcal{A}_{b, \delta}^{(j, 3)}  :=  
\left\{ \beta^j \cdot S_{\tau_b^j} \leq b(r_j +  \delta/4)  \right\},\\
 & \mathcal{A}_{b, \delta,\Theta}^{(j,4)}  := \bigcap_{\theta \in \Theta}   \left\{(\theta-\beta^j) \cdot S_{\tau_b^j} - \tau_b^j \Lambda(\theta)  \leq 
- \delta b 
\right\},\\
& \mathcal{A}_{b,\epsilon,\delta,\Theta}^{(j,*)}  := \mathcal{A}_{b}^{(j,1)} \cap \mathcal{A}_{b,\epsilon}^{(j,2)} \cap \mathcal{A}_{b,\delta}^{(j,3)} \cap \mathcal{A}_{b,\delta,\Theta}^{(j,4)}.
\end{split}
\end{align}

\begin{lemma}\label{aux_lemma_overall}
    Assume that conditions \eqref{assumptions_directions_of_mean}, \eqref{assumption_on_X},  \eqref{assumption_on_regions}, and \eqref{assumption:negative_results} hold.    Let   $\Theta \subset \operatorname{dom}(\Lambda)$ be a finite set such that $\beta^j \not \in \Theta$. There exist  $\epsilon, \delta > 0$ such that
    \begin{align*}
       \lim_{b \to \infty}  \frac{1}{b} \log \Pro_{\beta^j}\left(\mathcal{A}_{b,\epsilon,\delta,\Theta}^{(j,*)} \right) = 0.
    \end{align*}
\end{lemma}
\begin{proof}
By definition and the union bound,
\begin{align*}
\Pro_{\beta^j}\left(\mathcal{A}_{b,\epsilon,\delta,\Theta}^{(j,*)} \right) \geq &\Pro_{\beta^j}\left(\mathcal{A}_{b,\epsilon}^{(j,2)}\right) - 
\Pro\left( \left(\mathcal{A}_{b}^{(j,1)}\right)^c \right) \\
&- \Pro\left( \mathcal{A}_{b,\epsilon}^{(j,2)}  \setminus \mathcal{A}_{b,\delta}^{(j,3)} \right) -
\Pro\left( \mathcal{A}_{b,\epsilon}^{(j,2)} \setminus \mathcal{A}_{b,\delta,\Theta}^{(j,4)} \right).
\end{align*}
Thus, it suffices to show that there exist $\epsilon,\delta > 0$ such that
\begin{align}
\begin{split}
     &\liminf_{b \to \infty}  \frac{1}{b} \log    \Pro_{\beta^j}\left(\left|{\tau_b^j}/{b} - \rho^{j}\right| \leq \epsilon  \right) = 0, \\
 & \limsup_{b \to \infty}  \frac{1}{b} \log    \Pro_{\beta^j}\left(\tau_b^{k} < \infty \right) < 0  \text{ for each } k \neq j, \\
  & \limsup_{b \to \infty}  \frac{1}{b} \log    \Pro_{\beta^j}\left(\left|{\tau_b^j}/{b} - \rho^{j}\right| \leq   \epsilon , \;\; \beta^j \cdot S_{\tau_b^j} \geq  b(r_j + \delta/4) \right) < 0, \\
   & \limsup_{b \to \infty}  \frac{1}{b} \log    \Pro_{\beta^j}\left(\left|{\tau_b^j}/{b} - \rho^{j}\right| \leq \epsilon,\; (\theta-\beta^j) \cdot S_{\tau_b^j} - \tau_b^j \Lambda(\theta)  \geq 
-\delta b  \right) < 0, \text{ for } \theta \in \Theta.
\end{split}
\label{eqref:aux_overall}
\end{align}
The first two claims have been established in Lemma \ref{lemma:aux_prob1} and Lemma \ref{lemma:aux_prob2}, respectively, for any $\epsilon > 0$.

Now, we focus on the third claim in \eqref{eqref:aux_overall}. By definition and the union bound,
\begin{align*}
 \Pro_{\beta^j}\left(\left|{\tau_b^j}/{b} - \rho^{j}\right| \leq   \epsilon, \;\; \beta^j \cdot S_{\tau_b^j} \geq  b(r_j + \delta/4) \right) 
 \leq \sum_{n = b(\rho^j-\epsilon)}^{b(\rho^j+\epsilon)}  \Pro_{\beta^j}\left(\beta^j \cdot S_{n} \geq  b(r_j + \delta/4) \right).
\end{align*}
For integer $n \geq 0$ such that $|n/b - \rho^j| \leq \epsilon$, we have
\begin{align*}
    \frac{1}{n}  b(r_j + \epsilon/4) \geq \frac{1}{\rho^j +\epsilon} (r_j + \delta/4) = \frac{r_j}{\rho^j} + \tilde{\delta},
\end{align*}
where we define
$$
\tilde{\delta}:= \frac{1}{\rho^j + \epsilon}\left[\frac{\delta}{4} - \frac{r_j}{\rho^j} \epsilon \right].
$$
Thus, if $\delta/4 >  ({r_j}/{\rho^j}) \epsilon $, we have $\tilde{\delta} > 0$, and by Lemma \ref{lemma:aux_prob34}(i), there exists some constant $c > 0$ such that  for all sufficiently large $n$, 
\begin{align*}
           \frac{1}{n}  \log \Pro_{\beta^j}\left( \frac{1}{n}\beta^j \cdot S_n \geq \frac{r_j}{\rho^j} + \tilde{\delta}  \right) \leq -c < 0.
    \end{align*}
As a result, for sufficiently large $b > 0$, if $\delta/4 >  ({r_j}/{\rho^j}) \epsilon $,
\begin{align*}
    \sum_{n = b(\rho^j-\epsilon)}^{b(\rho^j+\epsilon)}  \Pro_{\beta^j}\left(\beta^j \cdot S_{n} \geq  b(r_j + \delta/4) \right) \leq (2b\epsilon) e^{-c b(\rho^j - \epsilon)}.
\end{align*}
Thus, if $\delta/4 >  ({r_j}/{\rho^j}) \epsilon $ and $\epsilon < \rho^j$, then the third claim in \eqref{eqref:aux_overall} holds.

The fourth claim in \eqref{eqref:aux_overall} follows from similar arguments as above, once we replace   Lemma \ref{lemma:aux_prob34}(i) with Lemma \ref{lemma:aux_prob34}(ii).
\end{proof}

\begin{proof}[Proof of Theorem \ref{thm:negative}]
Recall the $j$ in Theorem \ref{thm:negative} and $\beta^j$ in Definition \ref{def:optimal_beta_j}. By change of measure and the definition of $\Pro_{\Theta}$ in Subsection \ref{sec:IS_general}, we have
\begin{align*}
&\Exp_{\Theta}\left[ \left(\widehat{\mathbb{W}}_{b}(\Theta) \right)^2  \right] 
\geq \Exp\left[ \widehat{\mathbb{W}}_{b}(\Theta)  \mathbbm{1}\left\{\tau_b^* = \tau_b^j\right\}\right] \\
= & \Exp_{\beta^j}\left[
\mathbb{L}_{\beta^j}^{-2}(\tau_b^j) 
   \left( \frac{1}{|\Theta|} \sum_{\theta \in \Theta} 
\mathbb{L}_{\theta}(\tau_b^j)/\mathbb{L}_{\beta^j}(\tau_b^j) \right)^{-1} 
\mathbbm{1}\left\{\tau_b^* = \tau_b^j\right\}\right].
\end{align*}

Recall the definition of $\mathcal{A}_{b,\epsilon,\delta,\Theta}^{(j,*)}$ in \eqref{def:aux_A_epsilon}. By Lemma \ref{aux_lemma_overall}, there exist  $\epsilon, \delta > 0$ such that
    \begin{align}\label{aux:overall_approach_zero}
       \lim_{b \to \infty}  \frac{1}{b} \log \Pro_{\beta^j}\left(\mathcal{A}_{b,\epsilon,\delta,\Theta}^{(j,*)} \right) = 0.
\end{align}
By definition, on the event $\mathcal{A}_{b,\epsilon,\delta,\Theta}^{(j,*)}$, since $\Lambda(\beta^j) = 0$, we have that for each $\theta \in \Theta$,
\begin{align*}
    &\mathbb{L}_{\beta^j}(\tau_b^j) = \exp(\beta^j \cdot S_{\tau_b^j}) \leq \exp\left( b(r_j + \delta/4)\right), \\
    &\mathbb{L}_{\theta}(\tau_b^j)/\mathbb{L}_{\beta^j}(\tau_b^j) =\exp\left( (\theta - \beta^j) \cdot S_{\tau_b^j} - \tau_b^j \Lambda(\theta) \right)  \leq \exp\left(-\delta b\right).
\end{align*}
 As a result, since $r_j = r_*$,
\begin{align*}
\Exp_{\Theta}\left[ \left(\widehat{\mathbb{W}}_{b}(\Theta) \right)^2  \right]  
\geq \Exp_{\beta^j} \left[
\exp\left(-2b(r_* + \delta/4)\right) 
\exp\left(\delta b\right); \mathcal{A}_{b,\epsilon,\delta,\Theta}^{(j,*)}
\right],
\end{align*}
which implies that
\begin{align*}
    \liminf_{b \to \infty} \frac{1}{b}\log \Exp_{\Theta}\left[ \left(\widehat{\mathbb{W}}_{b}(\Theta) \right)^2  \right] \geq -2r_* + \delta/2 +  \lim_{b \to \infty}  \frac{1}{b} \log \Pro_{\beta^j}\left(\mathcal{A}_{b,\epsilon,\delta,\Theta}^{(j,*)} \right).
\end{align*}
Due to \eqref{aux:overall_approach_zero}, the last term is zero. As a result, the left-hand side of the above inequality is bounded below by $-2r_* + \delta/2$. 
Since $\delta > 0$, the proof is complete.
\end{proof}

\section{Proofs regarding the multidimensional Siegmund problem}
\label{app:multi_ruin}

\begin{proof}[Proof of Lemma  \ref{lemma:opt_probs}]    
Due to \eqref{def:r_j_alt},
$$
r_A = \sup_{\theta \in \mathcal{L}_0(\Lambda)} \inf_{x \in \overline{W^A}} \theta \cdot x.
$$
By the definition of $W^A$ in \eqref{def:W_A}, for any $\theta \in \bR^d$, if
$\theta_k \geq 0, \theta_{k'} \leq 0, \text{ for }k \in A, k' \in A^c$, then
\begin{align*}
    \inf_{x \in\overline{W^A}} \theta \cdot x = 
        u \sum_{k \in A} \theta_k - \ell \sum_{k' \in A^c} \theta_{k'}.
\end{align*}
Otherwise, $\inf_{x \in \overline{W^A}} \theta \cdot x = -\infty$. Then the proof of the first claim is complete.

For the second claim, due to \eqref{def:lower_bound_V_gamma} and since $\Lambda(\gamma) \leq 0$, we have
\begin{align*}
v_A(\gamma)  
\geq  
\sup_{\theta \in \bR^d: \Lambda(\gamma) + \Lambda(\theta-\gamma) \leq 0}\; \inf_{x \in \overline{W^A}} \theta \cdot x \geq 
\sup_{\theta \in \bR^d:  \Lambda(\theta-\gamma) \leq 0} \;\inf_{x \in \overline{W^A}} \theta \cdot x.
\end{align*}
The rest of the arguments are the same as above. The proof is complete.
\end{proof}

\subsection{KKT conditions: multidimensional Siegmund problem}\label{app:KKT}
Recall the first optimization problem in  Lemma \ref{lemma:opt_probs}.
\begin{lemma}\label{lemma:app_KKT}
For any $A \in \mathcal{P}_d^{-}$, there exist unique  Lagrange multipliers $\widehat{\lambda}_{0},\ldots,  \widehat{\lambda}_{d}$ such that for each $k \in A$ and $k' \in A^c$,
\begin{align*}
\begin{split}
& \Lambda(\beta^A) = 0, \quad \widehat{\lambda}_0 > 0, \\
&\partial_k \Lambda(\beta^A) = (u + \widehat{\lambda}_{k})/\widehat{\lambda}_0, \qquad \beta^A_{k} \geq 0, \quad \widehat{\lambda}_k \geq 0, \quad \widehat{\lambda}_k \beta^A_{k} = 0\\
& \partial_{k'} \Lambda(\beta^A) = -(\ell + \widehat{\lambda}_{k'})/\widehat{\lambda}_0, \quad \beta^A_{k'} \leq 0, \quad \widehat{\lambda}_{k'} \geq 0, \quad
\widehat{\lambda}_{k'} \beta^A_{k'} = 0,
\end{split}
\end{align*}
where $\partial_k$ denotes the partial derivative. Further, any $\theta \in \bR^d$ that solves the above problem is the maximizer of the first optimization problem in Lemma \ref{lemma:opt_probs}.
\end{lemma}
\begin{proof}
By assumption, $\Exp[X_1] \neq \mathbf{0}$.  Since $\operatorname{dom}(X)$ is open by condition \eqref{assumption_on_X} and $\Lambda(\mathbf{0}) = 0$, there exists $\overline{\theta}$ such that $\Lambda(\overline{\theta}) < 0$.

Then for the first optimization problem in Lemma \ref{lemma:opt_probs},  Assumption 6.4.3 in \cite{bertsekas2003convex} holds and then by Lemma 6.4.4 in \cite{bertsekas2003convex}, the KKT conditions are necessary and sufficient.
\end{proof}

\subsection{Proofs regarding independent coordinates}\label{app:independent_streams}
For the independent coordinates case, the KKT conditions in Lemma \ref{app:KKT} are simplified as follows.
Let $A \in \mathcal{P}_{d}^{-}$.  For each $k \in A$ and $k' \in A^c$,
\begin{align}\label{def:KKT_independent}
\begin{split}
& \sum_{k \in [d]} \Lambda^{k}(\beta^A_k) = 0, \quad \widehat{\lambda}_0 > 0,\\ 
&(\Lambda^{k})'(\beta^A_k) = u/\widehat{\lambda}_0, \qquad \beta^A_{k} > 0, \\
& (\Lambda^{k'})'(\beta^A_{k'}) = -(\ell + \widehat{\lambda}_{k'})/\widehat{\lambda}_0, \quad \beta^A_{k'} \leq 0, \quad \widehat{\lambda}_{k'} \geq 0, \quad
\widehat{\lambda}_{k'} \beta^A_{k'} = 0,
\end{split}
\end{align}

\begin{proof}[Proof of Lemma \ref{lemma:independent_one_change}]
Fix $k \in [d]$. $\beta^{\{k\}} = \gamma^k$ if and only if $\gamma^{k}$ satisfies the KKT conditions in \eqref{def:KKT_independent} for the independent coordinates case (some are satisfied by definition): for each $k' \neq k$,
\begin{align*}
&(\Lambda^{k})'(z_k) = u/\widehat{\lambda}_0, \quad
(\Lambda^{k'})'(0) = -(\ell + \widehat{\lambda}_{k'})/\widehat{\lambda}_0.
\end{align*}
Recall the definition of $\kappa_{k}^{(i)}, i = 1,2$ prior to Lemma \ref{lemma:independent_one_change}. We have that the above condition holds if and only if  for each $k' \neq k$,
$\kappa^{(0)}_{k'} \geq (\ell/u) \kappa_{k}^{(1)}
$, which completes the proof of the first claim.

For the second claim, by the first part, $r_{\{k^*\}} = u z_{k^*}$ and, in particular, $u z_{k^*} \geq \min_{k \in [d]} r_{\{k\}}$. Note that for the optimization problem in \eqref{def:gamma_k_kp}, one feasible solution is the following:
$$
\theta_{m} = z_{m}, \quad \theta_{m'}= z_{m'}, \quad
\theta_{k} = 0, \text{ for any } k \not\in\{m,m'\},
$$
which implies that for any $k \neq k'$, $s_{k,k'} \geq u(z_{k}+z_{k'})$. Thus, the two conditions in Lemma \ref{lemma:independent_one_change} imply conditions $(H1)$  and $(H2)$ in Theorem \ref{thm:sufficient}, respectively. 
\end{proof}


\subsection{Proofs regarding the homogeneous case}\label{app:homo_streams}
If the coordinates of $X_1$ are independent and identically distributed, we have $\Lambda^k(\cdot) = \Lambda^{1}(\cdot)$, $z_k = z_1$ and $\kappa_{k}^{(i)} = \kappa_{1}^{(i)}$ for $k \in [d]$ and $i \in \{1,2\}$. Further, the KKT conditions can be simplified due to homogeneity. Specifically, let $A \in \mathcal{P}_{d}^{-}$. For each $1 \leq a \leq d-1$, there exists a pair $v_a^{+}$ and $v_a^{-}$ such that if $|A| = a$, then
$$
\beta^{A}_{k} = v_a^{+}\; \text{ for } \;k \in A,\quad
\beta^{A}_{k'} = v_a^{-}\; \text{ for } \;k' \in A^c,
$$
and further for a unique multiplier $\hat{\lambda}_{a}^{-}$,
\begin{align}\label{def:homo_KKT}
\begin{split}
       &a \Lambda^{1}(v_a^{+}) + (d-a) \Lambda^{1}(v_a^{-}) = 0,\quad v_a^{+} > 0, \quad v_{a}^{-} \leq 0, \\
     & - \frac{1}{\ell} (\Lambda^1)'(v_a^{-}) = \frac{1}{u} (\Lambda^1)'(v_a^{+}) + \lambda_a^{-}, \quad 
   \lambda_a^{-} v_a^{-} = 0,
   \end{split}
\end{align}

Further, for $a = d$, we define $v_d^{+} = z_{1}$ and $v_d^{-} = -\infty$. Clearly, due to homogeneity and the KKT conditions in \eqref{def:KKT_independent}, 
\begin{align*}
    \beta_{k}^{[d]} = z_1 = v_d^{+}\;\;\text{ for each } k \in [d].
\end{align*}

The following lemma is crucial for proving Theorem \ref{thm:homo_case}.

\begin{lemma}\label{lemma:homo_ordering}
Assume the setup in Theorem \ref{thm:homo_case}. Then $v_1^{+} \geq (d-1)(-v_1^{-})$. Further,
$$
0 < v_1^{+} \leq v_2^{+} \leq \ldots \leq v_d^{+} = z_1,\quad
-v_1^{-} \leq -v_2^{-} \leq \ldots \leq v_d^{-} = \infty.
$$
\end{lemma}
\begin{proof}
Due to \eqref{def:homo_KKT} with $a=1$ and by the convexity of $\Lambda^1$, we have
$$
\Lambda^{1}\left(\frac{1}{d} v_1^{+} + \frac{d-1}{d} v_1^{-} \right) \leq 0.
$$
Since $\Lambda^{1}(\cdot)$ is non-positive on $[0,z_1]$ and non-negative outside this inteval, we have
$\frac{1}{d} v_1^{+} + \frac{d-1}{d} v_1^{-} \geq 0$, which proves the first claim.

Now we show the second claim. Note that $\Lambda^1(\cdot)$ is differentiable on its $\operatorname{dom}(\Lambda^1)$, which is assumed to be open, and its derivative $(\Lambda^1)'(\cdot)$ is strictly increasing on its $\operatorname{dom}(\Lambda^1)$.  Further, $\Lambda^1(0) = \Lambda^1(z_1) = 0$, $\Lambda^{1}(t) < 0$ for $0 < t < 1$, and $\Lambda^1(t) < 0$ for $t \in \operatorname{dom}(\Lambda^1)$ and $t < 0$ or $t > z_1$. Finally, denote by  $\iota \in (0,z)$ is the unique root of $\Lambda^1$. Then on the $\operatorname{dom}(\Lambda^1)$, $\Lambda^1(\cdot)$ is strictly decreasing to the left of $\iota$ and strictly increasing on to the right of $\iota$.   We consider two cases.

Case 1: $\ell^{-1} \kappa_1^{(0)} \geq u^{-1} \kappa_1^{(1)}$. Then for each $1 \leq a \leq d-1$, the solution to \eqref{def:homo_KKT} is
$$
v_a^{+} = z_1,\quad v_a^{-} = 0 \quad (\text{ with }
\lambda_{a}^{-} = \ell^{-1} \kappa_1^{(0)} - u^{-1} \kappa_1^{(1)}
).
$$
Thus, the conclusion holds.

Case 2: $\ell^{-1} \kappa_1^{(0)} < u^{-1} \kappa_1^{(1)}$.  
In this case, for each $1 \leq a \leq d-1$, the solution to \eqref{def:homo_KKT} must satisfy
\begin{align*}
\begin{split}
       &a \Lambda^{1}(v_a^{+}) + (d-a) \Lambda^{1}(v_a^{-}) = 0,\quad \iota < v_a^{+} < z_{1}, \quad v_{a}^{-} < 0, \\
     & - \frac{1}{\ell} (\Lambda^1)'(v_a^{-}) = \frac{1}{u} (\Lambda^1)'(v_a^{+}).
\end{split}
\end{align*}
In particular, $v_{d-1}^{+} \leq v_{d}^{+}$ and 
$v_{d-1}^{-} \leq v_{d}^{-}$. 

Now, assume the contrary that
$v_a^{+} > v_{a+1}^{+}$ for some $1 \leq a \leq d-2$. Due 
to the strict monotonicity of $(\Lambda^1)'$, due to the constraint on the second line above, we must have $(-v_a^{-}) > (-v_{a+1}^{-})$. Since on the $\operatorname{dom}(\Lambda^1)$, $\Lambda(\cdot)$ is strictly decreasing to the left of $\iota$ and strictly increasing on to the right of $\iota$, we have
\begin{align*}
    0 = &(a+1) \Lambda^1(v_{a+1}^{+}) + (d-a-1)\Lambda^1(v_{a+1}^{-})\\
    < &   (a+1) \Lambda^1(v_{a}^{+}) + (d-a-1)\Lambda^1(v_{a}^{-}) \\
    = & \Lambda^1(v_{a}^{+}) - \Lambda^1(v_{a}^{-}).
\end{align*}
However,  $\Lambda^1(v_{a}^{+}) > 0$ and $\Lambda^1(v_{a}^{-}) < 0$, which leads a contradiction. The proof is complete.
\end{proof}

\begin{proof}[Proof of Theorem \ref{thm:homo_case}]
Recall the discussion at the beginning of this subsection. Due to homogeneity, $r_A$ only depends on the size of $A \in \mathcal{P}_{d}^{-}$, and if $|A| = a$ for $a = 1,\ldots,d$, then
\begin{align*}
    r_A = u a v_a^{+} + \ell (d-a)(-v_a^{-}).
\end{align*}

Fix $A \in \mathcal{P}_{d}^{-}$ with $|A| = a$ for $a = 2,\ldots,d$. By the monotonicity in Lemma \ref{lemma:homo_ordering}, 
\begin{align}\label{aux_rA_minus_r1}
\begin{split}
        r_A - r_{\{1\}} \geq & u a v_1^{+} + \ell (d-a)(-v_1^{-}) - \left(u v_1^{+} + \ell (d-1) (-v_1^{-}) \right)  \\
    =& (a-1) \left( u v_1^{+} - \ell (-v_1^{-}) \right).
    \end{split}
\end{align}
Then due to first claim Lemma \ref{lemma:homo_ordering}, 
\begin{align*}
    r_A - r_{\{1\}} \geq (a-1)\left( u - \ell/(d-1)\right) v_1^{+}.
\end{align*}
By the assumption that $d \geq 1+\ell/u$,  $r_A - r_{\{1\}} \geq 0$ for any $A \in \mathcal{P}_d^{-}$. This shows that 
$r_* = r_{\{1\}}$.

Now we show the second claim. By Theorem \ref{thm:main_strategy}, it suffices to show that for any $A \subset [d]$ with $|A| \geq 2$ and $k \in A$, we have
\begin{align*}
v_A(\beta^{\{k\}}) 
\geq 2  r_{\{1\}}.    
\end{align*}

Fix $A \subset [d]$ with $a  :=  |A| \geq 2$ and $k \in A$. Note that by definition,
$$
\beta^{\{k\}}_{k} = v_1^{+},\quad
\beta^{\{k\}}_{k'} = v_1^{-}\; \text{ for } \;k' \neq k,
$$
and
$$
\beta^{A}_{m} = v_a^{+}\; \text{ for } \;m \in A,\quad
\beta^{A}_{m'} = v_a^{-}\; \text{ for } \;m' \in A^c.
$$
Define the following $\theta \in \bR^d$:
\begin{align*}
    &\theta_{k}^{*} = v_a^{+} + v_1^{+}, \quad
    \theta_{m}^{*} = v_a^{+} + v_1^{-}\; \text{ for } \;m \in A \setminus \{k\}\\
    &\theta_{m'}^{*} = v_a^{-} + v_1^{-}\; \text{ for } \;m' \in A^c.
\end{align*}
By Lemma \ref{lemma:homo_ordering}, $\theta_k^{*} \geq 0$, $\theta_{m'}^{*} \leq 0$ for $m' \in A^c$, and for $m \in A \setminus \{k\}$,
$$
\theta_{m}^{*} = v_a^{+} + v_1^{-} \geq v_1^{+} - (-v_1^{-}) \geq (1-1/(d-1)) v_1^{+} \geq 0,
$$
where the last inequality is because $d \geq 2$. Further,
\begin{align*}
    \sum_{i \in [d]} \Lambda^{1}(\theta_i^{*} - \beta^{\{k\}}_i) = a \Lambda^{1}(v_a^{+}) + (d-a) \Lambda^{1}(v_a^{-}) = 0,
\end{align*}
where the last equality is due to the KKT condition \eqref{def:homo_KKT}. This shows that $\theta^{*}$ is a feasible solution to the second optimization problem in Lemma \ref{lemma:opt_probs} with $\gamma = \beta^{\{k\}}$. In particular,
\begin{align*}
 v_A(\beta^{\{k\}}) 
 &\geq u \sum_{m \in A} \theta^{*}_m - \ell \sum_{m' \in A^c} \theta^{*}_{m'} \\
 & = u(v_a^{+} + v_1^{+}) + u (a-1) (v_a^+ + v_1^{-}) +\ell (d-a)(-v_a^{-} + (-v_1^{-})).
\end{align*}
Since $r_A = u a v_a^{+} + \ell (d-a)(-v_a^{-})$, we have
\begin{align*}
 v_A(\beta^{\{k\}}) 
 &\geq r_A + u v_1^{+}   + (\ell(d-a)-u(a-1))(-v_1^{-}).
\end{align*}
Due to the calculation in \eqref{aux_rA_minus_r1}, we have
\begin{align*}
   &v_A(\beta^{\{k\}}) 
  - 2r_{\{1\}} 
     \geq (a-1)(u-\ell/(d-1)) v_1^{+} - (a-1)(\ell+u)(-v_1^{-}).
\end{align*}
By Lemma \ref{lemma:homo_ordering}, we have $v_1^{+} \geq (d-1)(-v_1^{-})$ and thus
\begin{align*}
     &
     v_A(\beta^{\{k\}}) - 2r_{\{1\}} 
     \geq (a-1) v_1^{+} (u - (2\ell+u)/(d-1)).
\end{align*}
If $d \geq 2+2\ell/u$, then the lower bound above is non-negative. The proof is complete.
\end{proof}

\section{Proofs regarding the gap rule}\label{app:gap_rule}

\begin{proof}[Proof of Lemma \ref{lemma:gap_opt_probs}]
Fix $ A \in \mathcal{G}_m^{-}$. Due to \eqref{def:r_j_alt},
$$
r_A =  \sup_{\theta \in \mathcal{L}_0(\Lambda)} \inf_{x \in \overline{W^A}}  \theta \cdot x.
$$
By the definition of $W^A$ in \eqref{def:gap_W_A}, 
\begin{align*}
   \inf_{x \in \overline{W^A}}  \theta \cdot x =   \min_{x \in \bR^d, a \in \bR} \theta \cdot x, \;\;&\text{ subject to} \\
   & x_{k} \geq a,\; x_{k'} \leq a-1,\; \text{ for } \;k  \in A, k'\in A^c.
\end{align*}
By elementary computation or duality for linear programming, if $\sum_{i \in [d]}\theta_i = 0$, $\theta_k \geq 0$ and $\theta_{k'} \leq 0$ for $k \in A, k' \in A^c$, then
    \begin{align*}
   \inf_{x \in \overline{W^A}}  \theta \cdot x =   \sum_{k \in A} \theta_k;\;\; \text{ Otherwise, } \inf_{x \in \overline{W^A}}  \theta \cdot x = -\infty.
\end{align*}
The proof for the first claim is complete. For the second claim, due to \eqref{def:lower_bound_V_gamma} and since $\Lambda(\gamma) \leq  0$, we have
\begin{align*}
v_A(\gamma)  
\geq  
\sup_{\theta \in \bR^d: \Lambda(\gamma) + \Lambda(\theta-\gamma) \leq 0}\; \inf_{x \in \overline{W^A}} \theta \cdot x \geq 
\sup_{\theta \in \bR^d:  \Lambda(\theta-\gamma) \leq 0} \;\inf_{x \in \overline{W^A}} \theta \cdot x.
\end{align*}
The rest of the arguments are the same as above. The proof is complete.
\end{proof}

\subsection{KKT conditions: Gap rule}\label{app:KKT_gap}
Recall the first optimization problem in  Lemma \ref{lemma:gap_opt_probs}.

\begin{lemma}\label{lemma:app_Gap_KKT}
Let $A \in \mathcal{G}_m^{-}$. Then any $\beta^A \in \bR^d$ that satisfies the following KKT conditions is the maximizer of the first optimization problem in Lemma \ref{lemma:gap_opt_probs}:
for  Lagrange multipliers $\widehat{\lambda}_{0},\ldots,  \widehat{\lambda}_{d}$ and $\widehat{\lambda}_{*}$ such that for each $k \in A$ and $k' \in A^c$,
\begin{align*}
\begin{split}
& \Lambda(\beta^A) = 0, \quad \widehat{\lambda}_0 > 0, \\
& \partial_k \Lambda(\beta^A) =(1+\widehat{\lambda}_k)/\widehat{\lambda}_0 + \widehat{\lambda}_{*}, \quad  \beta^A_{k} \geq 0, \quad \widehat{\lambda}_k \geq 0, \quad \widehat{\lambda}_k \beta^A_{k} = 0,\\
&\partial_{k'}\Lambda(\beta^A) = -\widehat{\lambda}_{k'}/\widehat{\lambda}_0 + \widehat{\lambda}_{*},\qquad \beta^A_{k'} \leq 0, \quad \widehat{\lambda}_{k'} \geq 0, \quad
\widehat{\lambda}_{k'} \beta^A_{k'} = 0,
\\
&\sum_{i \in [d]} \beta^{A}_i = 0,
\end{split}
\end{align*}
where $\partial_k$ denotes the partial derivative. 
\end{lemma}
\begin{proof}
The arguments are similar to Lemma \ref{lemma:app_KKT}, and thus omitted.
\end{proof}

\subsection{Proofs regarding Example \ref{example:corr_normal_Gap} for the Gap rule}\label{app:gap_example_corrnormal_proof}

\begin{proof}[Proof of the statement in Example \ref{example:corr_normal_Gap}]
Let $\ell \in \Mstar, \; \ell' \not\in \Mstar$ and $A = (\Mstar \setminus \{\ell\}) \cup \{\ell'\}$.  The cumulant generating function and its gradient are as follows: for $\theta \in \bR^d$
\begin{align*}
    \Lambda(\theta) = \theta' \Exp[X_1]+ \frac{1}{2} \theta' \text{Cov}(X_1) \theta, \quad
    \nabla  \Lambda(\theta) = \Exp[X_1] + \text{Cov}(X_1) \theta.
\end{align*}
Elementary calculation shows that
$$
\tilde{\gamma}^{\ell,\ell'}_{\ell} = - \frac{\mu_{+} - \mu_{-}}{\sigma^2(1-\rho)}, 
\quad \tilde{\gamma}^{\ell,\ell'}_{\ell'} =  \frac{\mu_{+} - \mu_{-}}{\sigma^2(1-\rho)}, \quad 
\tilde{\gamma}^{\ell,\ell'}_{k} = 0\; \text{ for } \;k \not \in \{\ell,\ell'\}.
$$
Further, for each $k \in A$ and $k' \in A^c$,
\begin{align*}
   \partial_{k}  \Lambda(\tilde{\gamma}^{\ell,\ell'}) =  \mu_{+},
\quad
   \partial_{k'}  \Lambda(\tilde{\gamma}^{\ell,\ell'}) =  \mu_{-}.
\end{align*}
Finally, define
$$
\widehat{\lambda}_0 = \frac{1}{\mu_{+} - \mu_{-}}, \quad
\widehat{\lambda}_k = 0\; \text{ for } \;k \in [d], \quad
\widehat{\lambda}_{*} = \mu_{-}.
$$
We have that $\tilde{\gamma}^{\ell,\ell'}$ and  $\widehat{\lambda}_{0},\ldots,  \widehat{\lambda}_{d}$ and $\widehat{\lambda}_{*}$ satisfy the KKT condition in Lemma \ref{app:KKT_gap}. As a result, $\tilde{\gamma}^{\ell,\ell'} = \beta^{A}$ and 
$\tilde{z}_{\ell,\ell'} = r_A$, which completes the proof.
\end{proof}

\subsection{Additional simulation results under the setup of Example \ref{example:corr_normal_Gap}} \label{app:gap_example_corrnormal_more_sim}

Recall the Example \ref{example:corr_normal_Gap}. In this subsection, We let $\mu_{+} = -\mu_{-} = 1/2$, $\sigma^2 = 1$, $\rho \in \{0.1,0.4\}$, $d = 100$ and $m \in \{10,30\}$. In Figure \ref{fig:Gap_d100m1030}, we use $N = 10^6$ repetitions and plot the estimated relative error against $-\log_{10}$ of the estimated probability. We observe that for probabilities as small as $10^{-10}$, the estimated relative error is below $3.5\%$ for $m = 30$ and $2.5\%$ for $m = 10$. 

\begin{figure}[!t]
    \centering
    \includegraphics[width = \textwidth]{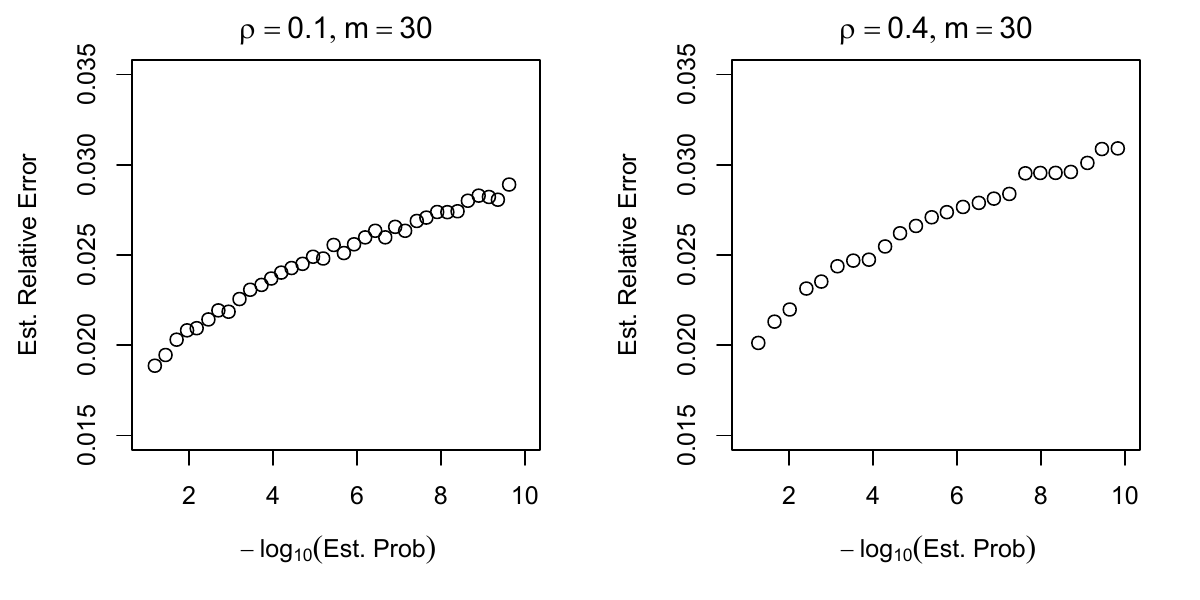} \\
    \includegraphics[width = \textwidth]{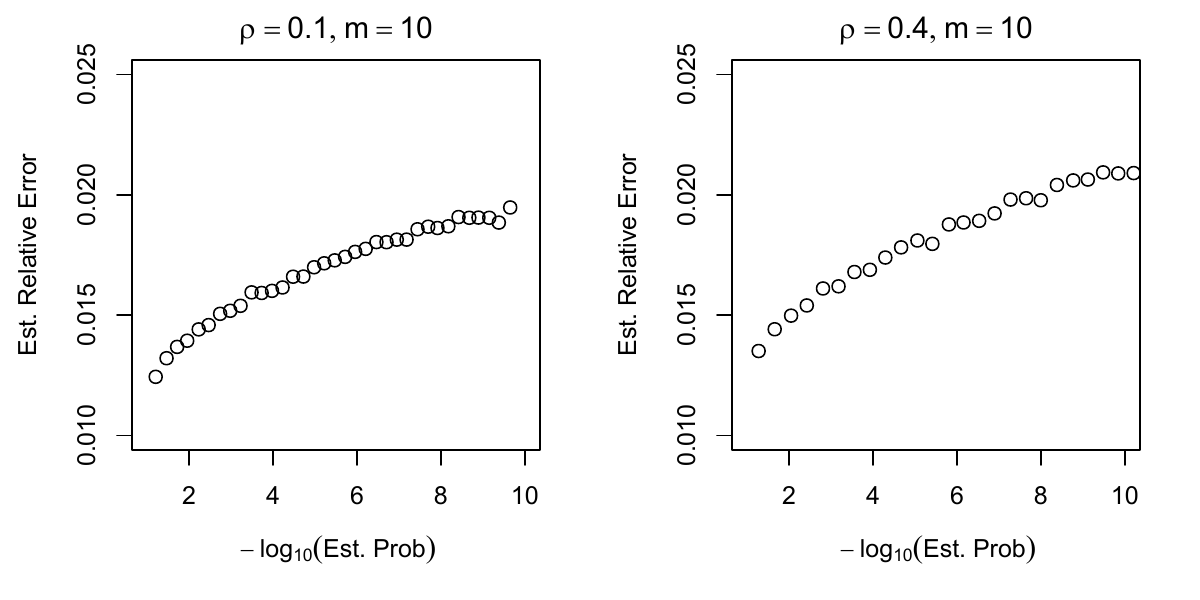}  
    \caption{Gap rule with $d=100$, $\rho \in \{0.1,0.4\}$ and $m \in \{10,30\}$ under the setup of Example \ref{example:corr_normal_Gap}. The x-axis is $-\log_{10}$ of the estimated probability, and the y-axis is the estimated relative error. The results are based on $10^6$ repetitions.}
    \label{fig:Gap_d100m1030}
\end{figure}

\section{Proofs regarding the sum-intersection rule}\label{app:sum_intersection}
Recall that for any $1 \leq L \leq d$, $\mathcal{G}_L$ is all subsets of $[d]$ with a size $L$ and 
$\mathcal{S}_L$ is all subsets of $[d]$ with a size \textit{at least} $L$.

\begin{proof}[Proof of Lemma \ref{lemma:sum_intersection_opt_probs}]
Fix some $ A \in \mathcal{S}_L$. Due to \eqref{def:r_j_alt},
$$
r_A =  \sup_{\theta \in \mathcal{L}_0(\Lambda)} \inf_{x \in \overline{W^A}}  \theta \cdot x.
$$
Let $\theta \in \bR^d$ and recall the definition of $W^A$ in \eqref{def:SI_W_A}. If there exists some $k \in A$ such that $\theta_k < 0$,  we have $\inf_{x \in \overline{W^A}}  \theta \cdot x =  -\infty$ since we can send $x_k \to +\infty$. Similarly, if there exists some $k' \in A^c$ such that $\theta_{k'} > 0$,  we have $\inf_{x \in \overline{W^A}}  \theta \cdot x =  -\infty$ since we can send $x_{k'} \to -\infty$.

Now, assume that $\theta_k \geq 0$ and $\theta_{k'} \leq 0$ for $k \in A$ and $k' \not \in A$. By the definition of $W^A$ and a change of variables,
\begin{align}\label{aux:sum-inter-opt_mid}
    \inf_{x \in \overline{W^A}}  \theta \cdot x 
=    \inf\left\{\sum_{k \in [d]} |\theta_k| x_k \;:\;  x_k \geq 0\; \text{ for } \;k \in [d],  \quad 
 \sum_{k \in B} x_k \geq 1\; \text{ for } \;B \in \mathcal{G}_L \right\}.
\end{align}
By Lemma \ref{lemma:ordering_opt1} ahead, we have
\begin{align*}
    \inf_{x \in \overline{W^A}}  \theta \cdot x 
=  \min_{1 \leq \ell \leq L} \frac{1}{\ell} \sum_{i = L - \ell + 1}^{d} |\theta|_{(i)},
\end{align*}
which completes the proof for the first claim. Since $\Lambda(\gamma) \leq 0$, the proof for the second claim is due to \eqref{def:lower_bound_V_gamma} and the result in \eqref{aux:sum-inter-opt_mid}.
\end{proof}

\begin{proof}[Proof of the statement in Remark \ref{remark:sum-intersection_convex}]
We note that $\min_{1 \leq \ell \leq L} \frac{1}{\ell} \sum_{i = L - \ell + 1}^{d} |\theta|_{(i)}$ is equal to the optimal value of the linear programming in \eqref{aux:sum-inter-opt_mid}.

Assume $\theta_k \geq 0$ and $\theta_{k'} \leq 0$ for $k \in A$ and $k' \not \in A$. By duality theory for linear programming \citep[Proposition 6.4.2]{bertsekas2003convex} (see also Exercise 6.5 in \cite{bertsekas2003convex}), 
\begin{align*}
&\inf\left\{\sum_{k \in [d]} |\theta_k| v_k \;:\;  v_k \geq 0\; \text{ for } \;k \in [d],  \quad 
 \sum_{k \in B} v_k \geq 1\; \text{ for } \;B \in \mathcal{G}_L \right\} = \\
&\max_{\lambda_B \in \bR:\; B \in \mathcal{G}_L} \sum_{B \in \mathcal{G}_L} \lambda_B, \;\; \text{ s.t. }\;\;  \lambda_B \geq 0 \;\text{ for }\; B \in \mathcal{G}_L,\;\; \sum_{B \in \mathcal{G}_L \text{ and } k \in B} \lambda_B \leq |\theta_k| \; \text{ for }\; k \in [d].
\end{align*}
As a result, the optimal value of the first optimization in Lemma \ref{lemma:sum_intersection_opt_probs} is equal to
\begin{align}\label{opt:sum_inter_equiv}
\begin{split}
       & \max_{\theta \in \bR^d,\;\; \{\lambda_B \in \bR:\; B \in \mathcal{G}_L\}} \sum_{B \in \mathcal{G}_L} \lambda_B, 
  \quad \textup{ subject to }  \\
& \Lambda(\theta) \leq 0, \quad \theta_k \geq 0 \;\text{ for }\; k \in A, \quad \theta_{k'} \leq 0 \;\text{ for }\; k' \in A^c,  \\ 
&   \lambda_B \geq 0 \;\text{ for }\; B \in \mathcal{G}_L,\quad \sum_{B \in \mathcal{G}_L \text{ and } k \in B} \lambda_B \leq |\theta_k| \; \text{ for }\; k \in [d].
\end{split}
\end{align}
Similarly, the optimal value of  the second optimization in Lemma \ref{lemma:sum_intersection_opt_probs} is equal to 
\begin{align}\label{opt:sum_inter_equiv2}
\begin{split}
       & \max_{\theta \in \bR^d,\;\; \{\lambda_B \in \bR:\; B \in \mathcal{G}_L\}} \sum_{B \in \mathcal{G}_L} \lambda_B, 
  \quad \textup{ subject to }  \\
& \Lambda(\theta - \gamma) \leq 0, \quad \theta_k \geq 0 \;\text{ for }\; k \in A, \quad \theta_{k'} \leq 0 \;\text{ for }\; k' \in A^c,  \\ 
&   \lambda_B \geq 0 \;\text{ for }\; B \in \mathcal{G}_L,\quad \sum_{B \in \mathcal{G}_L \text{ and } k \in B} \lambda_B \leq |\theta_k| \; \text{ for }\; k \in [d].
\end{split}
\end{align}
\end{proof}

\subsection{Proof of Theorem \ref{thm:sum_intersection_rule}} \label{app:proof_theorem_SI}
\begin{proof}[Proof of Theorem \ref{thm:sum_intersection_rule}]
Fix any $A \subset [d]$ such that $|A| \geq L+1$. Let $\tilde{A} \subset A$ such that $|\tilde{A}| = L$ and $\tilde{k} \in A \setminus \tilde{A}$. Recall the definition of $z_{\tilde{A}}$ and $\gamma^{\tilde{A}}$ in \eqref{opt:sum-inter_L}, and $s_{\tilde{A} \cup \{\tilde{k}\}}$ and $\tilde{\gamma}^{\tilde{A} \cup \{\tilde{k}\}}$ in \eqref{opt:sum-inter_L_1}.

Note that
$$
\gamma^{\tilde{A}}_k \geq 0 \; \text{ for } k \in \tilde{A}, \quad \text{ and } \quad
\gamma^{\tilde{A}}_k = 0 \; \text{ for } k \not \in \tilde{A},
$$
which implies
$$
\gamma^{\tilde{A}}_k \geq 0 \; \text{ for } k \in A,  \quad \text{ and } \quad
\gamma^{\tilde{A}}_{k'} \leq 0 \; \text{ for } k' \not \in A. 
$$
Applying the result in \eqref{aux:sum-inter-opt_mid} to $\theta = \gamma^{\tilde{A}}$, we have
\begin{align*}
 \inf_{x \in \overline{W^{{A}}}}  \gamma^{\tilde{A}} \cdot x = \inf\left\{\sum_{k \in [d]} |\gamma^{\tilde{A}}_k| x_k \;:\;  x_k \geq 0\; \text{ for } \;k \in [d],  \quad 
 \sum_{k \in B} x_k \geq 1\; \text{ for } \;B \in \mathcal{G}_L \right\}.
\end{align*}
Then due to Lemma \ref{lemma:ordering_opt1}, we have
\begin{align*}
\inf_{x \in \overline{W^{{A}}}}  \gamma^{\tilde{A}} \cdot x  =      \min_{1 \leq \ell \leq L} \frac{1}{\ell} \sum_{i = L - \ell + 1}^{d} |\gamma^{\tilde{A}}|_{(i)} = z_{\tilde{A}},
\end{align*}
where the final equality is because $\gamma^{\tilde{A}}$ has at most $L$ non-zero coordiantes.

By the same argument, applying the result in \eqref{aux:sum-inter-opt_mid} to $\theta = \tilde{\gamma}^{\tilde{A} \cup \{\tilde{k}\}}$ and 
$\theta = \gamma^{\tilde{A}} + \tilde{\gamma}^{\tilde{A} \cup \{\tilde{k}\}}$, and due to  Lemma \ref{lemma:ordering_opt1}, we have
\begin{align*}
&\inf_{x \in \overline{W^{{A}}}}  \tilde{\gamma}^{\tilde{A} \cup \{\tilde{k}\}} \cdot x  =      \min_{1 \leq \ell \leq L} \frac{1}{\ell} \sum_{i = L - \ell + 1}^{d} \left|\tilde{\gamma}^{\tilde{A} \cup \{\tilde{k}\}} \right|_{(i)} = s_{\tilde{A} \cup \{\tilde{k}\}}, \\
&\inf_{x \in \overline{W^{{A}}}}  \left( \gamma^{\tilde{A}}+\tilde{\gamma}^{\tilde{A} \cup \{\tilde{k}\}}\right) \cdot x  =      \min_{1 \leq \ell \leq L} \frac{1}{\ell} \sum_{i = L - \ell + 1}^{d} \left| \gamma^{\tilde{A}}+\tilde{\gamma}^{\tilde{A} \cup \{\tilde{k}\}}\right|_{(i)}.
\end{align*}
As a result, we immediately have
$$
  \min_{1 \leq \ell \leq L} \frac{1}{\ell} \sum_{i = L - \ell + 1}^{d} \left| \gamma^{\tilde{A}}+\tilde{\gamma}^{\tilde{A} \cup \{\tilde{k}\}}\right|_{(i)} \geq z_{\tilde{A}} + s_{\tilde{A} \cup \{\tilde{k}\}}.
$$
Finally, for the second optimization problem in Lemma \ref{lemma:sum_intersection_opt_probs} with $\gamma = \gamma^{\tilde{A}}$, we have that $\gamma^{\tilde{A}} + \tilde{\gamma}^{\tilde{A} \cup \{\tilde{k}\}}$ is a feasible solution, and as a result,
\begin{align*}
    v_{A}(\gamma^{\tilde{A}}) \geq  \min_{1 \leq \ell \leq L} \frac{1}{\ell} \sum_{i = L - \ell + 1}^{d} \left| \gamma^{\tilde{A}}+\tilde{\gamma}^{\tilde{A} \cup \{\tilde{k}\}}\right|_{(i)} \geq z_{\tilde{A}} + s_{\tilde{A} \cup \{\tilde{k}\}}.
\end{align*}
The proof is then complete due to Theorem \ref{thm:main_strategy}.
\end{proof}

\subsection{Supporting lemmas: sum-intersection rule}

\begin{lemma} \label{lemma:ordering_opt1}
Let $2 \leq L < d$, and $\theta \in \bR^d$ such that $\theta_1 \geq \theta_2 \geq \ldots \geq \theta_d \geq 0$. Consider the following optimization problem:
\begin{align*}
    \min_{x \in \bR^d} \theta \cdot x, \quad &\text{ subject to }\; x_k \geq 0\;\; \text{ for } \;k \in [d], \quad
     \sum_{k \in B} x_k \geq 1 \;\; \text{ for } \;B \in \mathcal{G}_L.
\end{align*}
Denote by $\tilde{\theta}_L = \sum_{k=L}^{d} \theta_k$. Then the optimal value is given by
\begin{align*}
    \min\left\{
    \tilde{\theta}_L,\; \frac{1}{2}(\tilde{\theta}_{L} + {\theta}_{L-1}),\; \frac{1}{3}(\tilde{\theta}_L + {\theta}_{L-1} + \theta_{L-2}),\; \ldots,\; \frac{1}{L}\left(\tilde{\theta}_m + {\theta}_{L-1} +  \ldots + \theta_{1} \right)
    \right\}
\end{align*}
\end{lemma}
\begin{proof}
Let $x \in \bR^d$ be a feasible solution, and assume for some $1 \leq k \leq d-1$, $x_k \geq x_{k+1}$. Define $\tilde{x} \in \bR^d$ as follows:
$$
\tilde{x}_j = x_j\; \text{ for } \;j \not \in \{k,k+1\}, \quad
\tilde{x}_k = x_{k+1}, \quad \tilde{x}_{k+1} = x_k.
$$
Then $\tilde{x}$ is also a feasible solution and 
$$
\theta \cdot x - \theta \cdot \tilde{x} = (\theta_k - \theta_{k+1}) (x_k - x_{k+1}) \geq 0.
$$
That is, $\tilde{x}$ is at least as good as $x$. Then, by an inductive argument,
\begin{align*}
&\min_{x \in \bR^d} \theta \cdot x, \quad \text{ s.t. }\; x_k \geq 0\;\; \text{ for } \;k \in [d], \quad
     \sum_{k \in B} x_k \geq 1 \;\; \text{ for } \;B \in \mathcal{G}_L \\
=& \min_{x \in \bR^d} \theta \cdot x, \quad \text{ s.t. }\; 0 \leq x_1  \leq \ldots \leq x_{d} \;\; \text{ and } \;\; \sum_{k=1}^{L} x_i \geq 1.
\end{align*}
Further, since $\theta_k \geq 0$ for $k \in [d]$, for the latter optimization problem, we may assume $x_{d} = x_{d-1} = \ldots = x_{L+1} = x_{L}$ and $\sum_{k=1}^{L} x_i = 1$. Thus,
\begin{align*}
&\min_{x \in \bR^d} \theta \cdot x, \quad \text{ s.t. }\; x_k \geq 0\;\; \text{ for } \;k \in [d], \quad
     \sum_{k \in B} x_k \geq 1 \;\; \text{ for } \;B \in \mathcal{G}_L \\
=& \min_{x \in \bR^L} \theta_1 x_1 + \ldots + \theta_{L-1}x_{L-1} + \tilde{\theta}_L x_L, \;\; \text{ s.t. }\; 0 \leq x_1  \leq \ldots \leq x_{L} \;\; \text{ and } \;\; \sum_{k=1}^{L} x_i = 1.
\end{align*}
Then the proof is complete by Lemma \ref{lemma:ordering_opt2}.
\end{proof}

\begin{lemma} \label{lemma:ordering_opt2}
Let $L \geq 2$, and $\theta \in \bR^L$ such that $\theta_1 \geq \theta_2 \geq \ldots \geq \theta_{L-1} \geq 0$ and $\theta_{L} \geq 0$. Consider the following optimization problem:
\begin{align*}
    \min_{x \in \bR^L} \theta \cdot x, \quad &\text{ subject to } \; 0 \leq x_1  \leq \ldots \leq x_{L}, \quad
     \sum_{k =1}^{L} x_k = 1.
\end{align*}
Then the optimal value is given by
\begin{align*}
  v^*  :=   \min\left\{
    {\theta}_L,\; \frac{1}{2}({\theta}_L + {\theta}_{L-1}),\; \ldots,\; \frac{1}{L}\left({\theta}_L + {\theta}_{L-1} +  \ldots + \theta_{1} \right)
    \right\}.
\end{align*}
\end{lemma}
\begin{proof}
The KKT conditions are as follows. There exist $x \in \bR^L$, $\lambda \in \bR$ and $\lambda_k \geq 0$ for $k \in [L]$ such that
\begin{align*}
&\lambda_k = \sum_{\ell=k}^{L} (\theta_\ell - \lambda) \;\; \text{ for } \;k \in [L],\quad  0 \leq x_1  \leq \ldots \leq x_{L}, \quad \sum_{k=1}^{L} x_i = 1,\\
& \lambda_1 x_1 = 0, \quad
\lambda_k(x_{k-1} - x_{k}) = 0 \;\text{ for } 2 \leq k \leq L.
\end{align*}
Denote by $k^* \in [L]$ an index that attains $v^*$, that is,
$$
v^* = \frac{1}{L-k^*+1} {\sum_{\ell=k^*}^{L} \theta_\ell}.
$$
It is elementary to verify that the following triplet $x \in \bR^L$, $\lambda \in \bR$ and $\{\lambda_{k}, k \in [L]\}$ satisfies the KKT conditions:
\begin{align*}
&x_1 = \ldots = x_{k^*-1} = 0, \quad x_{k^*} = \ldots = x_L = 1/(L-k^*+1), \quad \lambda = v^*,\\
&\lambda_k = \sum_{\ell=k}^{L} (\theta_\ell - v^*)\; \text{ for } \; k\in [L].
\end{align*}
In particular, by the definition of $k^*$ and $v^*$, $\lambda_k \geq 0$ for $k \in [L]$ and $\lambda_{k^*} = 0 $ As a result, the optimal value is $v^*$.
\end{proof}

\subsection{Additional simulation results under the setup of Example \ref{example:SI_corr}} \label{app:SI_corrnormal_more_sim}

Recall Example \ref{example:SI_corr}, where we set $d=100, L=2$, and consider $\rho \in \{0,0.05,0.1,0.2\}$. In Figure \ref{fig:SI_d100L2}, we plot the estimated relative error against $-\log_{10}$ of the estimated probability with $N=10^6$ repetitions. We observe that for probabilities as small as $10^{-10}$, the estimated relative error is below $3\%$ for $\rho \in \{0.05,0.1,0.2\}$ and below $4.2\%$ for $\rho = 0$.

\begin{figure}[t!]
    \centering
        \includegraphics[width=\textwidth]{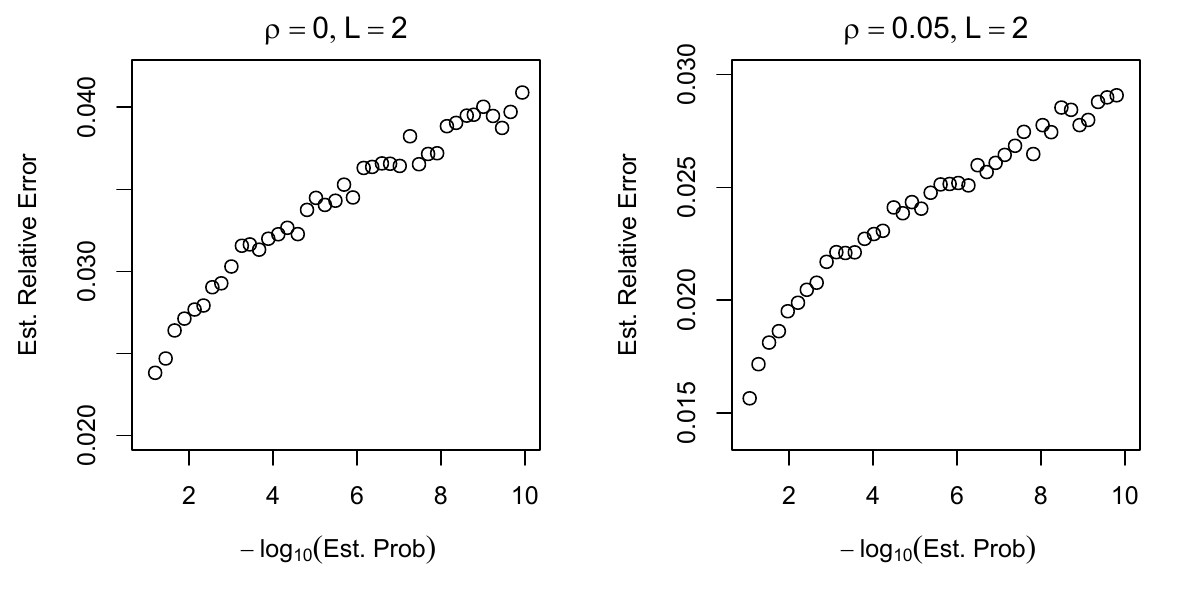}
\\
    \includegraphics[width=\textwidth]{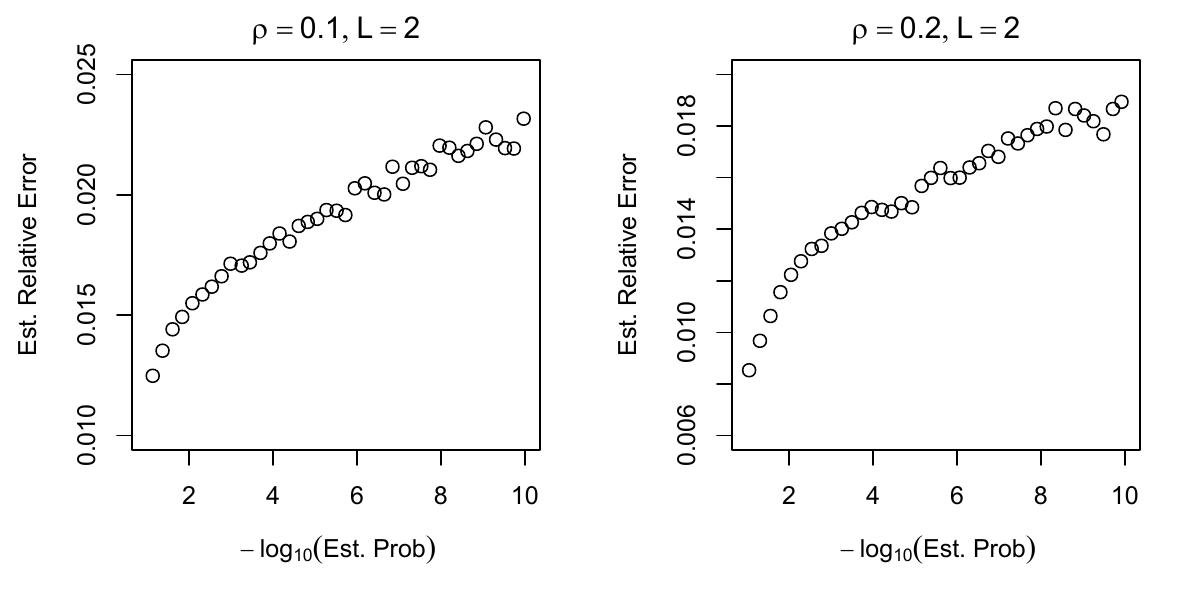}
    \caption{Sum-intersection rule with $d=100, L =2$ under the setup of Example \ref{example:SI_corr} for various $\rho$. The x-axis is $-\log_{10}$ of the estimated probability, and the y-axis is the estimated relative error. The results are based on $10^6$ repetitions.}
    \label{fig:SI_d100L2}
\end{figure}

\end{appendix}

\end{document}